\documentclass{article}

\usepackage[final]{neurips_2022}

\usepackage[utf8]{inputenc} 
\usepackage[T1]{fontenc}    
\usepackage{hyperref}       
\usepackage{url}            
\usepackage{booktabs}       
\usepackage{amsfonts}       
\usepackage{nicefrac}       
\usepackage{microtype}      
\usepackage{xcolor}         
\usepackage{amsmath, amsfonts, amssymb, amsthm,bm}

\input{content/command}

\title{Reproducibility in Optimization:\\ Theoretical Framework and Limits}

\author{Kwangjun Ahn\thanks{Part of this work was done when Kwangjun Ahn was an intern at Google Research.}  \\ MIT EECS \\ \texttt{kjahn@mit.edu} \And 
Prateek Jain \\ Google Research  \\ \texttt{prajain@google.com} \And 
Ziwei Ji \\ Google Research \\ \texttt{ziweiji@google.com} \And
Satyen Kale \\ Google Research\\ \texttt{satyenkale@google.com} \And
Praneeth Netrapalli \\ Google Research  \\ \texttt{pnetrapalli@google.com}  \And
Gil I. Shamir \\ Google Research, Brain Team\\ \texttt{gshamir@google.com}}

\begin{document}

\maketitle

\begin{abstract}
  We initiate a formal study of reproducibility in optimization. We define a quantitative measure of reproducibility of optimization procedures in the face of noisy or error-prone operations such as inexact or stochastic gradient computations or inexact initialization. We then analyze several convex optimization settings of interest such as smooth, non-smooth, and strongly-convex objective functions and establish tight bounds on the limits of reproducibility in each setting. Our analysis reveals a fundamental trade-off between computation and reproducibility: more computation is necessary (and sufficient) for better reproducibility.
\end{abstract}


\doparttoc 
\faketableofcontents 
 
\section{Introduction}

\label{sec:intro}
Machine learned models are increasingly entering wider ranges of domains in our lives, driving a constantly increasing number of important systems.  
Large scale systems can be trained in highly parallel and distributed training environments, with a large amount of randomness in training the models.  While some systems may tolerate such randomness leading to models that differ from one another every time a model retrains, for many applications, {\em reproducible} models are required, where slight
changes in training do not lead to drastic differences in the model learned. 

Beyond practical deployments of machine learned models, the reproducibility crisis in the machine learning academic world has also been well-documented: see \citep{pineau2021improving} and the references therein for an excellent discussion of the reasons for irreproducibility (insufficient exploration of hyperparameters and experimental setups, lack of sufficient documentation, inaccessible code, and different computational hardware) and for mitigation recommendations. Recent papers \citep{chen20,damour20,dusenberry20,snapp2021synthesizing,summers21,yu2021dropout} have also demonstrated that even when models are trained on identical datasets with identical optimization algorithms, architectures, and hyperparameters, they can produce significantly different predictions on the same example.
This type of irreproducibility may be caused by multiple factors \citep{damour20,fort2020deep,frankle2020linear,shallue18,snapp2021synthesizing,summers21}, such as non-convexity of the objective, random initialization, nondeterminism in training such as data shuffling, parallelism, random schedules, hardware used, and round off quantization errors.  Perhaps surprisingly, even if we control for the randomness by using the same ``seed" for model initialization, other factors such as numerical errors introduced due to nondeterminism of modern GPUs (see, e.g., \citep{zhuang2021randomness}) may still lead to significant differences. It was empirically shown (see, e.g., \citet{achille17}) that slight deviations early in training can lead to different optima, with substantial differences in resulting models. 
Thus we are forced to accept some {\em fundamental} level of irreproducibility that persists even after fixing all aspects of the training process under our control. 

The goal of this paper is to initiate a formal study of the fundamental limits of irreproducibility.
In particular, we focus on the most basic training process: {\em convex optimization}. At first glance, it might seem surprising that convex optimization procedures can exhibit irreproducibility since they're guaranteed to converge to an optimal solution.
However, in practice, the default convex optimization algorithms are iterative first-order methods; methods that only use a first-order oracle to provide an approximate gradient of the function at a given point, and converge to an approximately optimal solution.  
The first-order oracle is a source of irreproducibility. In stochastic gradient descent, it returns a random vector whose expectation is the true gradient.  The randomness in the stochastic gradients can lead to different outcomes of the optimization process. Similarly, there are numerical errors that can arise in the computation of the gradients due to inherent nondetermism in modern GPUs. Beyond the first-order oracle, irreproducibility may also arise in convex optimization procedures because the initial point is chosen randomly. 
Thus, we attempt to answer the following questions for convex optimization procedures operating with the above sources of 
irreproducibility: \vspace{-3pt}
\begin{list}{{\tiny $\blacksquare$}}{\leftmargin=1.5em}
		\setlength{\itemsep}{-1pt}
    \item What are the fundamental limits of reproducibility for any convex optimization procedure? 
    \item Can we design practical and efficient first-order methods  
    that achieve these 
    limits?
\end{list}  
We study these questions in a variety of settings; including general non-smooth convex functions, smooth convex functions, strongly-convex functions, finite-sum functions, and stochastic convex optimization, under the different sources of irreproducibility mentioned above.
To the best of our knowledge, no prior theoretical work considered such questions.

The primary contribution of this paper is conceptual: the development of a rigorous theoretical framework to study the fundamental limits of reproducibility in convex optimization. The concepts developed in this framework can be extended easily to other settings of interest such as non-convex optimization. The technical contribution of this paper is the development of lower bounds on the amount of reproducibility, and matching upper bounds via analysis of specific first-order algorithms in all the different settings of convex optimization described above. Detailed technical descriptions of the results appear below in \autoref{sec:summary}.

At a high level, our study provides the same message for all the different optimization settings we consider. On the lower bound side, we find that {\em any} first-order method would need to trade-off convergence rate (computational complexity) for more reproduciblity. On the upper bound side, we find that various forms of gradient descent, when run with {\em lower} step-size (and correspondingly, {\em more} iterations) already achieve the fundamental limits of reproducibility. One (somewhat surprising) consequence of our results is that advanced techniques like regularization, variance reduction, and acceleration do not improve reproducibility over standard gradient descent methods.  

We show, for example, that when optimizing a Lipschitz non-smooth convex function $f$ on a bounded domain using a first-order oracle that computes gradients with even {\em vanishingly small} error, two runs of {\em any} first-order method that obtains an $\eps$ suboptimal solution of $f$ after $T$ iterations can generate solutions that are $\Omega(\nicefrac{1}{(\sqrt{T}\eps)})$ apart in $\ell_2$ distance.  
Thus, if we run the method for the standard $T=O(\nicefrac{1}{\eps^2})$ iterations required to obtain $\eps$-approximate solution for a non-smooth function, then obtained solutions can deviate by $\Omega(1)$ distance; i.e., the method is \emph{maximally\/} irreproducible. To ensure that irreproducibility is small, i.e. that the solutions are within a small distance $\gamma$ of each other, we will need to run at least $\Omega(\nicefrac{1}{(\eps^2\gamma^2)})$ iterations. Interestingly, standard gradient descent with appropriately chosen learning rate and number of iterations already achieves this trade-off. 

  Our results demonstrate the challenge of reproducibility even for standard convex optimization. While we provide matching lower and upper bounds in certain general settings, in \autoref{sec:discussion}, we outline several important open directions. 
  Solutions to these problems should enable better understanding of reproducibility even for deep learning. 

\subsection{Summary of results}
\label{sec:summary}

\begin{table}[!t] 
\caption{\small Summary of $(\eps, \delta)$-deviation bounds for various convex optimization settings. 
}
\centering
\begin{tabular}{ |l |c|c|c|  }
\hline \multirow{3}{*}{}  & Stochastic Inexact & Non-stochastic Inexact & Inexact Initialization \\ 
& Gradient Oracle & Gradient Oracle & Oracle \\
&  \autoref{thm:sto}  & \autoref{thm:nonsto}  &  \autoref{thm:init}\\
\hline\hline  
Smooth   & $\Theta(\nicefrac{\delta^2}{(T\eps^2)})$ & $\Theta(\nicefrac{\delta^2}{\eps^2})$ & $\Theta(\delta^2)$  \\
\hline
 Smooth Strongly-Cvx. & $\Theta(\nicefrac{\delta^2}{T}\wedge \eps)$ & $\Theta(  \delta^2  \wedge  \eps  )$   &   $\Theta(e^{-\Omega(T)}\delta^2\wedge \eps)$ \\
\hline
 Nonsmooth   & $\Theta(\nicefrac{1}{(T\eps^2)})$ & $\Theta( \nicefrac{1}{(T\eps^2)}  + \nicefrac{\delta^2}{\eps^2}  )$ &  $\Theta( \nicefrac{1}{(T\eps^2)} +\delta^2)$  \\
\hline
  Nonsmooth Strongly-Cvx. & $\Theta(\nicefrac{1}{T} \wedge {\eps})$ & $\Theta( (\nicefrac{1}{T}  + {\delta^2})  \wedge {\eps} )$   &   $\Theta(\nicefrac{1}{T} \wedge {\eps})$ \\
  \hline  
\end{tabular}
\label{tab:sto1} 
\end{table}

 \autoref{tab:sto1} summarizes our key results for our measure of irreproducibility, $(\eps, \delta)$-deviation (see  \autoref{def:deviation}). The $(\eps, \delta)$-deviation measures the amount of change between the outputs of two independent runs of an optimization algorithm, that is guaranteed to achieve $\eps$-suboptimality after $T$ iterations, when the computations of the algorithm incur errors of magnitude up to $\delta$.
We specifically focus on three different sources of errors: 
i) stochastic gradient oracles, ii) gradient oracles with non-deterministic numerical errors (\autoref{def:inexactgrad}), and iii) inexact initialization for the optimizer (\autoref{def:inexactinit}). We analyze the deviation under these sources of errors for four types of function classes: smooth convex functions, non-smooth but Lipschitz convex functions and strongly-convex restrictions of the two.
Throughout the paper, $a\wedge b$ denotes the minimum between $a$ and $b$.

All lower bounds are for first-order iterative algorithms (\emph{\`{a} la} \cite{nesterov2018lectures}) that we formally define in \eqref{exp:foi}. This is a large class of iterative optimization methods, including Stochastic Gradient Descent (SGD), which construct successive iterates {\em adaptively} in the linear span of previous iterates. Additionally, for smooth costs and stochastic inexact gradient oracle, we have an information theoretic lower bound of $\Omega(\nicefrac{\delta^2}{(T\eps^2)})$ when $\eps\lesssim \delta^2$  (\autoref{thm:lb_sto_info}). We believe such informtation-theoretic lower bounds can be shown for all the settings in this paper. As for the upper bounds, they are all obtained using {\em slowed-down} SGD: i.e. SGD using smaller learning rates and more iterations.
 
For the non-strongly convex cases, one may expect to have high irreproducibility if the minima form a large flat region; however, surprisingly, our upper bounds show that we can always bound the extent of irreproducibility via slowed-down SGD.
In the non-smooth cases, the main observation is that the deviation {\em does not} depend on scale of perturbation by the gradient oracle, i.e., any $\delta>0$ can lead to fairly irreproducible solutions. The non-stochastic gradient oracle setting is strictly harder than the stochastic setting. Naturally, the lower and upper bounds on reproducibility are worse. Interestingly, even though strong convexity implies uniqueness of the global optimum, which intuitively should lead to highly reproducible solutions, we show that when faced with sources of error in computations, the deviation can still be significantly large for any algorithm. 

Finally, we study reproducibility of optimization in machine learning settings. Here we have additional structure such as finite-sum minimization (for optimizing training loss) and stochastic convex optimization (for optimizing population loss). We define appropriate notions of errors for these problems and analyze two settings of particular interest. Our main results (\autoref{thm:ub_finite_nonsmooth} and \autoref{thm:lb_online}) show that despite the additional structure in these problems, the bounds given by \autoref{tab:sto1} for the specific settings are nonetheless {\em tight}. One consequence is that more sophisticated techniques for these problems such as variance reduction don't improve reproducibility.

\subsection{Related work}

\paragraph{Related notions.} In the scientific world, the terms {\em reproducibility} and {\em replicability} are often used interchangeably, but here we distinguish the two, following \citet{pineau2021improving}. Reproducibility refers to the requirement that results obtained by a computational procedure (e.g. an experiment or a statistical analysis of a data set) should be the same (or largely similar) when the procedure is repeated using the same code on the same data, whereas replicability is a different notion that requires that results be reliably the same or similar when the data are changed. The field of statistical hypothesis testing \citep{lehmann-romano} provides rigorous and principled techniques to minimize false discoveries and thereby promote replicability. The notion of {\em algorithmic stability}  can also be seen as quantifying the amount of change in the output when a single data sample is changed. This notion has been extensively studied in the context of providing algorithm-dependent generalization bounds~\citep{bousquet2001algorithmic,KutinN02} and in developing differentially private algorithms~\citep{DworkMNS06,McSherryT07}. In very recent concurrent work, \citet{impagliazzo2022reproducibility} define a notion of {\em replicability} in learning that is quite different from ours: they aim to develop algorithms that generate the {\em exact same output} with reasonable probability given a fresh sample. Note that despite the title of their paper, technically the notion studied is replicability, not reproducibility, since they study the output of algorithms when the input data are changed.

In this paper, we specifically focus on reproducibility: how much can the results of a computation differ when it is re-run on the same data with the same code? Hence, both hypothesis testing and algorithmic stability
are orthogonal to the study in this paper, although some of our upper bounds use similar analysis techniques as algorithmic stability.
On a different note, similar to replicability, the boundary between the notions of irreproducibility and {\em uncertainty} in deep models is rather blurred.  Several papers considered different aspects of uncertainty (see, e.g., \citep{lakshminarayanan17} and references therein), but this line of work has been empirical in nature.  

\paragraph{Inexact oracles in optimization.} The optimization community has studied the consequences of using inexact or error-prone gradient oracles in optimization. Several papers (e.g. \citep{aybat2020robust,devolder2014first,d2008smooth,cohen2018acceleration} have developed bounds on the optimization error incurred due to the use of inexact oracles. While the sources of errors are similar to the ones studied in this paper, the quantities of interest in these papers are convergence rate and optimization error rather than reproducibility. Interestingly, despite the different objective, some of the high-level conclusions are similar to our paper: for example, accelerated gradient methods do not outperform standard classical methods when used with inexact gradient oracles.







\paragraph{Techniques to improve reproducibility in practice.}
Several recent empirical papers considered methods that can reduce levels of irreproducibility in deep models despite nondeterminism in training. Smooth activations \citep{du19, mhaskar97} have been shown \citep{shamir2020smooth} to improve reproducibilty over popular activations, as the Rectified Linear Unit \citep{nair10}.  
Ensembles \citep{dietterich00} leverage diversity of multiple different solutions to produce an average more reproducible one \citep{allen2020towards}. Co-distillation \citep{anil18} and Anti-distillation \citep{shamir2020anti} leverage ensembles to further push deployed models to be more reproducible.  Imposing constraints \citep{bhojanapalli2021reproducibility,shamir18} forces models to prefer some solutions over others, but may come at the cost of reducing model accuracy performance.

\paragraph{Robustness of dynamical systems.} The upper bound results in our paper can be interpreted as robustness results of the (sub)gradient descent dynamics against disturbances.
In particular, our upper bounds can be viewed as some variants of the input-to-state stability~\citep{sontag1995characterizations}  results for the dynamics (see, e.g., \cite[Definition 3.2]{tu2021sample}).

\section{Problem Formulation}
\label{sec:problem}
In this section, we define a quantitative measure of {\em irreproducibility} amenable to a theoretical analysis. Intuitively, a computation is {\em reproducible} if it generates the exact same output given the same inputs on two different runs.  Irreproducibility arises because low-level operations of a computation produce different answers on two runs due to either {\em randomness} or {\em non-determinism}. 

Our computation of interest is convex optimization via first-order methods, where {\em initialization} and {\em gradient computations} are the primary operations that constitute the computation and are subject to errors leading to inexact outputs. A natural measure of irreproducibility is the amount of change in the computed solution to the convex optimization problem under inexact gradient computations or inexact initialization. However, there are two nuances that must be carefully handled here. First, a trivial procedure which ignores its input and outputs a constant solution is perfectly reproducible! Unfortunately, it is perfectly useless as a convex optimization procedure as well. Thus, in order to compare different procedures by their reproducibility metrics, we must assume that the procedures are guaranteed to converge to an optimal solution. The second nuance is that we need to assume that the errors in the gradient or initialization computations are {\em bounded} in some manner. Evidently, without such an assumption, any non-trivial convex optimization procedure will be extremely irreproducible. We now use the above considerations to develop a precise definition of a measure of irreproducibility.

 
\paragraph{Convex function classes.} We assume that the function to be optimized is chosen from a certain class, $\mathcal{F}$, of convex functions, along with their domains, satisfying suitable {\em regularity} conditions (e.g. Lipschitzness, smoothness, strong-convexity, etc.) to develop convergence rates.  
For clarity, we will suppress exact dependence on smoothness, Lipschitzness, and strong-convexity parameters. In particular, ``smooth'' will denote a convex function whose gradients are $O(1)$-Lipschitz continuous, ``non-smooth'' a convex function that is $O(1)$-Lipschitz continuous, and ``strongly-convex'' an $\Omega(1)$-strongly-convex function. 
Here, $O(1)$ and $\Omega(1)$ denote universal constants independent of the dimension or other problem dependent quantities, which we leave unspecified to ease the exposition.

\paragraph{Convex optimization procedures.} 
A {\bf \em first-order} convex optimization procedure for $\mathcal{F}$ is an algorithm that, given any function $f \in \mathcal{F}$, and access to two (potentially noisy) oracles -- an initialization oracle, which generates the initial point, and a gradient oracle, which computes gradients for $f$ at any given query point --  generates a candidate solution $\vx_{\mathrm{out}}$ for the problem of minimizing $f$ over its domain. Note that the algorithm can only access $f$ via the oracles provided. We call such an algorithm {\bf \em $\eps$-accurate} if it guarantees that $\E f(\vx_{\mathrm{out}}) - \inf_{\vx \in \text{dom} f} f(\vx) \leq \eps$, where the expectation is over any randomness in the computation of $\vx_{\mathrm{out}}$. Several of our lower bounds require more structure for the algorithm: specifically, a {\bf \em first-order iterative} (\sfa) algorithm (\emph{\`{a} la} \cite{nesterov2018lectures}) is one that starting from the point $\vx_0$ generated by the initialization oracle, constructs successive iterates
\begin{align} \tag{\sf FOI} \label{exp:foi}
     \vx_t =  \vx_0 - \textstyle\sum_{i=0}^{t-1} \la^{(t)}_i g(\vx_i)\quad \text{for some }\lambda_i^{(t)},~ i=0,\dots, t-1,
\end{align} 
where $g(\vx_i)$ is the output of the gradient oracle query at $\vx_i$, and outputs $\vx_T$ for some integer $T > 0$.
We emphasize that for all $t$, the coefficients $\lambda^{(t)}_i$ can be chosen adaptively based on all the previous computations.
The above class of algorithms is a canonical one to consider when proving lower bounds against gradient oracle based optimization algorithms. We refer readers to \cite[\S 2.1.2]{nesterov2018lectures} for more background.
For the case of nonsmooth costs, we additionally assume that the coefficient of the latest gradient is nonzero, i.e., $\la^{(t)}_{t-1}\neq 0$ for all $t$. 
We also note that one of our lower bound results (\autoref{thm:lb_sto_info}) is \emph{information-theoretic} (in the sense of  \cite{nemirovskij1983problem}).

\paragraph{Sources of errors in computation.} Errors arise due to inexactness in the outputs of the initialization or gradient oracles. Queries to these oracles on two different runs of the same algorithm might yield different outputs, but we will control the errors by assuming that the outputs are close to some reference point (that remains fixed over different runs) in a suitable metric. 
\begin{definition}[$\delta$-bounded inexact initialization oracle]
\label{def:inexactinit}
  Given a function $f \in \mathcal{F}$ and a reference initialization point $\vx_0^\text{ref} \in \text{dom} f$, a {\bf \em $\delta$-bounded} inexact initialization oracle for $f$ is one that generates an initial point $\vx_0 \in \text{dom} f$ such that $\|\vx_0 - \vx_0^{\text{ref}}\| \leq \delta$. 
\end{definition}
The gradient computation oracle is said to be {\bf \em $\delta$-bounded} if  for any $f \in \mathcal{F}$ and any point $\vx \in \text{dom} f$, it outputs a vector $g(\vx)$ such that $\E\norm{g(\vx) -\nabla f(\vx)}^2 \leq \delta^2$ for some $\nabla f(\vx) \in \partial f(\vx)$, where the expectation is over any randomness in the computation of $g(\vx)$. We consider both {\em stochastic} and {\em non-stochastic} inexact $\delta$-bounded gradient oracles. A stochastic gradient oracle has the additional property that its output $g(\vx)$ is a random vector such that $\E g(\vx) = \nabla f(\vx)$, with different queries being independent of each other. Stochastic inexact gradient oracles arise naturally in machine learning applications due to randomness in minibatching. 
Non-stochastic inexact gradient oracles model non-deterministic numerical errors due to the accumulation of floating point errors; for giant machine learning models with billions of parameters, individual floating point errors could add up to a noticeable large error.
We formally define the two types of inexact oracles below. 
\begin{definition}[$\delta$-bounded inexact gradient oracle]
\label{def:inexactgrad}
Given a function $f \in \mathcal{F}$, and $\vx\in \text{dom} f$, let $\partial f(\vx)$ denote the sub-differential of $f$ at $\vx$.
\vspace{-7pt}
\begin{list}{{\tiny $\blacksquare$}}{\leftmargin=0.5em}
		\setlength{\itemsep}{-4pt} 
\item[(a)] A  {\bf stochastic inexact}  $\delta$-bounded gradient oracle outputs a random vector $g(\vx)$ such that $\E{g(\vx)}=\nabla f(\vx)$ and $\E\norm{g(\vx) -\nabla f(\vx)}^2 \leq \delta^2$ for some $\nabla f(\vx) \in \partial f(\vx)$. The expectation is over the randomness in $g(\vx)$ which is also assumed to be independent for each oracle call. 
  \item[(b)] A  {\bf non-stochastic inexact}  $\delta$-bounded gradient oracle outputs a non-deterministic vector $g(\vx)$ such that
  $\norm{g(\vx)-\nabla f(\vx)}^2 \leq \delta^2$ for some $\nabla f(\vx) \in \partial f(\vx)$.
\end{list}
\end{definition}  

\paragraph{Measure of irreproducibility.}
Let $\alg$ be a first-order, $\eps$-accurate convex optimization procedure for $\mathcal{F}$ with access to either a $\delta$-bounded initialization oracle (and an exact gradient oracle), or a $\delta$-bounded gradient oracle (and an exact initialization oracle). 
The {\bf \em $(\eps, \delta)$-deviation} is
\begin{definition}[$(\eps,\delta)$-deviation]
\label{def:deviation}
Given a function  $f \in \mathcal{F}$, let $\vx_{f}$ and $\vx'_{f}$ denote the outputs of $\alg$ on two independent runs of $\alg$  that result in $\eps$-accurate solutions.\vspace{-7pt}
\begin{list}{{\tiny $\blacksquare$}}{\leftmargin=0.5em}
		\setlength{\itemsep}{-4pt} 
  \item[(a)] If a stochastic inexact $\delta$-bounded gradient oracle is used, the $(\eps,\delta)$-deviation of $\alg$ is defined as $\sup_{f \in \mathcal{F}} \E \|\vx_f - \vx'_f\|^2$ where the randomness is over the stochastic oracle in the two runs.
  \item[(b)] If either an inexact $\delta$-bounded initialization oracle or a $\delta$-bounded non-stochastic inexact gradient oracle is used, the $(\eps,\delta)$-deviation of $\alg$ is defined as $\sup_{f \in \mathcal{F}} \sup 
 \|\vx_f - \vx'_f\|^2$,
  where the inner supremum is over the two runs.
\end{list}
\end{definition}
Here we note that $\vx_f$ is not necessarily the last iterate of a first-order algorithm. Depending on cases, sometimes we choose $\vx_f$ to be the (weighted) average iterate.
We consider $\|x_f - x'_f\|^2$ for the notion of deviation instead of $\|x_f - x^*\|^2$; note that $x^*$ may not even be unique without strong convexity.

\paragraph{Other notions?} Alternate definitions are certainly plausible. For example, in machine learning applications, the computed solution $\vx$ may be used as a parameter vector for a predictor function $g_{\vx}$ which maps inputs $z$ to real-valued predictions $\hat{y}$ whose quality is measured by a loss function $\ell(\hat{y}, y)$ where $y$ is the true label of $z$.
Let $\vx$ and $\vx'$ be outputs of two different runs of the optimization algorithm. Then one can also define $(\eps, \delta)$-deviation based on \emph{prediction reproducibility}: e.g. $\sup_{z} |g_{\vx}(z) - g_{\vx'}(z)|$ or $\E_{z} |g_{\vx}(z) - g_{\vx'}(z)|$, or \emph{loss reproducibility}: e.g. $\sup_{(z, y)} |\ell(g_{\vx}(z), y) - \ell(g_{\vx'}(z), y)|$ or $\E_{(z, y)} |\ell(g_{\vx}(z), y) - \ell(g_{\vx'}(z), y)|$, where the expectation is over the distribution of the examples. Nevertheless, we adopt \autoref{def:deviation}, which is based on \emph{parameter reproducibility}, in this paper for the following reasons:
\begin{itemize}[leftmargin=*]
\item The convex optimization problems studied in this paper are {\bf more basic/fundamental} than more structured ML optimization. To the best of our knowledge, there is no pre-existing theory even for this basic setting. Parameter reproducibility is a more natural definition here since there is no notion of prediction or loss. 
\item In many ML applications, the predictor function $g_{\vx}$ is Lipschitz in $\vx$ for any input $z$. In such cases, $(\eps, \delta)$-deviation bounds for parameter reproducibility immediately transform into $(\eps, \delta)$-deviation bounds for prediction reproducibility. Similar transformations are also generally possible $(\eps, \delta)$-deviation bounds for loss reproducibility.
\item Without knowledge of how a learned parametric function is deployed for making predictions, it is difficult to analyze prediction reproducibility even in ML settings.  Hence, parameter deviations provide a reasonable first approximation. Furthermore, in real systems, we often optimize multiple metrics (e.g. performance on sub-segments of populations, for fairness). Parameter reproducibility gives greater assurance on all metrics. A similar argument applies when the test distribution is different from the training distribution.
\item Finally, several recent works~\citep{shamir2020smooth, damour20} have empirically observed that even if two different runs of the algorithm resulted in parameters that have nearly the same loss, the predictions on test examples could be very different. One reason this happens is because {\em surrogate} losses are used in place of the true metric for optimization. So simply achieving loss reproducibility may not be sufficient for practical applications.
\end{itemize}

\section{Reproducibility with Stochastic Inexact Gradient Oracles}

In this section we consider reproducibility of optimizing a convex function $f \in \mathcal{F}$ where we can access $f$ only via a stochastic gradient oracle (see \autoref{def:inexactgrad}). This setting covers several important ML optimization scenarios, e.g., when the training data is randomly sampled from a  population or  when the selection of mini-batches is randomized. Our main result is the following theorem.

\begin{theorem}\label{thm:sto}
For any $\eps, \delta > 0$, and number of iterations $T$, the $(\eps, \delta)$-deviation for optimizing convex functions with a stochastic inexact gradient oracle is as follows. Unless indicated otherwise, the lower bounds hold for any \sfa algorithm, and the upper bound is achieved by stochastic gradient descent for $T = \Omega(\nicefrac{1}{\eps^2})$ in the non-strongly-convex settings and $T = \Omega(\nicefrac{1}{\eps})$ in the strongly-convex settings. \vspace{-8pt}
\begin{list}{{\tiny $\blacksquare$}}{\leftmargin=0.5em}
		\setlength{\itemsep}{-4pt} 
\item \textbf{Smooth functions}: $(\eps, \delta)$-deviation is $\Theta(\nicefrac{\delta^2}{(T\eps^2)})$.  
Furthermore,  for the case $\eps\lesssim \delta^2$ there is a matching {\bf information theoretic lower bound}.

\item \textbf{Smooth and strongly convex functions}: $(\eps, \delta)$-deviation is $\Theta(\nicefrac{\delta^2}{T}\wedge  \eps)$.  

\item \textbf{Lipschitz (non-smooth) functions}: $(\eps, \delta)$-deviation is $\Theta(\nicefrac{1}{(T\eps^2)})$.  
\item \textbf{Lipschitz (non-smooth) and strongly convex functions}:
$(\eps, \delta)$-deviation is $\Theta(\nicefrac{1}{T}\wedge  \eps)$. 
\end{list}
\end{theorem}
This theorem is proved in several pieces, one corresponding to each setting (i.e., smooth, smooth \& strongly convex etc.) and type of bound (i.e., upper or lower). The precise details are given in \autoref{sec:summary-app}.
A few remarks are in order.

\textbf{Smoothness}: Intuitively, smoothness would ensure that a slight error in gradient computation need not imply catastrophic deviation in the iterates. Our matching upper and lower bounds confirm this intuition.

\textbf{Non-smoothness}: Our results show that even a slight amount of noise in the gradients can lead to drastic irreproducibility in the non-smooth case. Intuitively, the reason for this phenomenon is that non-smooth functions have non-differentiable points where the slightest amount of noise in the gradients can lead to drastically different behavior. This is in line with the empirical observations of \citet{shamir2020smooth}.

\textbf{Strong convexity}:  Note that the deviation is smaller than other cases due to the existence of a unique minimizer; the $\eps$-accuracy in the cost already implies  an $O(\eps)$ upper bound on the deviation.  

Furthermore, in all the settings, our results show that using an algorithm with a larger number of iterations is more helpful; intuitively that is because more gradient samples from the oracle can help reduce the sample noise as an averaging. For example, for smooth functions, if a gradient descent type method is run for the standard $T = O(\nicefrac{1}{\eps^2})$ iterations, then the solution obtained might still suffer deviation of $\delta^2$, which is independent of $\eps$. Thus, in order to obtain low deviation, we are forced to run the method for $\omega(\nicefrac{1}{\eps^2})$ iterations. 
We also note that by relying on the standard convergence rate lower bounds~\citep{nemirovskij1983problem}, the requirement on $T$ e.g., $T=\Omega(1/\eps^2)$ for the non-strongly convex setting, is without loss of generality.
Finally, we remark that the lower bound for smooth functions is information theoretic (i.e., holds against \emph{any} algorithm) for $\eps \lesssim \delta^2$.

\section{Reproducibility with Non-Stochastic Inexact Gradient Oracles}
In this section, we study reproducibility for optimizing a function $f\in \mathcal{F}$ with non-stochastic inexact gradient oracle access (\autoref{def:inexactgrad}). In particular, we establish lower and  upper  bounds for first-order algorithms \eqref{exp:foi} on the $(\eps,\delta)$-deviation (\autoref{def:deviation}) despite initialization with the same point $x_0 \in \text{dom} f$. Recall that this setting allows us to capture reproducibility challenges due to non-deterministic numerical errors introduced by a computing device (like GPUs) during floating point computations.  The main high-level message here is that unlike in the stochastic gradient oracle case, in the non-stochastic gradient oracle setting the iteration complexity $T$ has little to no effect on the $(\eps,\delta)$-deviation: intuitively this is because unlike the stochastic setting, it is not possible to reduce the error in the gradients by taking more samples and averaging.

\begin{theorem}\label{thm:nonsto}
For $D = O(1)$, $\eps, \delta >0$ such that $\delta \leq \nicefrac{\eps}{(2D)}$, and number of iterations $T$, the $(\eps, \delta)$-deviation for optimizing convex functions  whose optimum has norm at most $D$ with a non-stochastic inexact gradient oracle is as follows. 
Unless indicated otherwise, the lower bounds hold for any \sfa algorithm, and the upper bound is achieved by projected gradient descent for $T = \Omega(\nicefrac{1}{\eps})$ in the smooth settings and non-smooth strongly convex setting, and $T = \Omega(\nicefrac{1}{\eps^2})$ in the non-smooth setting.
\vspace{-8pt}
\begin{list}{{\tiny $\blacksquare$}}{\leftmargin=0.5em}
		\setlength{\itemsep}{-4pt} 
\item \textbf{Smooth functions}: $(\eps, \delta)$-deviation is $\Theta(\nicefrac{\delta^2}{\eps^2})$.\footnote{Note that \autoref{lem:without_proj} provides a similar upper bound result without the assumption that the optimum lies in a ball of radius $O(1)$.}

\item \textbf{Smooth and strongly convex functions}: $(\eps, \delta)$-deviation is $\Theta(\delta^2\wedge \eps)$.

\item \textbf{Lipschitz (non-smooth) functions}: $(\eps, \delta)$-deviation is $\Theta(\nicefrac{1}{(T\eps^2)} +\nicefrac{\delta^2}{\eps^2})$. 

\item \textbf{Lipschitz (non-smooth) and strongly convex functions}: $(\eps, \delta)$-deviation is $\Theta( (\nicefrac{1}{T} +\delta^2 )\wedge \eps )$. 
\end{list}
\end{theorem}
Like \autoref{thm:sto}, this theorem is proved in several pieces; see \autoref{sec:summary-app}. Some remarks follow:

{\bf Acceleration v.s. reproducibility?}
  Note that for smooth functions, accelerated methods can get $\eps$-suboptimality in only $T=O(\nicefrac{1}{\sqrt{\eps}})$ iterations. A natural question is if we can achieve similar $(\eps, \delta)$-deviation for such accelerated methods. However, it is known that accelerated methods are unstable (see e.g., \citep{devolder2014first,attia2021instability}), and cannot even achieve the desired $\eps$-accuracy under the inexact oracle model. Hence, we conjecture that iteration complexity of $\Omega(1/\eps)$ is necessary to achieve the desired reproducibility.

 {\bf Nonsmoothness:} Note that the lower bound shows that if gradient descent style methods are run for the standard $T=1/\eps^2$ iterations, then the solutions can be forced to have $\Omega(1)$  deviation. Hence, to ensure reproducibility, any method would need to have a slower convergence rate than  $T=1/\eps^2$. 

\section{Reproducibility with Inexact Initialization Oracles}\label{sec:init}

We now study reproducibility for optimizing a function $f\in \mathcal{F}$ with an inexact initialization oracle (\autoref{def:inexactinit}) and establish tight bounds on the $(\eps,\delta)$-deviation for \sfa algorithms. An inexact initialization oracle models situations where the initial values of the parameters are set randomly, or incur some non-deterministic numerical error.

\begin{theorem}\label{thm:init}
For any $\eps, \delta > 0$ and number of iterations $T$,  the $(\eps, \delta)$-deviation for optimizing convex functions  with an inexact initialization oracle is as follows. 
In all cases, the lower bounds hold for any \sfa algorithm, and the upper bound is achieved by gradient descent for $T = \Omega(\nicefrac{1}{\eps})$ in the smooth settings and non-smooth strongly convex setting, and $T = \Omega(\nicefrac{1}{\eps^2})$ in the non-smooth setting. 
\vspace{-8pt}
\begin{list}{{\tiny $\blacksquare$}}{\leftmargin=0.5em}
		\setlength{\itemsep}{-4pt} 
\item \textbf{Smooth functions}: $(\eps, \delta)$-deviation is $\Theta (\delta^2)$.
\item \textbf{Smooth and strongly convex functions}: $(\eps, \delta)$-deviation is $\Theta((\exp(-\Omega(T))\delta^2 \wedge \eps )$.

\item \textbf{Lipschitz (non-smooth) functions}: $(\eps, \delta)$-deviation is $\Theta(\nicefrac{1}{(T\eps^2)}+\delta^2)$.

\item \textbf{Lipschitz (non-smooth) and strongly convex functions}: $(\eps, \delta)$-deviation is $\Theta(\nicefrac{1}{T}\wedge \eps)$.
\end{list}
\end{theorem}
Like \autoref{thm:sto}, this theorem is proved in several pieces; see \autoref{sec:summary-app}. A remark follows:


 {\bf Nonsmoothness:} As in the case of non-smooth optimization with inexact gradient oracles, here as well we see that even the slightest amount of inexactness in initialization can lead to non-negligible irreproducibility. Intuitively this is a consequence of non-smoothness: If the function is not differentiable at the reference point, even a slight bit of inexactness can lead to drastically different trajectories right from the beginning.

\section{Reproducibility in Optimization for Machine Learning}
So far, we have studied reproducibility in general convex optimization with different sources of perturbation. In machine learning, however, optimization problems come with more structure, and hence a more nuanced analysis of reproducibility is called for. In this section, we study reproducibility in two specific optimization settings of interest in machine learning: finite sum minimization which corresponds to minimizing training loss, and stochastic convex optimization (SCO), which corresponds to minimizing population (or test) loss. Instead of a detailed study of different types of convex functions as done in the previous sections, here we focus on a few specific cases that yield particularly interesting insights.

\subsection{Optimizing Training Loss (Finite Sum Minimization)}
Optimizing the training loss in ML can be cast as minimizing a function that can be written as a finite sum of component functions: $
    f(\vx) := \frac{1}{m}\sum_{i=1}^m f_i(\vx)\,.$ 
Typical optimization methods such as stochastic gradient descent can be implemented to solve such problems by iteratively sampling one of the component functions and taking its gradient. So randomness in sampling as well as non-deterministic numerical errors in computing of gradient can lead to irreproducibility. Randomness in sampling is a property of the algorithm and hence under the control of the algorithm designer; however non-deterministic numerical errors in gradient computations are beyond the control of the designer, and hence we aim to quantify irreproducibility caused by this specific source (numerical errors). To capture this intuition, we consider the following inexact gradient oracle. 
 
\begin{definition}[$\delta$-bounded inexact component gradient oracle]
\label{def:det_finite}
A $\delta$-bounded inexact {\bf component} gradient oracle takes as input $i \in \{1, 2, \ldots, m\}$ and a point $x$, and outputs a vector $g_i(\vx)$ such that
\begin{align*}
    \norm{g_i(\vx) -\nabla f_i(\vx)}^2 \leq \delta^2 \text{ for some } \nabla f_i(\vx) \in \partial f_i(\vx).
\end{align*} 
\end{definition}
The number of calls to the component gradient oracle is now our primary measure of complexity. Then, a natural question that arises is whether we can we tightly characterize reproduciblity of typical first-order methods when given access to an inexact component gradient oracle. We first observe that in the special case that all $f_i$ are identical, the inexact component gradient oracle reduces to the non-stochastic inexact oracle (\autoref{def:inexactgrad}). Hence, for the smooth and non-smooth function settings, we immediately obtain lower bounds on the $(\eps, \delta)$-deviation via our previous analysis; see  \autoref{thm:lb_det_smooth} and \autoref{thm:lb_det_nonsmooth}, respectively. 

The non-smooth case in particular is interesting, as the deviation lower bound is $\Omega(\nicefrac{1}{(T\eps^2)} +\nicefrac{\delta^2}{\eps^2})$, where $T$ is the number of oracle calls. Now, if we use projected (full-batch) gradient descent (as used for the matching upper bound in   \autoref{thm:ub_det_nonsmooth}), then each iteration would require $m$ oracle calls, hence the $(\eps, \delta)$-deviation for such a method is $O(\nicefrac{m}{(T\eps^2)} +\nicefrac{\delta^2}{\eps^2})$, which is worse than the lower bound mentioned above.  Now, the key question is whether we can reduce this deviation by other means. It turns out that {\em stochastic} gradient descent (SGD), where in each step we sample a component function to query randomly, when run with an appropriate learning rate and using averaging, 
matches the optimal deviation up to constant factors. We provide a detailed proof of the result in \autoref{sec:ub_finite}. This provides another justification for using SGD instead of gradient descent in practice.

 \begin{restatable}{theorem}{ubfinitenonsmooth}
 \label{thm:ub_finite_nonsmooth}
For $G=O(1)$ and $D=O(1)$, let $f_i$ be an $G$-Lipschitz convex cost function for each $i\in[m]$, and assume that the optimum of $f$ lies in a ball of radius $D$. 
Let $\eps, \delta > 0$ be given parameters such that $\delta \leq \nicefrac{\eps}{(2D)}$, and $T=\Omega(1/\eps^2)$ be a given number of iterations. Define the SGD updates as follows: initialize $\vx_0 = 0$, and for $t = 0, 1, \ldots, T-1$, set $\vx_{t+1} = \vx_t - \eta_t  g_{i_t}(\vx_t)$ where $i_t\sim [n]$ uniformly at random. Under the inexact component gradient oracle (\autoref{def:det_finite}),  the average iterate $\bar{\vx}_T$ of SGD with  stepsize $\eta=\Theta(\nicefrac{1}{(\eps T)})$ satisfies
$\E f(\bar{\vx}_T) - \inf_{\vx \in \text{dom} f} f(\vx) \leq \eps$ and $\E \norm{\bar{\vx}_T -\bar{\vx}'_T}^2 =O(\nicefrac{1}{(T\eps^2)} +\nicefrac{\delta^2}{\eps^2})$, where $\bar{\vx}'_T$ is the output of an independent run of SGD.
    \end{restatable}


\subsection{Optimizing Population Loss (Stochastic Convex Optimization)}
The fundamental problem of machine learning is to find the solution to the following {\em population} (or test) loss minimization problem: minimize $F(\vx) = \E_{\xi\sim \Xi} f(\vx,\xi)$, where $\Xi$ is an unknown distribution on the examples $\xi$, given access to an oracle that sample from $\Xi$. When the function $F$ is convex, this is also known as Stochastic Convex Optimization (SCO).

In this setting, the sampling oracle is a natural source of irreproducibility of optimization methods. To model this, we consider the \emph{stochastic global oracle}, inspired by  \citet{foster2019complexity}.
\begin{definition}[$\delta$-bounded stochastic global oracle] \label{def:global}
Given a function $F$, a $\delta$-bounded stochastic global oracle for $F$ is an algorithm that, on each query, draws an independent sample $\xi \sim \Xi$ from some distribution $\Xi$, and outputs a function $f(\cdot, \xi): \text{dom} F \rightarrow \mathbb{R}$ such that for all $x \in \text{dom} F$, we have $F(\vx) = \E_{\xi\sim \Xi} f(\vx,\xi)$, and, $ \E_{\xi\sim \Xi} \norm{\nabla f(\vx,\xi)-\nabla F(\vx)}^2 \leq \delta^2\,.$
\end{definition}
Since in each query, the stochastic global oracle returns the complete specification of $f(\cdot,\xi)$, an algorithm using such an oracle has access to $(\vx, f(\vx,\xi), \nabla f(\vx,\xi), \nabla^2 f(\vx,\xi), \cdots) $ for all $\vx$. 

Note that the inexact stochastic gradient oracle of \autoref{def:inexactgrad} is a weaker form of the stochastic global oracle. Thus, the key question is: {\em  under the powerful stochastic global oracle, can we provide better reproducibility than the lower bounds in \autoref{thm:sto}?}  
Surprisingly, the answer to the above question is negative. That is, despite the stronger oracle setting, there is a function class for which the lower bound on $(\eps,\delta)$-deviation matches that of  \autoref{thm:sto}. 

\begin{restatable}{theorem}{lbonline}  \label{thm:lb_online} Assume that $\eps<1/200$ and $\eps \lesssim\delta^2$.
Then there exists a family of population costs $\Fth{\theta}(\vx) = \E_{\xi\sim P_\theta} f_{\theta}(\vx,\xi)$ parametrized by $\theta\in[1,2]$, and a $\delta$-bounded stochastic global oracle for $\Fth{\theta}$, satisfying the following property. Suppose that $A$ is any algorithm that for each $\theta\in[1,2]$ uses at most $T$ queries to the stochastic global oracle  and  outputs $\vx_T^{\theta}$ that  is   $\eps$-accurate, i.e., $\E \Fth{\theta}(\vx_T^{\theta}) - \inf_{\vx} \Fth{\theta}(\vx)\leq \eps$.
 Then, there exists $\thebad\in[1,2]$ such that  $\var(\vx_T^{\thebad}) \geq \Omega(\nicefrac{\delta^2}{(T\eps^2)})$.
 \end{restatable}
 \noindent See \autoref{pf:lb_online} for a detailed proof. On the other hand,  \autoref{thm:sto} also shows that standard SGD with the weaker inexact stochastic gradient oracle achieves this lower bound.

 \section{Conclusions}
 \label{sec:discussion}
We presented a framework for studying reproducibility in convex optimization, and explored limits of reproducibility for optimizing various function classes under different sources of errors. 
For each of these settings, we provide tight lower and upper bounds on reproducibility of first order iterative algorithms. 
Overall, our results provide the following insights: a) Non-smooth functions can be highly susceptible to even tiny errors which can creep in due to say numerical errors in GPU. Thus introducing smoothness in deep learning models might help with reproducibility. b) Generally, gradient descent type methods with small learning rate are more reproducible. c) In finite-sum settings, despite more randomness, SGD is more reproducible than full-batch gradient descent. 
 
The study in this paper is a first step towards addressing the challenging problem of reproducibility in a rigorous framework. So, many important directions remain unexplored. For example, study of reproducibility of adaptive methods like AdaGrad that are not captured by \eqref{exp:foi} is interesting. We study reproducibility in a strict form where we measure deviation in terms of the learned parameters of the model. However, in practice, one might care for reproducibility in model predictions only. 
So, extending our results to such ML-driven scenarios should be relevant in practice. Finally, while our lower bounds on reproducibility obviously also hold for non-convex optimization, extension of the upper bounds to non-convex settings and potentially, designing novel rigorous methods in this setting is a fascinating direction for future work.

\begin{ack} 
Kwangjun Ahn was supported by graduate assistantship from the NSF Grant (CAREER: 1846088), the ONR grant (N00014-20-1-2394) and MIT-IBM Watson as well as a Vannevar Bush fellowship from Office of the Secretary of Defense.
Part of this work was done while Kwangjun Ahn was visiting the Simons Institute for the Theory of Computing, Berkeley, California.
\end{ack}

\bibliographystyle{abbrvnat} 
\bibliography{ref}

\newpage
\appendix


\addcontentsline{toc}{section}{Appendix} 
\part{Appendix} 
\parttoc

\section{Summary of results in Appendix}
 
\label{sec:summary-app}
In order to help readers navigates the results in the appendix, we summarize the results in the appendix in the following table. 
In the main paper, we have three main theorems: Theorems \ref{thm:sto}, \ref{thm:nonsto} and \ref{thm:init}, each corresponding to one column in the following table. Further each of these theorems have four components corresponding to four settings: smooth, smooth \& strongly convex, nonsmooth and finally nonsmooth \& strongly convex. Each cell in the below table lists the corresponding theorems which prove the lower and upper bounds corresponding to a given setting. Theorems \ref{thm:sto}, \ref{thm:nonsto} and \ref{thm:init} follow immediately from the constituent theorems.
In each cell of results, ``Theorems {\bf X}\sep{} {\bf Y}'' indicates that the lower bound appears in Theorem {\bf X}, and the upper bound in Theorem {\bf Y}. All lower bounds are for first-order iterative algorithms (\emph{\`{a} la} \cite{nesterov2018lectures}) that we formally defined in \eqref{exp:foi}. 
\begin{table}[h]  
\centering
\begin{tabular}{ |l |c|c|c|  }
\hline \multirow{3}{*}{}  & Stochastic Inexact & Non-stochastic Inexact & Inexact Initialization \\ 
& Gradient Oracle & Gradient Oracle & Oracle \\ 
& (\autoref{thm:sto}) &  (\autoref{thm:nonsto}) &  (\autoref{thm:init}) \\
\hline\hline  
\multirow{2}{*}{Smooth}   & $\Theta(\nicefrac{\delta^2}{T\eps^2})$ & $\Theta(\nicefrac{\delta^2}{\eps^2})$ & $\Theta(\delta^2)$  \\
 & Theorems~\ref{thm:lb_sto_info}$^{\dag}$\&\ref{thm:lb_sto_first}\sep{}  \ref{thm:ub_sto_smooth}  & Theorems~\ref{thm:lb_det_smooth}\sep{}  \ref{thm:ub_det_smooth} & Theorems~\ref{thm:lb_init_smooth}\sep{}  \ref{thm:ub_init_smooth}\\
   \hline \hline
 Smooth  & $\Theta(\nicefrac{\delta^2}{T}\wedge \eps)$ & $\Theta(  \delta^2  \wedge  \eps  )$   &   $\Theta(e^{-\Omega(T)}\delta^2\wedge \eps)$ \\
  Strongly-Convex & Theorems~\ref{thm:lb_sto_smooth_str}\sep{} \ref{thm:ub_sto_smooth_str}  & Theorems~\ref{thm:lb_det_smooth_str}\sep{} \ref{thm:ub_det_smooth_str} & Theorems~\ref{thm:lb_init_smooth_str}\sep{} \ref{thm:ub_init_smooth_str}\\ 
  \hline \hline
 \multirow{2}{*}{Nonsmooth}   & $\Theta(\nicefrac{1}{T\eps^2})$ & $\Theta( \nicefrac{1}{T\eps^2}  + \nicefrac{\delta^2}{\eps^2}  )$ &  $\Theta( \nicefrac{1}{T\eps^2} + \delta^2)$  \\
  & Theorems~\ref{thm:lb_sto_nonsmooth}\sep{} \ref{thm:ub_sto_nonsmooth}  & Theorems~\ref{thm:lb_det_nonsmooth}\sep{} \ref{thm:ub_det_nonsmooth} &  Theorems~\ref{thm:lb_init_nonsmooth}\sep{} \ref{thm:ub_init_nonsmooth}  \\ 
    \hline \hline
  Nonsmooth  & $\Theta(\nicefrac{1}{T} \wedge {\eps})$ & $\Theta( (\nicefrac{1}{T}  + {\delta^2})  \wedge {\eps} )$   &   $\Theta(\nicefrac{1}{T} \wedge {\eps})$ \\
  Strongly-Convex & Theorems~\ref{thm:lb_sto_nonsmooth_str}\sep{} \ref{thm:ub_sto_nonsmooth_str}  & Theorems~\ref{thm:lb_det_nonsmooth_str}\sep{} \ref{thm:ub_det_nonsmooth_str} & Theorems~\ref{thm:lb_init_nonsmooth_str}\sep{} \ref{thm:ub_init_nonsmooth_str}\\
  \hline  
\end{tabular}
\label{tab:app} 
\\[1em]
\noindent $\dag$ For smooth costs and stochastic inexact gradient oracle, we also have an additional {\bf information-theoretic} lower bound of $\Omega(\nicefrac{\delta^2}{T\eps^2})$ when $\eps\lesssim \delta^2$ in \autoref{thm:lb_sto_info}.
\end{table}
\subsection{General guidance for navigating Appendix}
\begin{mdframed}[backgroundcolor=emph]
\begin{list}{{\tiny $\blacksquare$}}{\leftmargin=1.5em}
\setlength{\itemsep}{-1pt}
    \item  The {\bf information-theoretic lower bounds} are presented in \autoref{sec:info_theory} and can be read independently:
     \begin{list}{{\tiny $\bullet$}}{\leftmargin=1.5em}
		\setlength{\itemsep}{-1pt}
     \item The lower bound proof for the smooth costs (\autoref{thm:lb_sto_info})  and is presented in \autoref{pf:lb_sto_info}. 
     \item The lower bound proof for the stochastic global oracle case (\autoref{thm:lb_online}) is similar to that of \autoref{thm:lb_sto_info} and is presented in \autoref{pf:lb_online}.
     \end{list}
    \item  The {\bf \sfa{} lower bounds} are quite technical and often rely on delicate constructions. Hence, we present the proofs as follows:
    \begin{list}{{\tiny $\bullet$}}{\leftmargin=1.5em}
		\setlength{\itemsep}{-1pt}
 
	    \item In \autoref{sec:foi_smooth}, to help readers understand the general proof strategy, we first present the lower bound proof for smooth costs with the stochastic inexact oracle (\autoref{thm:lb_sto_first}). 
	    
	    \item In \autoref{sec:lb_smooth},  we present the proofs of other \sfa{} lower bounds for smooth costs.

	    \item In \autoref{sec:lb_nonsmooth}, we present the proofs of \sfa{} lower bounds for nonsmooth costs. The lower bound constructions are more complicated than the case of smooth costs.
	    Hence, to help reader understand the crux of arguments, we first present the proof of a (weaker) lower bound against gradient descent (as opposed to the entire class of \sfa{}) in \autoref{sec:warmup_gd}.
	
    \end{list}
    \item The {\bf upper bounds} are presented in \autoref{sec:ub_smooth} (smooth costs), \autoref{sec:ub_nonsmooth} (non-smooth costs), and \autoref{sec:ub_finite} (finite-sum setting).
\end{list}
\end{mdframed}

\section{Information-theoretic lower bounds}
\label{sec:info_theory}
\subsection{ Information-theoretic lower bound for stochastic inexact gradient model}\label{pf:lb_sto_info} 
 
 We state  and prove the information-theoretic lower bound for smooth costs.

\begin{restatable}{theorem}{lbstoinfo} {\bf (Information-theoretic Lower Bound)} \label{thm:lb_sto_info} Assume that $\eps<1/200$ and  $\eps \lesssim\delta^2$. Then there exists a family of smooth cost functions $\{\Fth{\theta}:\R \to \R\}$ parameterized by $\theta\in[-1,1]$ with the following property. Suppose  $A$ is any algorithm that for each $\theta\in[-1,1]$ uses at most $T$ queries to stochastic inexact gradient oracle and  outputs $\vx_T^{\theta}$ that  is   $\eps$-accurate, i.e., $\E \Fth{\theta}(\vx_T^{\theta}) - \inf_\vx \Fth{\theta}(\vx)\leq \eps$.
 Then, there exists $\thebad\in[-1,1]$ such that  $(\eps, \delta)$-deviation (\autoref{def:deviation} (a)) is lower bounded by $\Omega(\frac{\delta^2}{T\eps^2})$.
 \end{restatable}

\begin{proof} 
Let   $\eps>0$ be a fixed small constant.
We consider the following   family of cost functions $\{\Fth{\theta}:\R \to \R\}$ parametrized by $\theta\in[-1,1]$:  
\begin{align} \label{cost:family_2}
    \Fth{\theta}(x) =  \begin{cases}
     100\eps \cdot (x-\theta)^2, & \text{for}~x\in[-1,1],\\
     100\eps \cdot (1-\theta)^2 + 200\eps (1-\theta)(x -1) & \text{for}~x>1\\
     100\eps \cdot (-1-\theta)^2 + 200\eps (-1-\theta)(x + 1) & \text{for}~x<-1.
    \end{cases}
\end{align} 
Note that $\Fth{\theta}$ is $\epss$-smooth for each $\theta\in[-1,1]$.

 We consider the following stochastic first order oracle for each $\Fth{\theta}$.
For a queried point $x$, the oracle outputs $g_{\theta}(x)$ defined as  
\begin{align} \label{sto_grad}
    g_{\theta}(x) &= \begin{cases} 
 \frac{1}{200\eps}\nabla \Fth{\theta}(x)  +z \quad \text{for }z\sim N(0,\frac{\delta^2}{2\cdot \epss}), &\text{w.p.}~\epss\,,\\
     0, &\text{w.p.}~1-\epss\,.
    \end{cases}
\end{align}
Since we define $\Fth{\theta}$ as a linear extension outside $[-1,1]$, below we may assume that all gradient queries are made within $[-1,1]$.

Assuming that all queries are made in $[-1,1]$, we can rewrite \eqref{sto_grad} simply as
\begin{align} \label{sto_grad:2}
   g_{\theta}(x) = (x-\theta +z) \cdot  \sam\,,
\end{align} 
where $\sam \sim \bern(\epss)$.
Let us verify that this is a valid stochastic first order oracle. First, it is clear that 
 \begin{align*}
     \E[g_{\theta}(x)]  = \epss\cdot  (x-\theta+0) +(1- \epss) \cdot 0 = \epss\cdot (x-\theta)= \nabla \Fth{\theta} (x).
 \end{align*}
 Next, for the variance, note that
 \begin{align*}
     \E \left[g_\theta(x) -\nabla \Fth{\theta}(x)\right]^2 &= \E\left[(x-\theta +z) \cdot  \sam -\nabla \Fth{\theta}(x)\right]^2\\ 
     &= \epss \cdot \E\left[ x-\theta +z -\epss \cdot (x-\theta) \right]^2 + (1-\epss)  \cdot  \left[  -\epss \cdot (x-\theta) \right]^2 \\
     &= \epss (1-\epss)^2\cdot (x-\theta)^2 +(\epss)^2(1-\epss)\cdot (x-\theta)^2 + \epss \cdot \E[z^2]\\
     &= \epss (1-\epss) \cdot (x-\theta)^2 + \delta^2/2.
 \end{align*}
 Hence, the variance is always upper bounded by $\delta^2/2+\epss(1-\epss) 2^2\leq \delta^2$, as long as $\eps\lesssim \delta^2$.

We now prove the theorem.
From the fact that the output  $x^{\theta}_T$  is $\eps$-accurate, we have
 \begin{align*} 
     \forall \theta\in[-1,1],\quad 100\eps \cdot \E\left[x^{\theta}_T-\theta\right]^2\leq  \eps.  
 \end{align*}
 Using Jensen's inequality, we know $\left|\E[x^{\theta}_T] - \theta\right|^2 = \left|\E[x^{\theta}_T-\theta]\right|^2 \leq \E[x^{\theta}_T-\theta]^2 $ which implies the following condition: 
  \begin{align}\label{condition:1}
    \forall \theta\in[-1,1],\quad \left|\E[x^{\theta}_T]-\theta\right| \leq 0.1.  
 \end{align}
 This condition says if we regard $A$ as an estimator of $\theta$ for each $\theta$, the bias  is less than equal to $0.1$. 
 For the following argument, we hence change our perspective and regard $A$ as an  estimator of $\theta$ based on $T$ inexact gradient queries rather than an optimization algorithm.
 
 Let us fix $\theta\in [-1,1]$. 
 We first  that we may assume that all  gradient queries are made at the point $0$.
 Indeed, from the expression for the stochastic oracle \eqref{sto_grad:2}, we know
 \begin{align*}
     g_\theta(x) \overset{d}{=} g_\theta(0) + x\cdot \sam.
 \end{align*}
 In particular, this implies that one can reconstruct a gradient query at any point $x$ based on a gradient query at the point $0$.
 Hence, without loss of generality, we may assume that all gradient queries are made at $x=0$.
 
 Hence, we can regard $A$ as an estimator of $\theta$   based on  $T$ independent  measurements  $y_1,\dots,y_T$ of form
 \begin{align} 
 y_i = (-\theta+z_i)\cdot \sam_i,\quad i=1,\dots, T\,,
 \end{align}
 where $z_i\sim N(0,\frac{\delta^2}{\epss})$ and $\sam_i \sim \bern(\epss)$.
 Now with this new perspective in mind, we can lower bound the variance of the estimator $\var(\vx_T^\theta)$ using the Cramer-Rao lower bound.
 To that end, we first calculate the fisher information of the measurement distribution.

 Recall that each measurement is of form
\begin{align*}
    y=(-\theta +z) \cdot  \sam\,,
\end{align*}
where $z\sim N(0,\frac{\delta^2}{2\cdot \epss})$ and $\sam\sim \bern(\epss)$. Then, the log likelihood is 
 \begin{align*}
     \ell(\theta; y) &=  \ln\left[\frac{1 }{\sqrt{2\pi\cdot \frac{\delta^2}{2\cdot \epss}}} \exp\left(-\frac{2\cdot \epss}{\delta^2}\cdot \frac{(y+\theta)^2}{2}\right)\right] \ind{\sam=1}   +  \delta_{[y=0]} \cdot  \ind{\sam=0}\,.
 \end{align*}
 Taking derivatives of the log likelihood, we get
 \begin{align*}
     \nabla_\theta \ell(\theta; y)  &= - \frac{2\cdot \epss}{\delta^2}\cdot (\theta+y)\cdot \ind{\sam=1} \,,\\
     \nabla^2_\theta \ell(\theta ; y)  &= -  \frac{2\cdot \epss}{\delta^2}\cdot \ind{\sam=1}\,.
 \end{align*}
 Hence,  the Fisher information is equal to 
 \begin{align*}
     I(\theta) =\cov_\theta \nabla_\theta \ell(\theta ; y) &= -\E_{\theta} \nabla^2_\theta \ell(\theta; y) =\frac{2\cdot \epss}{\delta^2}\cdot
    \Pr[\sam =1]= \frac{2\cdot (\epss)^2}{\delta^2}\,.
 \end{align*}
 Hence, the fisher information $I_T(\theta)$ for the $T$ independent measurements is equal to $T   \frac{2\cdot (\epss)^2}{\delta^2}\cdot$.

  Let us recall the Cram\'er-Rao lower bound for biased estimators. (see, e.g., \cite[(3)]{eldar2004minimum}).
   \begin{proposition} \label{prop:cramerrao}
   Let $b(\theta):= \E[\hat{\theta}]-\theta$ be the bias of an estimator $\hat{\theta}$.
   Then, the following bounds hold:
   \begin{align*}
       \var(\hat{\theta}) \geq \frac{[1+b'(\theta)]^2}{I(\theta)}.
   \end{align*}
   \end{proposition}

 From \eqref{condition:1}, it must be that $b'(\thebad) \geq -\frac{1}{2}$ for some $\thebad\in[-1,1]$.
To see this, suppose to the contrary that $b'(\theta) < -\frac{1}{2}$ for all $\theta\in[-1,1]$.
Then, it must be that
\begin{align*}
    b(1) \leq -\frac{1}{2}\cdot 2 + b(-1) \overset{\eqref{condition:1}}{\leq } -1 +0.1 = -0.9,
\end{align*}
which is a contradiction since \eqref{condition:1} ensures that $ b(1)\geq -0.1$.
Thus, Proposition~\ref{prop:cramerrao} gives
\begin{align*}
    \var(\vx_T^{\thebad}) \geq \frac{(1+b'(\thebad))^2}{I_T(\thebad)} \geq \Omega\left( \frac{\delta^2}{   T\eps^2}\right).
\end{align*}
This concludes the proof of the lower bound.
 \end{proof}

 \subsection{Proof of lower bound (stochastic global oracle)} \label{pf:lb_online}

Recall \autoref{thm:lb_online} from the main text.
\lbonline*
\begin{proof}
The construction and argument are analogous to the proof of \autoref{thm:lb_sto_info} (\autoref{pf:lb_sto_info}).
 Fix $\eps>0$ and consider the following   family of cost functions $\{\Fth{\theta}:\R \to \R\}$ parametrized by  unknown ${\theta\in[1,2]}$:  
\begin{align} \label{cost:family_1}
    \Fth{\theta}(x) = \begin{cases}  
     200\eps \cdot \left\{\frac{1}{2} x^2 -\theta x \right\}, & \text{for}~x\in[1,2]\\
    \text{linear extension} & \text{for}~x\not\in[1,2].
    \end{cases}
\end{align} 
Note that $\Fth{\theta}$ is $\epss$-smooth for all $\theta\in[-1,1]$ and the minimum is achieved at $x=\theta$.
Below, let us fix a ground truth parameter $\theta$ and let $F(x) = \Fth{\theta}(x)$.

Now define $f(x,\xi)$ as follows: with probability $200\eps$,
\begin{align}
    f(x,\xi)  =  \begin{cases}  
     \frac{1}{200\eps} \Fth{}(x) +zx \quad \text{for }z\sim N(0,\frac{\delta^2}{2\cdot \epss}), &\text{if}~ x\in[1,2]\,,\\
    \frac{1}{200\eps} \Fth{}(x) + 2z. &\text{if}~x>2\,,\\ \frac{1}{200\eps} \Fth{}(x) + z&\text{if}~x<1\,, \end{cases}
\end{align}
and $f(x,\xi)=0$ with probability $1-200\eps$.
Then clearly we have $\E_{\xi} f(x,\xi)  =F(x)$.

We first check that this construction satisfies \autoref{def:global}.
It is sufficient to check the condition for $x\in[1,2]$ since outside the interval the cost is defined as the linear extension.
For $x\in[1,2]$, we have 
\begin{align} 
    \nabla f(x,\xi) &= \begin{cases} 
     x-\theta +z \quad \text{for }z\sim N(0,\frac{\delta^2}{2\cdot \epss}), &\text{w.p.}~\epss\,,\\
     0, &\text{w.p.}~1-\epss\,.
    \end{cases}
\end{align}
 This is precisely the expression \eqref{sto_grad} in the proof of \autoref{thm:lb_sto_info} (\autoref{pf:lb_sto_info}), and hence, 
 this clearly satisfies \autoref{def:global}.

 Now the key fact of the proof is that one can reconstruct the complete specification of the function $f(\cdot,\xi)$ based on a gradient query, provided that it is nonzero. 
 This is because if nonzero, the gradient query at $x$ is equal to $(x-\theta +z)$. This reveals $(\theta-z)$, from which one can reconstruct the complete characterization $f(x,\xi) = \frac{1}{2}x^2-(\theta-z)x$. 
 
 Hence, the information revealed by a single  query to the stochastic global oracle is as good as that revealed by a single query to the stochastic global oracle.
 Thus, the setting is reduced to that of \autoref{thm:lb_sto_info} (\autoref{pf:lb_sto_info}), and using the same argument, the proof follows. 
\end{proof}

  \section{Proof of lower bounds (smooth costs)} 
\label{sec:lb_smooth}

We first introduce a helper function we will use throughout the proofs of \sfa{} lower bounds in the remaining sections.
 
 \paragraph{Helper function for smooth costs lower bounds.} 
 We will frequently use the following function for the \sfa{} lower bounds for smooth cost.
 Let $\euF:\R\to \R$ be an one-dimensional function defined as
\begin{align} \label{hh}
\euF(x):=  \begin{cases}
 x^2 &\text{if } x\in[0,1],\\
  2x-1 &\text{if } x\geq 1,\\
    0 &\text{if } x\leq 0.
 \end{cases}
\end{align}
For reader's convenience, we illustrate the helper function $\euF$ in \autoref{fig:helper} below.
 \begin{figure}[H]
 \centering
  \begin{tikzpicture}[scale=0.7]
\begin{axis}[
    xmin=-1, xmax=3,
    ymin=-0.5, ymax=5,
    ]
\addplot [
    domain=0:1, 
    samples=100, 
    color=red, 
    line width=3pt
]
{x^2}; 
\addplot [
    domain=1:3, 
    samples=100, 
    color=red, 
    line width=3pt
]
{2*x-1};
\addplot [
    domain=-1:0, 
    samples=100, 
    color=red, 
    line width=3pt
]
{0};
\addplot[
    color=black,
    line width=2pt,dashed
    ]
    coordinates {
    (1,-0.5)(1,1) 
    }; 
\addplot[
    color=black,
    line width=2pt,dashed
    ]
    coordinates {
    (0,-0.5)(0,0) 
    }; 
\addplot[
    color=black,
    line width=2pt,dashed
    ]
    coordinates {
    (0.5,-0.5)(0.5,0.25) 
    }; 

\addplot[
    color=black,
    line width=2pt,dashed
    ]
    coordinates {
    (-1 ,0.25)(0.5,0.25) 
    };    
\draw (axis cs:0.5,-0.5) node[anchor=south east]{$0.5$};

\draw (axis cs:-1,0.25) node[anchor=south west]{$0.25$};

\end{axis}
\end{tikzpicture}
\caption{Illustration of the helper function $\euF$ for the smooth costs lower bounds.}
\label{fig:helper}
\end{figure}
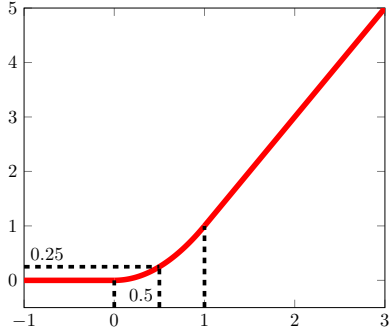

\subsection{Stochastic inexact gradient model}
\label{sec:foi_smooth}
The proofs of FOI lower bounds are quite technical and rely on delicate constructions, and to illustrate our general proof strategy, we first present a proof that is  relatively simpler, yet captures the essence of the later complicated constructions.
More specifically, in this section, we will prove the following FOI lower bounds for smooth costs against stochastic inexact oracle.

\begin{restatable}{theorem}{lbstofirst}{{\bf (Lower Bound)}}
 \label{thm:lb_sto_first}
Let $\eps>0$, and $T$ be a given number of iterations. There exists an $O(1)$-smooth convex function $f:\R^{O(T)}\to \R$ and a stochastic inexact gradient oracle   such that any \sfa{} algorithm $\alg$   that starts at $\vx_0=\vzero$ has its $(\eps,\delta)$-deviation  lower bounded by $\Omega(\frac{\delta^2}{T\eps^2})$.
\end{restatable}

\begin{proof}
Now consider the cost $f: \vx=(\dumx,y)\in \R^T \times \R \to \R$ defined as
    \begin{align}\label{construct:lb_sto_first}
        f(\vx) = 4\eps\cdot  \euF(y+1)\,,
    \end{align}
where $\euF$ is defined in \eqref{hh}.
Here note that $\dumx \in \R^T$ is dummy coordinates which do not appear in the cost $f$.
Next, we define the stochastic inexact oracle for $t=0,1,\dots, T-1$ as
    \begin{align} \label{def:lb_sto_first_oracle}
        g(\vx_t) =\nabla f(\vx_t) + \delta r_t \ve_{1+t} \quad\text{where}~r_t\sim   \unif\{\pm 1\}\,,
    \end{align}
where $\ve_j$ is the $j$-th coordinate vector.
Then, clearly this stochastic gradient fulfills the definition of stochastic inexact gradient oracle.
Here the stochastic gradient noises are designed such that during the $t$-th iteration, the noise is added to the coordinate $\dumx[1+t]$.
In other words, the noises will be added to each coordinate of $\dum$ incrementally.
    
From here one, let us write iterates $\vx_t = (\dumx_t,y_t)$.
Note first that for $\eps$-accuracy, it must be that $|y_T- y_0|\geq 1/2$; otherwise $f(\vx_T) > 4\eps \cdot \euF(0.5)  = \eps$; see \autoref{fig:helper}.
Based on the definition of \sfa{} (see \eqref{exp:foi}), let us write 
\begin{align*}
    \vx_T = \vx_0 -\sum_{t=0}^{T-1} \lambda_t^{(T)} g(\vx_t)\,.
\end{align*}
Then from the construction \eqref{construct:lb_sto_first}, we know that for any $\vx$, we know $\frac{\partial f}{\partial y}(\vx)\in [0,8\eps]$. On the other hand, as we discussed, we need  $|y_T- y_0|\geq 1/2$. 
Hence, in order for iterates to move far enough from the starting point, the coefficients have to add up to a sufficiently large number:
\begin{align}   \label{cond:coeff}
    \sum_{t=0}^{T-1} |\lambda_t^{(T)}| \geq \frac{1}{16\eps}\,,
\end{align}  
since otherwise, $|y_T- y_0|< 8\eps \cdot \frac{1}{16\eps} =1/2$. 

Now we will make use of the condition \eqref{cond:coeff} to show that there is a large deviation in the coordinates $\dumx$.
More specifically, let us lower bound $\E\norm{\dumx_T -\E[\dumx_T]}^2$. From the construction of inexact oracle \eqref{def:lb_sto_first_oracle}, it follows that 
    \begin{align*}
    \E\norm{\dumx_T - \E[\dumx_T]}^2 & = \E\norm{\sum_{t=0}^{T-1}\lambda_t^{(T)} \delta r_t\cdot \ve_{1+t} }^2 = \sum_{t=0}^{T-1} (\lambda_t^{(T)})^2 \delta^2\E[r_t^2]  \\
    &=\sum_{t=0}^{T-1}(\lambda_t^{(T)})^2 \delta^2 \overset{(a)}{\geq} \delta^2 \cdot \frac{1}{T}\cdot \left( \sum_{t=0}^{T-1} |\lambda_t^{(T)}|\right)^2 \gtrsim \frac{\delta^2}{T\eps^2},
    \end{align*} 
    where ($a$) is due to Cauchy-Schwarz inequality.
This concludes the proof. \end{proof}

  \subsection{Non-stochastic inexact gradient model} \label{pf:lb_det_smooth}

 \begin{restatable}{theorem}{lbdetsmooth}{\bf (Lower Bound)}\label{thm:lb_det_smooth}
Let $\eps>0$ be a small constant, and $T$ be a given number of iterations. There exists a $O(1)$-smooth convex function $f:\R^{2}\to \R$ with a non-stochastic inexact gradient model  such that for any \sfa{} algorithm $\alg$  that starts at $\vx_0=\vzero$ has the $(\eps,\delta)$-deviation  lower bounded by $\Omega(\frac{\delta^2}{\eps^2})$.
 \end{restatable} 
\begin{proof}
With $\euF$ defined as \eqref{hh}, this time we consider a simpler construction: the cost $f: \vx=(x,y)\in \R \times \R \to \R$ is defined as
    \begin{align}\label{construct:lb_det_first}
        f(x,y) = 4\eps\cdot  \euF(y+1).
    \end{align}

Let us write the iterate  as $\vx_t = (x_t,y_t)$. Note that for $\eps$-accuracy, it must be that $y_T\geq 1/2$; otherwise $f(y_T)  >\eps$.
This means that in order to achieve $\eps$-suboptimality, the $y$-component of the iterate has to move at least constant distance away from the starting point.

Now consider the following non-stochastic inexact oracle
  \begin{align} \label{def:lb_det_1}
    g(\vx_{t}) = \nabla f(\vx_{t}) + \delta\cdot  \frac{\partial }{\partial y} \euF(y_{t})\cdot \ve_{1},
    \end{align}
    where $\ve_1$ is the first coordinate vector.
    Note that this is a valid oracle because $0\leq \frac{\partial }{\partial y} \euF(y)\in [0, 8\eps] \in [0,1]$ for all $y$.
    Then from the construction of the inexact gradient oracle \eqref{def:lb_det_1}, it follows that 
    \begin{align} \label{rel:xy}
        \frac{\delta }{4\eps}y_T = x_T
    \end{align}
    Letting $\vx_T^{\sf exact}= (x_t^{\sf exact},y_t^{\sf exact})$ be the iterate with exact gradients (without the gradient noises), since we know $x_T^{\sf exact}=0$,   \eqref{rel:xy} implies  
    \begin{align*}
     \norm{\vx_T - \vx_T^{\sf exact}}^2   &\geq  \left|x_T - x_T^{\sf exact}\right|^2 = x_T^2 \gtrsim \frac{\delta^2}{\eps^2}.
    \end{align*}
    This is the desired lower bound.
\end{proof}
 \subsection{Inexact initialization model}
 \label{lb:init_smooth}

   \begin{restatable}{theorem}{lbinitsmooth}{\bf (Lower Bound)}\label{thm:lb_init_smooth}
Let $\eps>0$ be a small constant, and $T$ be a given number of iterations. There exists a $O(1)$-smooth convex function $f:\R^{O(T)}\to \R$  such that for any \sfa{} algorithm $\alg$    the $(\eps,\delta)$-deviation lower bounded by $\Omega(\delta^2)$ w.r.t. the reference point  $\vx_0^{\sf ref}=\vzero$.
 \end{restatable} 
  \begin{proof}
 Consider $f:\R^2 \to \R$ defined as $f(x,y)=(y-1)^2$. Choose $x_0^{\sf ref}=0$ and the inexact initialization to be $x_0=(\delta,0)$.
Then, any first order algorithm only updates the second coordinate, which implies that after $T$ iterations, we still have $\norm{x_T-x_T^{\sf ref}}\geq \delta^2$.
\end{proof}

  \subsection{Stochastic inexact gradient model (strongly convex costs)} \label{lb:smooth_sto_str}
 
 \begin{restatable}{theorem}{lbstosmoothstr}{\bf (Lower Bound)}\label{thm:lb_sto_smooth_str}
Let $\eps>0$  and $T$ be a given number of iterations. There exists a $O(1)$-smooth  and $\mu$-strongly convex function $f:\R^{O(T)}\to \R$ with a stochastic inexact gradient model such that any \sfa{} algorithm $\alg$   that starts at $\vx_0=\vzero $ has its $(\eps,\delta)$-deviation   lower bounded by $\Omega(\frac{\delta^2}{T\mu^2} \wedge  \frac{\eps}{\mu})$.
 \end{restatable} 

\begin{proof} 
For   $\vx=(\dumx,y)$ where $\dumx \in \R^{T}$ and $y\in\R$, consider the cost
\begin{align*}
f(x,y) := y+ \frac{\mu}{2} y^2 +\frac{\mu}{2}\norm{\vx}^2\,.
\end{align*}
We consider the initialization $\vx_0 =(0,0,0,\dots,0)$.

Next, we define the stochastic inexact oracle for $t=0,1,\dots, T-1$ as
    \begin{align} \label{def:smooth_str_oracle}
       \bm{g}(\vx_t) =\nabla f(\vx_t) + \delta r_t \cdot \ve_{1+t} \quad\text{for}~r_t\sim   \unif\{\pm 1\}\,,
    \end{align}
    where $\ve_j$ is the $j$-th coordinate vector.
    Also, throughout the proof we use the notation: 
\begin{align*}
     \dumg_t:=\bm{g}(\dumx_t,y_t)[1,2,\dots,T] \quad \text{and} \quad  g^y_t:= \bm{g}(\dumx_t,y_t)[T+1]\,.
\end{align*}
    
\paragraph{Warm-up: the case of simplified gradient noises.}
For a moment, we assume that the inexact gradient oracle is non-stochastic with
 \begin{align*} 
        \bm{g}(\vx_t) =\nabla f(\vx_t) + \delta \cdot \ve_{1+t}\,.
    \end{align*}
We consider this case first to build the key intuition of the proof.
The first prove the following result that is crucial for the proof.

    \begin{lemma} \label{lem:smooth_str}
For each $t=0,1,\dots, T-1$,  the output of a \sfa{} algorithm satisfies 
\begin{align} \label{sum:smooth_str}
 \delta\cdot  y_{t}    = \sum_{i=1}^T x_{t}[i]  \quad \text{and} \quad \delta\cdot g^y_t = \sum_{i=1}^T \dumg_t[i]\,,
   \quad \text{for each }t=0,1,2,\dots, T.
 \end{align}
\end{lemma}
\begin{proof}
We prove by induction on $t$.
The statement trivially holds for $t=0$.
Assume that the conclusion holds for some $t$.
We will first show that
\begin{align}  \label{ind:smooth_1}
 \delta \cdot  y_{t+1}    &= \sum_{i=1}^T x_{t+1}[i] \,.
\end{align}
By the definition of FOI (see \eqref{exp:foi}), we have  
\begin{align*}
  \delta \cdot y_{t+1} &=   -\delta \cdot\sum_{j=0}^{t} \la^{(t+1)}_j g^y_j \\
     &=   -\sum_{j=0}^{t} \la^{(t+1)}_j \left( \sum_{i=1}^T \dumg_j[i] \right)
     =  -  \sum_{i=1}^T\sum_{j=0}^{t} \la^{(t+1)}_j  \dumg_j[i] \\
     &= \sum_{i=1}^T x_{t+1}[i].
\end{align*}
This completes the proof of \eqref{ind:smooth_1}. Next, we will show that 
\begin{align} \label{ind:smooth_2}
\delta \cdot g^y_{t+1}  = \sum_{i=1}^{T} \dumg_{t+1}[i].
\end{align}
This follows because
\begin{align*}  
 \delta \cdot g^y_{t+1} &= \delta (1+ \mu y_{t+1}) = \delta + \mu \left( \sum_{i=1}^{T} \dumx_{t+1}[i] \right)  \\
&\overset{(a)}{=} \dumg_{t+1}[t+2]+ \sum_{i=1}^{t+1} \dumg_{t+1}[i]  =  \sum_{i=1}^T \dumg_{t+1}[i],
\end{align*}
 where ($a$) uses the fact that $\dumx_{t+1}[i]=0$ for all $i>t+1$.
\end{proof}
By Lemma~\ref{lem:smooth_str}, it holds that
\begin{align} \label{eq:final}
    \delta\cdot  y_{T}    &= \sum_{i=1}^T \dumx_{T}[i] 
\end{align}
On the other hand, in order to achieve $\eps$-suboptimality, we need $y_T^2 \gtrsim \frac{1}{\mu^2}$.
Moreover, for  $\eps$-suboptimality, we also need $\sum_{i=1}^T \dumx_T[i]^2\lesssim \frac{\eps}{\mu}$.
Therefore, letting $\vx_T^{\sf exact}=(
(\dumx_T)^{\sf exact},
y_T^{\sf exact})$ be the iterate with exact gradients, since $(\dumx_T)^{\sf exact}=0$,  the condition \eqref{eq:final} implies the following:
whenever $\sum_{i=1}^T \dumx_T[i]^2\lesssim \frac{\eps}{\mu}$,
    \begin{align*}
     \norm{\vx_T - \vx_T^{\sf exact}}^2   &\geq  \norm{\dumx_T - (\dumx_T)^{\sf exact}}^2 = \sum_{i=1}^T \dumx_T[i]^2 \\
     &\geq \frac{1}{T }  \left( \sum_{i=1}^T \dumx_T[i] \right)^2  =\frac{1}{T }\cdot   \delta^2\cdot \left( y_T \right)^2 \gtrsim \frac{\delta^2}{T\mu^2}.
    \end{align*}
    This is precisely equal to the desired lower bound.

\paragraph{Actual proof for the stochastic noise case.}
    Now coming back to the stochastic inexact gradient \eqref{def:smooth_str_oracle}, one can prove the following analog of Lemma~\ref{lem:smooth_str}:
    \begin{align} \label{sum:smooth_str_sto}
\begin{split}
 \delta\cdot  |y_{t}|  \leq  \sum_{i=1}^T |\dumx_{t}[i]|  \quad \text{and}\quad
\delta\cdot |g^y_t| \leq  \sum_{i=1}^T |\dumg_t[i]|\,,
   \end{split}
   \quad \text{for each }t=0,1,2,\dots, T.
 \end{align}
Here we note that the construction ensures that $|\dumx_T[i]|$'s are deterministic quantities (because regardless of whether $r_t=\pm 1$ the absolute value is the same), and that is why we do not write the expectation operators next to them.

The above result holds for the following reason,
when $r_t =+1$ for all $t$, the stochastic inexact gradient reduces to the non-stochastic inexact gradient, in which case the equality holds in \eqref{sum:smooth_str_sto} without absolute values.
With $r_t =\pm 1$, one can no longer argue this.
On the other hand, one can apply triangle inequalities to obtain \eqref{sum:smooth_str_sto}.

Again,  in order to achieve $\eps$-suboptimality, we need $y_T^2 \gtrsim \frac{1}{\mu^2}$ and $\sum_{i=1}^T \dumx_T[i]^2\lesssim \frac{\eps}{\mu}$.
Therefore, whenever $\sum_{i=1}^T \dumx_T[i]^2\lesssim \frac{\eps}{\mu}$, we have 
    \begin{align*}
    \E\norm{\vx_T - \E[\vx_T]}^2 &\geq    \sum_{i=1}^{T} |\dumx_T[i]|^2 \geq \frac{1}{T} \left(\sum_{i=1}^{T} |\dumx_T[i]|\right)^2  \geq \frac{1}{T} \cdot \delta^2\cdot  |y_T|^2 \gtrsim \frac{\delta^2}{T\mu^2}.
    \end{align*}  
This completes the proof. \end{proof}

 \subsection{Non-stochastic inexact gradient model} \label{lb:smooth_det_str}

\begin{restatable}{theorem}{lbdetsmoothstr}{\bf (Lower Bound)}\label{thm:lb_det_smooth_str}
Let $\eps>0$ be a small constant, and $T$ be a given number of iterations. There exists a $O(1)$-smooth $\mu$-strongly convex function $f:\R^{2}\to \R$ with a non-stochastic inexact gradient model  such that for any \sfa{} algorithm $\alg$  that starts at $\vx_0=\vzero$  has the $(\eps,\delta)$-deviation  lower bounded by $\Omega(\frac{\delta^2}{\mu^2} \wedge \frac{\eps}{\mu})$.
 \end{restatable}

\begin{proof}
For simplicity, we assume throughout the proof that $D=1$ and for $\vx=(x,y)$ where $x,y\in \R$, consider the cost
\begin{align*}
f(x,y) := y+ \frac{\mu}{2} y^2 +\frac{\mu}{2}x^2
\end{align*} 
We consider the initialization $\vx_0 =(0,0)$.
Next, consider the following non-stochastic inexact oracle
  \begin{align} \label{def:lb_det_2}
    g(\vx_{t}) = \nabla f(\vx_{t}) + \delta \ve_{1}
    \end{align}
    where $\ve_1$ is the first coordinate vector.
Then from this construction, one can verify similarly to \autoref{lem:smooth_str} that 
\begin{align*}
    g(\vx_t)[1] = \delta \cdot  g(\vx_t)[2]\quad \text{and}\quad x_t =\delta\cdot y_t
\end{align*}
for $t=0,1,\dots, T$.

Now from the $\eps$-suboptimality, it must be that $y_T^2 \gtrsim \frac{1}{\mu^2}$ and $x_T^2 \lesssim \frac{\eps}{\mu}$.
Hence, whenever $x_T^2 \lesssim \frac{\eps}{\mu}$ holds, we have the following deviation bound since $x_T^{\sf exact }=0$:
 \begin{align*}
     \norm{\vx_T - \vx_T^{\sf exact}}^2   &\geq  |x_T - x_T^{\sf exact}|^2 = \delta^2\cdot y_t^2  \gtrsim \frac{\delta^2}{\mu^2}.
    \end{align*} 
    This completes the proof.
\end{proof}

\subsection{Inexact initialization model (strongly convex costs)}
\label{lb:smooth_init_str}

   \begin{restatable}{theorem}{lbinitsmoothstr}{\bf (Lower Bound)}\label{thm:lb_init_smooth_str}
Let $\eps>0$ be a small constant, and $T$ be a given number of iterations. There exists a $O(1)$-smooth $\mu$-strongly convex function $f:\R^{\Omega(T)}\to \R$  such that for any \sfa{} algorithm $\alg$    the $(\eps,\delta)$-deviation lower bounded by $\Omega(\exp(-\Omega( T))\delta^2 \wedge \frac{\eps}{\mu})$ w.r.t. the reference point  $\vx_0^{\sf ref}=\vzero$.
 \end{restatable} 
\begin{proof} 
We use the construction in \cite[Theorem 2.1.13]{nesterov2018lectures}. 
In particular, for simplicity, we consider the construction for the infinite dimensional Hilbert space $\ell_2$ as it simplifies the proof; in fact, a similar argument works for $\R^{\Omega(T)}$.
Let us recall the construction (we follow the presentation in \cite[Theorem 3.15]{bubeck2014convex}).
Let $A : \ell_2 \rightarrow \ell_2$ be the linear operator that corresponds to the infinite tri-diagonal matrix with $2$ on the diagonal and $-1$ on the upper and lower diagonals. 
For some constant $\kappa\geq 1$, consider the following $\mu$-strongly convex cost:
\begin{align*}
f^{\sf Nes}(x) = \frac{\mu (\kappa-1)}{8} \left(\langle Ax, x\rangle - 2 \langle \ve_1, x \rangle \right) + \frac{\mu}{2} \norm{x}^2\quad \text{and} \quad q:= \frac{\sqrt{\kappa}-1}{\sqrt{\kappa}+1} .
\end{align*}
For the zero initialization  $x_0 = (0,0,\dots)\in \ell_2$, the cost satisfies the following properties (see  the proof of  \cite[Theorem 2.1.13]{nesterov2018lectures}): 
 \begin{itemize}
 \item Output of any \sfa{} satisfies $x_t[i] = 0, \forall i \geq t$.
    \item $x^*[i]=q^i$.
     \item $\norm{x_0 -x^*}^2 = \sum_{i=1}^\infty (x^*[i])^2 =\sum_{i=1}^\infty q^{2i} =\frac{q^2}{1-q^2}$.
     \item $\norm{x_t-x^*}^2 \geq \sum_{i=k+1}^\infty q^{2i} =\frac{q^{2(t+1)}}{1-q^2}=q^{2t} \norm{x_0-x^*}^2$.
 \end{itemize}
 Now we consider the following cost function: for $\vx = (x,y)\in \ell_2\times \ell_2$
 \begin{align*}
     f(\vx) = f^{\sf Nes}(x) +f^{\sf Nes}(y) \,,
 \end{align*}
 and we consider the two initializations:
 \begin{align*}
     \vx_0^{\sf ref} &= \Big((0,0,0,\dots),(q,q^2,q^3,\dots) \Big)\\
     \vx_0 &= \Big((0,0,0,\dots),(q,q^2,q^3,\dots,q^L,0,0,\dots) \Big)
 \end{align*}
 for $L= \Omega(\log(1/\delta))$ is chosen such that  $q^{L+1}/(1-q) \leq \delta$.
 Then, it follow that $\norm{\vx_0-\vx_0^{\sf ref}} = \frac{q^{L+1}}{1-q} \leq \delta$. 
 On the other hand, it follows from the above property that
  \begin{align*}
     \norm{\vx_t-\vx_t^{\sf ref}}^2 \geq q^t\frac{q^{2L+1}}{1-q^2}  \approx q^t \cdot\delta^2.
 \end{align*}
 Hence, as long as $\norm{\vx_t-\vx^*}^2,\norm{\vx_t^{\sf ref}-\vx^*}^2\lesssim \frac{\eps}{\mu}$, the deviation lower bound follows.
  \end{proof}

\section{Proof of lower bounds (nonsmooth costs)}
\label{sec:lb_nonsmooth}

In this section, we present the proofs of lower bounds for nonsmooth costs.
The proof will be based on more complicated constructions than those for the case of smooth costs, so before we dive into the proofs, we first present some intuition behind the constructions.

\subsection{Warm-up: lower bound against GD}
\label{sec:warmup_gd}
In this section, as a warm-up, we will prove a (weaker) lower bound for a simplified setting. 
In particular, we prove the lower bound against gradient descent (GD).
Formally, in the definition of \ref{exp:foi}, we restrict that $\la^{(t)}_i \equiv \la_i$ (i.e., the coefficient is a positive number does not depend on the iterations $t$). In other words, for $\la_i$, $i=0,1,\dots, T-1$,
\begin{align}\label{exp:gd} 
     \vx_t =  \vx_0 - \sum_{i=0}^{t-1} \la_i g(\vx_i)\quad\text{for each $t=1,2,\dots, T$}.
\end{align}
Note that this is precisely GD with step sizes $\la_t$'s.
For the lower bound construction, let  $\vx = (\errx, w) \in \R^{T}\times  \R$ and  consider the cost
 \begin{align*}
     f(\errx,w) = \underbrace{\max\left\{0,~\max_{i=1,\dots, T}\left\{ \errx[i]   \right\}\right\}}_{=:G(\errx )} +\underbrace{2 \eps\cdot \max\{w+1,0\}}_{=:\ell(w)}\,.
 \end{align*}
 Since the above cost function is nonsmooth, we specify the  subgradient oracle as follows: for  both max terms above, we consider the subgradient oracle that outputs the subgradient corresponding to the first argument that achieves the maximum.
 Note that $\inf_{\vx}f(\vx)=0$.
  Consider the zero initialization, i.e., $(\errx_0,w_0)=({\bm 0},0)$, and we write the iterates as $\vx_t :=(\errx_t, w_t) \in \R^T \times \R$. 
 
     For intuition, we describe the role of each coordinate:
\begin{list}{{\tiny $\blacksquare$}}{\leftmargin=1.5em}
\setlength{\itemsep}{-1pt}
         \item The first $T$ coordinates, $\errx\in \R^T$, correspond to the part where the errors due to inexact oracle are added. 
         \item The last coordinate, $w\in \R$,  governs the overall cost; in order to achieve $\eps$-accuracy, the optimization algorithm has to decrease  coordinate $w$ by at least $1/2$.
\end{list}

  The proof proceeds by considering two different scenarios:

\paragraph{Scenario 1 (exact gradients).} Consider the case where there is no noise in the gradients, i.e., $g(\vx_t) = \nabla f(\vx_t)$ for all $t$. 
Then, since $\errx_0= \bm{0}$,  it follows that  $\nabla G(\errx_t) =\bm{0}$ for all $t$. 
Hence the algorithm will only update coordinate $w_t$.
Note that $\ell (w_0)=2\eps$, and hence in order to achieve  $\ell(w_T) \leq \eps $, it must be that $w_T\leq -1/2$. 
 On the other hand, we have
 \begin{align*}
    \frac{\partial}{\partial w}\ell(w) = 0~~\text{or}~~2\eps\quad \text{for any}~~w\in\R.
 \end{align*}
Hence, in order to achieve $w_T\leq -1/2$, it must be that 
\begin{align}\label{cond:has_to_move}
    \sum_{t=0}^{T-1} \la_t\geq \Omega(1/\eps).
\end{align}
This condition is analogous to \eqref{cond:coeff} from the lower bound proof for smooth costs (\autoref{sec:foi_smooth}).

\paragraph{Scenario 2 (inexact gradients).}  Now let us consider the case  where the gradient error during the $t$-th iteration is non-stochastic and equal to $-\delta \ve_t$, i.e.,
\begin{align*}
    g(\vx_t) = \nabla f(\vx_t) -\delta \ve_{t+1}\quad \text{for all $t=0,1,\dots,T-1$}.
\end{align*} 
Here $\ve_t$ denotes the $t$-th coordinate vector. 
Let us assume that $\delta$ is much smaller than all the step sizes $\la_t$, in particular, such that $\la_i \delta \ll \la_{i+1}$ for all $i=0,\dots ,T-2$. 
Then from GD iterations defined as \eqref{exp:gd}, one can deduce that 
\begin{align*}
    \errx_t  = (-\la_1+\la_0 \delta ,\ -\la_2+\la_1\delta ,\ \cdots ,\ -\la_{t-1} + \la_{t-2}\delta,\ +\la_t\delta,0,\dots,0)\,.
\end{align*} 
Thus, the following estimate on the deviation holds:
 \begin{align}\label{cond:dev}
     \norm{\errx_T}^2 = \norm{\sum_{t=1}^{T-1}(\la_t -\la_{t-1}\delta) \ve_t + \la_{T-1}\delta \ve_{T}}^2 = \sum_{t=1}^{T -1}(\la_t-\la_{t-1}\delta)^2  \approx  \sum_{t=1}^{T -1}\la_t ^2\,.
 \end{align} 
 
\paragraph{Combining the two scenarios.}
Thus far, we have obtained   \eqref{cond:has_to_move} and \eqref{cond:dev}
from the two different scenarios.
The condition  \eqref{cond:has_to_move} shows that in order to achieve $\eps$-suboptimality, stepsizes have to add up to a large number, more precisely,  $\sum_{t=0}^{T-1} \la_t=\Omega(1/\eps)$.
On the other hand,   \eqref{cond:dev} characterizes that the deviation is on the order of the quantity $\sum_{t=1}^{T-1} \la_t^2$.
In order to formally connect these two conditions, we make the following assumption:
\begin{align}
    \la_0 \leq O\left( \sum_{t=1}^{T-1} \la_t\right)\,.
\end{align}
Then with this assumption, one obtain the following deviation bound:
 \begin{align*}
  \norm{\errx_T}^2 \approx \sum_{t=1}^{T-1} (\la^{(T)}_t)^2 \overset{(a)}{\geq} \frac{1}{T-1}\cdot  \left(\sum_{t=1}^{T-1}\la^{(T)}_t\right)^2\gtrsim \frac{1}{T-1}\cdot  \left(\sum_{t=0}^{T-1}\la^{(T)}_t\right)^2\gtrsim  \frac{1}{T\eps^2}\,,
 \end{align*}
 where $(a)$ is due to  the Cauchy-Schwartz inequality.
 This is precisely the desired  lower bound.
 
  For the lower bound against the entire class of \sfa{},  there are some other technical challenges arising from the fact that the coefficients $\la^{(t)}_i$'s not only  depend on $t$, but also could take negative values. 
  We need a more elaborate lower bound construction, as we explain in the subsequent subsections.

\subsection{Helper function}

Before actual proofs, we introduce a helper function that we will use throughout the proofs of \sfa{} lower bounds. 
Let 
$\component:\R\to \R$ is a non-smooth convex function defined as $\component(x) := \max\{x, 0\}$ and the subgradients are defined as
    \begin{align}\label{def:component}
    \begin{split}
        \nabla_{x} \component(x) &:= \begin{cases}
      +1, &\text{if }x\geq 0,\\
       0, &\text{if }x<y.
        \end{cases} 
    \end{split}
    \end{align}
The choice of subgradient $+1$ at the origin will play a crucial role in the later proofs.
For $\vx,\vy,\vz\in \R^T$ and $v\in\R$, let $\euG: (\vx,\vy,\vz)\in \R^{3T}\to \R$  be defined as
 \begin{align} \label{def:euG}
     \euG(\vx,\vy,\vz)&:=\max\{0 ,~\euK(\vx,\vy,\vz)\}\quad \text{and}
        \\
\euK(\vx,\vy,\vz) &:=\max_{i=1,\dots, T} \Big\{  \component(\vy[i])  + \sum_{j=1}^{i-1} \frac{ \left|\vx[j]\right|}{2^{j-1}} + \frac{\vx[i]}{2^{i-1}} , ~~ \component(\vz[i])  + \sum_{j=1}^{i-1} \frac{ \left|\vx[j]\right|}{2^{j-1}} - \frac{\vx[i]}{2^{i-1}}\Big\}\,. \label{def:euK}
    \end{align}
Then $\euG$ is clearly convex, as it is the maximum of convex functions.    

We specify the  subgradients of $\euG$ as follows: for all max terms in \eqref{def:nonsmooth_lower}, we get the subgradient of the first argument that achieves the maximum.
Then $\euG$ is $O(1)$-Lipschitz: for any $\vx,\vy,\vz \in \R^T$,
    \begin{align*}
        \norm{\nabla \euG(\vx,\vy,\vz) }^2 \leq 1+ \sum_{j=1}^{T} (\frac{1}{2^{j-1}})^2   \leq 1+ \sum_{j=1}^{\infty} \frac{1}{4^{j-1}} \leq 1+ 4/3\,.
    \end{align*}

\subsection{Stochastic inexact gradient model} 
\label{pf:lb_sto_nonsmooth}

 \begin{restatable}{theorem}{lbstononsmooth}{\bf (Lower Bound)}\label{thm:lb_sto_nonsmooth}
Let $\eps>0$  and $T$ be a given number of iterations.  
There exists a $O(1)$-Lipschitz (nonsmooth) convex function $f:\R^{O(T)}\to \R$ with a stochastic inexact gradient model such that any \sfa{} algorithm $\alg$   that satisfies $|\la_0^{(T)}| \leq O\left( 
\left|\sum_{t=1}^{T-1} \la^{(T)}_t\right|\right)$ and  starts at $\vx_0=\vzero$  has its $(\eps,\delta)$-deviation   lower bounded by $\Omega(\frac{1}{T\eps^2})$.
 \end{restatable} 
 
\begin{proof}
For $\vx=(\errx,\vy,\vz,w)$ where $\errx,\vy,\vz \in\R^T$ and $w\in \R$, consider the cost
    \begin{align}\label{def:nonsmooth_lower}
     f(\errx,\vy,\vz,w) = \euG(\errx,\vy,\vz) +\underbrace{2\eps\cdot \max\{  w+1,0 \}}_{=:\ell(w)}\,,
    \end{align}
where $\euG$ is defined  in \eqref{def:euG}.  
Then, $f$ is convex since both $\euG$ and $\ell$ are convex, and $O(1)$-Lipschitz since both $\euG$ and $\ell$ are $O(1)$-Lipschitz. 

    For intuition, we describe the role of each coordinate as we did in the warm-up section (\autoref{sec:warmup_gd}):
\begin{list}{{\tiny $\blacksquare$}}{\leftmargin=1.5em}
\setlength{\itemsep}{-1pt}
         \item The first $T$ coordinates, $\errx\in \R^T$, correspond to the part where the errors due to inexact oracle are added.
         \item The next $2T$ coordinates, $\vy,\vz\in \R^T$ will contribute to large deviation when there are errors in the gradients.
         \item The last coordinate, $w\in \R$,  governs the overall cost; in order to achieve $\eps$-accuracy, the optimization algorithm has to decrease  coordinate $w$ by at least $1/2$.
\end{list}

We use the following notation throughout the proof: $\vx_t=(\errx_t,\vy_t,\vz_t,w_t)\in  (\R^T)^3\times \R$.
Consider the zero initialization $\vx_0=({\bm 0},{\bm 0},{\bm 0}, 0)$ and the  following inexact gradient error for $t=0,1,2,\dots, T-1$: 
\begin{align} \label{exp:inexact}
     g(\vx_t) = 
     \nabla f(\vx_t) + \delta  r_t \cdot  \ve_{1+t}\quad \text{for $r_t\overset{iid}{\sim}   \unif\{\pm 1\}$},
\end{align} 
where $\ve_j$ is the $j$-th coordinate vector.
The following lemma characterizes the key feature of the above construction.
\begin{remark}
Note that \autoref{lem:pattern} is the place where we use the following additional assumption that we made for the case of nonsmooth costs:
``\emph{for the case of nonsmooth costs, we additionally assume that the coefficient of the latest gradient is nonzero, i.e., $\la^{(t)}_{t-1}\neq 0$ for all $t$.}''
\end{remark}
 
\begin{lemma} \label{lem:pattern}
   Under the inexact gradient \eqref{exp:inexact}, the subgradient $\nabla \euG$ has the following properties: \vspace{-10pt}
\begin{list}{{\tiny $\blacksquare$}}{\leftmargin=1.5em}
\setlength{\itemsep}{-1pt}
\item  For each $t=1,2,\dots, T-1$, there exists $i_t \in \{1,\dots, t\}$ such that the following holds: 
\begin{align} \label{exp:pattern}
\begin{cases}
     \frac{\partial}{\partial \vy[i_t]} \euG(\errx_t,\vy_t,\vz_t) =1,~~ \frac{\partial}{\partial \vz[i_t]} \euG(\errx_t,\vy_t,\vz_t) = 0   & \text{with probability  } $1/2$, \\
       \frac{\partial}{\partial \vy[i_t]}  \euG(\errx_t,\vy_t,\vz_t) =0, ~~\frac{\partial}{\partial \vz[i_t]} \euG(\errx_t,\vy_t,\vz_t) = 1  & \text{with probability  } $1/2$.
    \end{cases}
\end{align}
Moreover, for $i\neq i_t$, $\frac{\partial}{\partial \vy[i]} \euG(\errx_t,\vy_t,\vz_t)=\frac{\partial}{\partial \vz[i]} \euG(\errx_t,\vy_t,\vz_t)=0$.
\item If $i_t \neq t$ (i.e., $i_t<t$), then $i_t=i_{t'}$ for some $t'<t$, and it holds that $ \frac{\partial}{\partial \vy[i_t]} \euG(\errx_t,\vy_t,\vz_t) = \frac{\partial}{\partial \vy[i_{t'}]} \euG(\errx_{t'},\vy_{t'},\vz_{t'}) $ and $ \frac{\partial}{\partial \vz[i_t]} \euG(\errx_t,\vy_t,\vz_t) = \frac{\partial}{\partial \vz[i_{t'}]} \euG(\errx_{t'},\vy_{t'},\vz_{t'})$.
\end{list}
\end{lemma}
\begin{proof}[Proof of \autoref{lem:pattern}] 
Let us recall  the definition of $\euG(\vx,\vy,\vz)$:
    \begin{align}  \label{recall:euG}
    \max\left\{0~~,\max_{i=1,\dots, T} \Big\{  \component(\vy[i])  + \sum_{j=1}^{i-1} \frac{ \left|\vx[j]\right|}{2^{j-1}} + \frac{\vx[i]}{2^{i-1}} , ~~ \component(\vz[i])  + \sum_{j=1}^{i-1} \frac{ \left|\vx[j]\right|}{2^{j-1}} - \frac{\vx[i]}{2^{i-1}}\Big\}\right\}\,.
    \end{align}
    From this, it is clear that there must be at most one $i\in\{1,\dots,T\}$ for which either  $\frac{\partial}{\partial \vy[i]} \euG(\errx_t,\vy_t,\vz_t)\neq 0$ or $\frac{\partial}{\partial \vz[i]} \euG(\errx_t,\vy_t,\vz_t)\neq 0$.
    This proves the ``\emph{Moreover, for $i\neq i_t$, $\frac{\partial}{\partial \vy[i]} \euG(\errx_t,\vy_t,\vz_t)=\frac{\partial}{\partial \vz[i]} \euG(\errx_t,\vy_t,\vz_t)=0$}'' part of the first bullet point.
    
Next, we prove the expression \eqref{exp:pattern}. We begin with $t=1$. Since $g(\vx_0) = \nabla f(\vx_0) + \delta  r_0 \cdot  \ve_{1}$ and $\lambda^{(1)}_0\neq 0$, it follows that $\errx_1[1]\neq 0$.
Then, we claim that  the first bullet point holds for $t=1$.
Since we know $\errx_1[1]\neq 0$ and $\errx_1[2],\dots, \errx_1[T] = 0$,  for $\vx_1$, the maximum in \eqref{recall:euG} is achieved by $i=1$.
This implies that $i_1=1$.
    Moreover, depending on the sign of $r_0$, we either have $\frac{\partial}{\partial \vy[1]}  \euG(\errx_1,\vy_1,\vz_1)=1$ or $\frac{\partial}{\partial \vz[1]}  \euG(\errx_1,\vy_1,\vz_1)=1$ with equal probability.
    Thus, \eqref{exp:pattern} holds for $t=1$.

   Next, consider $t>1$. 
   Since $g(\vx_{t-1}) = \nabla f(\vx_{t-1}) + \delta  r_{t-1} \cdot  \ve_{t}$ and $\lambda^{(t)}_{t-1}\neq 0$, it follows that $\errx_t[t]\neq 0$.
  Moreover, we know $\errx_t[t+1],\dots, \errx_t[T] = 0$.
  Hence, we have the following two scenarios:
  \begin{list}{{\tiny $\bullet$}}{\leftmargin=1.5em}
\setlength{\itemsep}{-1pt}
\item {\bf Case 1:} $\vy[i],\vz[i]\leq 0$ for all $i\in\{1,2,\dots, t-1\}$. Note that this hold---for instance---if  the coefficients FOI are all non-negative, i.e., $\la^{(t)}_i \geq 0$ of for all $i\in\{1,2,\dots, t-1\}$ (most first order optimization algorithms usually follow this). In that case, the maximum in \eqref{recall:euG} is achieved by $i=t$
This implies that $i_t = t$.  Moreover, depending on the sign of $r_{t-1}$, we either have $\frac{\partial}{\partial \vy[t]}  \euG(\errx_t,\vy_t,\vz_t)=1$ or $\frac{\partial}{\partial \vz[t]}  \euG(\errx_t,\vy_t,\vz_t)=1$ with equal probability.
    Thus, again \eqref{exp:pattern} holds for $t$.

\item {\bf Case 2:} Somehow FOI chooses to follow positive gradient directions (which is unlikely in practice) and it happens that $\vy[i]>0$ or $\vz[i]> 0$ for some $i\in\{1,2,\dots, t-1\}$. 
In such a case, the maximum in \eqref{recall:euG} could be achieved by $i\in\{1,2,\dots, t-1\}$, i.e., $i_t \in \{1,2,\dots, t-1\}$.
Then it must be that $\vy[i_t]>0$ or $\vz[i_t]>0$. This can happen only if $i_t = i_{t'}$ for some $t'<t$.
Hence, it follows that $\frac{\partial}{\partial \vy[i_t]} \euG(\errx_t,\vy_t,\vz_t) = \frac{\partial}{\partial \vy[i_{t'}]} \euG(\errx_{t'},\vy_{t'},\vz_{t'}) $ and $\frac{\partial}{\partial \vz[i_t]} \euG(\errx_t,\vy_t,\vz_t) = \frac{\partial}{\partial \vz[i_{t'}]} \euG(\errx_{t'},\vy_{t'},\vz_{t'})$.
 This proves the second bullet point in the statement.
In particular, \eqref{exp:pattern} holds for $t$. 

\end{list} 
This completes the proof of \autoref{lem:pattern}.
\end{proof}

Now we use \autoref{lem:pattern} to prove  \autoref{thm:lb_sto_nonsmooth}.
From the construction \eqref{def:nonsmooth_lower}, we know that $f(\vx_0) = 2\eps$.
In order to achieve $\eps$-accuracy, we need $w_T\leq -1/2$.
Note that $\frac{\partial}{\partial w}\ell(w_t) =  2\eps$ for all $t\geq 0$.
From the fact that $w_T\leq -\frac{1}{2}$, it follows that
\begin{align*} 
\sum_{t=0}^{T-1} \la^{(T)}_t \frac{\partial}{\partial w}\ell(v_t,w_t) \geq \frac{1}{2} \quad \Longleftrightarrow \quad    \sum_{t=0 }^{T-1} \la^{(T)}_t \geq \frac{1}{4\eps}\,.
\end{align*}
Now using the assumption that $|\la_0^{(T)}| \leq O\left( 
\left|\sum_{t=1}^{T-1} \la^{(T)}_t\right|\right)$, we obtain
\begin{align*}
    \left|\sum_{t=1}^{T-1} \la^{(T)}_t\right| \gtrsim    |\la^{(T)}_0| +  \left|\sum_{t=1}^{T-1} \la^{(T)}_t\right| \geq \left|\sum_{t=0}^{T-1} \la^{(T)}_t\right| \geq \frac{1}{4\eps}\,,
\end{align*}
which leads to the following conditon:
\begin{align}
    \label{large_coeff_nonsmooth}
 \left|\sum_{t=1}^{T-1} \la^{(T)}_t\right|\geq \Omega\left(\frac{1}{\eps}\right)
\end{align}
  This condition is analogous to \eqref{cond:has_to_move} from \autoref{sec:warmup_gd}.
  Now to better illustrate our proof strategy for the remaining part, we first consider a special case.
  
\paragraph{Warm-up: proof for the special case.} 
 As a warm-up, we first consider the special case where in the definition of \ref{exp:foi}, all the coefficients $\la^{(t)}_i$ are non-negative, i.e., 
\begin{align} 
     \vx_t =  \vx_0 - \sum_{i=0}^{t-1} \la^{(t)}_i g(\vx_i)\quad \text{for some }\lambda_i^{(t)} \geq 0,~ i=0,\dots, t-1,
\end{align}  
Then this case belongs to {\bf Case 1} in the proof of \autoref{lem:pattern}. As a consequence, $i_t=t$ in  \autoref{lem:pattern} and the following conclusion holds:
\begin{list}{{\tiny $\blacksquare$}}{\leftmargin=1.5em}
\setlength{\itemsep}{-1pt}
\item  For each $t=1,2,\dots, T-1$, 
\begin{align} \label{exp:pattern_special}
\begin{cases}
     \frac{\partial}{\partial \vy[t]} \euG(\errx_t,\vy_t,\vz_t) =1,~~ \frac{\partial}{\partial \vz[t]} \euG(\errx_t,\vy_t,\vz_t) = 0   & \text{with probability  } $1/2$, \\
       \frac{\partial}{\partial \vy[t]}  \euG(\errx_t,\vy_t,\vz_t) =0, ~~\frac{\partial}{\partial \vz[t]} \euG(\errx_t,\vy_t,\vz_t) = 1  & \text{with probability  } $1/2$.
    \end{cases}
\end{align}
Moreover, for $i\neq t$, $\frac{\partial}{\partial \vy[i]} \euG(\errx_t,\vy_t,\vz_t)=\frac{\partial}{\partial \vz[i]} \euG(\errx_t,\vy_t,\vz_t)=0$. 
\end{list}

We use \eqref{exp:pattern_special} to lower bound $\E\norm{\vx_T -\E[\vx_T]}^2$. From  \eqref{exp:pattern_special}, it  holds that for $t=1,2,\dots, T$,
\begin{align*}
   \vy_T[t] = \begin{cases}
   \lambda^{(T)}_{t}, & \text{with probability}~$1/2$,\\
    0,& \text{with probability}~$1/2$\,.
    \end{cases}
\end{align*}
  Hence, we have 
\begin{align*}
    \E\norm{\vy_T - \E[\vy_T]}^2 &\geq  \sum_{t=1}^{T-1} \E\left( \vy_T[t] -\E \vy_T[t]\right)^2  = \sum_{t=1}^{T-1} \E\left( \vy_T[t] - \frac{1}{2}\lambda^{(T)}_{t}\right)^2 = \frac{1}{4}\sum_{t=1}^{T-1}  (\lambda^{(T)}_{t})^2\\
      & \overset{(a)}{\geq}  \frac{1}{4} \cdot \frac{1}{T-1}\cdot \left( \sum_{t=1}^{T-1}  \lambda^{(T)}_{t} \right)^2 \overset{\eqref{large_coeff_nonsmooth}}{\gtrsim} \frac{1}{4} \cdot \frac{1}{T-1}\cdot \left( \frac{1}{\eps} \right)^2 \gtrsim \frac{1}{T\eps^2}.
\end{align*} 
Here $(a)$ follows from Cauchy-Schwarz inequality. 
Therefore, we get the desired deviation bound as follows:
    \begin{align*}
    \E\norm{\vx_T - \E[\vx_T]}^2 &\geq  \E\norm{\vy_T - \E[\vy_T]}^2 \gtrsim \frac{1}{T\eps^2}.
    \end{align*}

\paragraph{The proof for the general case.}
Now we consider the case of general \sfa{} where the coefficients $\la^{(t)}_i$ are not necessarily non-negative.
With $i_t$ defined in the statement of  \autoref{lem:pattern},
let 
\begin{align*}
    \Iw := \left\{i_t\in\{1,2,\dots,T\}~:~ t=1,2,\dots, T-1 \right\}\,.
\end{align*}
Then  \autoref{lem:pattern} ensures that that for each $i\in \Iw$,
\begin{align*}
   \vy_T[i] = \begin{cases}
   \sum_{\substack{t=1,2,\dots, T-1\\~\text{s.t.}~i_t=i}}\lambda^{(T)}_{t}, & \text{with probability}~$1/2$,\\
    0,& \text{with probability}~$1/2$\,.
    \end{cases}
\end{align*}
  Hence, we have 
\begin{align*}
    \E\norm{\vy_T - \E[\vy_T]}^2 &\geq  \sum_{t=1}^{T-1} \E\left( \vy_T[t] -\E \vy_T[t]\right)^2  =  \frac{1}{4}\sum_{i\in\Iw}  \left(\sum_{\substack{t=1,2,\dots, T-1\\~\text{s.t.}~i_t=i}}\lambda^{(T)}_{t}\right)^2\\
    & \overset{(a)}{\geq}  \frac{1}{4} \cdot \frac{1}{|\Iw|}\cdot \left(\sum_{i\in\Iw}  \sum_{\substack{t=1,2,\dots, T-1\\~\text{s.t.}~i_t=i}}\lambda^{(T)}_{t}\right)^2  = \frac{1}{4} \cdot \frac{1}{|\Iw|}\cdot \left(\sum_{t=1}^{T-1} \lambda^{(T)}_{t}\right)^2 \overset{\eqref{large_coeff_nonsmooth}}{\gtrsim}  \frac{1}{T\eps^2}.
\end{align*} 
Here $(a)$ follows from Cauchy-Schwarz inequality. 
Therefore, we get the desired deviation bound as follows:
    \begin{align*}
    \E\norm{\vx_T - \E[\vx_T]}^2  \geq  \E\norm{\vy_T - \E[\vy_T]}^2   \gtrsim \frac{1}{T\eps^2}.
    \end{align*} 

\end{proof} 

\subsection{Non-stochastic inexact gradient model} \label{pf:lb_det_nonsmooth}

 \begin{restatable}{theorem}{lbdetnonsmooth}\label{thm:lb_det_nonsmooth}{\bf (Lower Bound)}
Let $\eps>0$, and $T$ be a given number of iterations. There exists a $O(1)$-Lipschitz  and nonsmooth convex function $f:\R^{O(T)}\to \R$ with  a non-stochastic inexact gradient oracle such that for any \sfa{} algorithm $\alg$ that satisfies $|\la_0^{(T)}| \leq O\left( 
\left|\sum_{t=1}^{T-1} \la^{(T)}_t\right|\right)$  and starts at $\vx_0=\vzero$ has a minimum of  $(\eps,\delta)$-deviation of  $\Omega(\frac{1}{T\eps^2} +\frac{\delta^2}{\eps^2})$.
 \end{restatable} 

\begin{proof}
We consider an almost identical construction to the one considered in the proof of \autoref{thm:lb_sto_nonsmooth}, namely \eqref{def:nonsmooth_lower}.
The only difference is that  now we add an extra dummy coordinate, namely the $(3T+2)$-th coordinate, which does not appear in the cost. Let us denote this dummy coordinate by $\dum$.
Concretely, we consider the following cost: For $\vx=(\errx,\vy,\vz,w,\dum)$ where $\errx,\vy,\vz \in\R^T$ and $w,u\in \R$, consider the cost
    \begin{align}\label{def:nonsmooth_nonsto}
     f(\errx,\vy,\vz, w,\dum) = \euG(\errx,\vy,\vz) +\underbrace{2\eps\cdot \max\{ w+1,0 \}}_{=:\ell(w)}\,.
    \end{align}

We denote the iterates of FOI due to \emph{exact} gradients
 by $$\vx_t^{\sf exact}=((\errx_t)^{\sf exact},\vy_t^{\sf exact},\vz_t^{\sf exact}, w_t^{\sf exact},\dum_t^{\sf exact})\,.$$ 

Now we define the inexact gradient oracle. The noise in the inexact gradient oracle consists of two parts. For $t=0,1,2,\dots, T-1$,
\begin{align} \label{def:det_noise_nonsmooth}
        g(\vx_t) =\nabla f(\vx_t) + \frac{\delta}{\sqrt{2}}   \ve_{1+t} +  \frac{\delta}{\sqrt{2}}\ve_{3T+2} \,,
    \end{align}
    where $\ve_j$ is the $j$-th coordinate vector.
    Note that the first part of the error is similar to the error for the stochastic error case, and  the second part of the error is added to the dummy coordinate $\dum$.

Then analogous to \autoref{lem:pattern}, one can establish the following result.
We skip the proof since it is  very analogous to that of \autoref{lem:pattern}.
\begin{corollary} \label{cor:pattern_nonsto}
   Under the inexact gradient \eqref{exp:inexact}, the subgradient $\nabla \euG$ has the following properties: \vspace{-10pt}
\begin{list}{{\tiny $\blacksquare$}}{\leftmargin=1.5em}
\setlength{\itemsep}{-1pt}
\item  For each $t=1,2,\dots, T-1$, there exists $i_t \in \{1,\dots, t\}$ such that either one of the following holds:
\begin{align*} 
\begin{cases}
     \frac{\partial}{\partial \vy[i_t]} \euG(\errx_t,\vy_t,\vz_t) =1,~~ \frac{\partial}{\partial \vz[i_t]} \euG(\errx_t,\vy_t,\vz_t) = 0, \quad   \text{or} \\
       \frac{\partial}{\partial \vy[i_t]}  \euG(\errx_t,\vy_t,\vz_t) =0, ~~\frac{\partial}{\partial \vz[i_t]} \euG(\errx_t,\vy_t,\vz_t) = 1.  
    \end{cases}
\end{align*}
Moreover, for $i\neq i_t$, $\frac{\partial}{\partial \vy[i]} \euG(\errx_t,\vy_t,\vz_t)=\frac{\partial}{\partial \vz[i]} \euG(\errx_t,\vy_t,\vz_t)=0$.
\item If $i_t \neq t$ (i.e., $i_t<t$), then $i_t=i_{t'}$ for some $t'<t$, and it holds that $ \frac{\partial}{\partial \vy[i_t]} \euG(\errx_t,\vy_t,\vz_t) = \frac{\partial}{\partial \vy[i_{t'}]} \euG(\errx_{t'},\vy_{t'},\vz_{t'}) $ and $ \frac{\partial}{\partial \vz[i_t]} \euG(\errx_t,\vy_t,\vz_t) = \frac{\partial}{\partial \vz[i_{t'}]} \euG(\errx_{t'},\vy_{t'},\vz_{t'})$.
\end{list}
\end{corollary}
    
The rest of the proof is similar to that of \autoref{thm:lb_sto_nonsmooth}. 
From the construction \eqref{def:nonsmooth_nonsto}, we know that $f(\vx_0) = 2\eps$.
In order to achieve $\eps$-accuracy, one can similarly deduce that \eqref{large_coeff_nonsmooth} holds.

With $i_t$ defined in the statement of  \autoref{cor:pattern_nonsto},
let 
\begin{align*}
    \Iw := \left\{i_t\in\{1,2,\dots,T\}~:~ t\in \Tw \right\}\,.
\end{align*}  
Then one can similarly argue using \autoref{cor:pattern_nonsto} that for each $i\in \Iw$, either 
\begin{align*}
   \vy_T[i] = \sum_{\substack{t=1,2,\dots, T-1\\~\text{s.t.}~i_t=i}}\lambda^{(T)}_{t} \quad \text{or}\quad \vz_T[i] = \sum_{\substack{t=1,2,\dots, T-1\\~\text{s.t.}~i_t=i}}\lambda^{(T)}_{t}\,.
\end{align*}
This, together with the fact that $\vy_T^{\sf exact}={\bm 0}$ and $\vz_T^{\sf exact}={\bm 0}$, implies that
\begin{align*}
    &\norm{\vy_T -  \vy_T^{\sf exact}}^2+ \norm{\vz_T -  \vz_T^{\sf exact}}^2 \geq \sum_{i\in\Iw}  \left(\sum_{\substack{t=1,2,\dots, T-1\\~\text{s.t.}~i_t=i}}\lambda^{(T)}_{t}\right)^2\\
    &\quad \overset{(a)}{\geq}    \frac{1}{|\Iw|}\cdot \left(\sum_{i\in\Iw}  \sum_{\substack{t=1,2,\dots, T-1\\~\text{s.t.}~i_t=i}}\lambda^{(T)}_{t}\right)^2  = \frac{1}{4} \cdot \frac{1}{|\Iw|}\cdot \left(\sum_{t=1}^{T-1} \lambda^{(T)}_{t}\right)^2 \overset{\eqref{large_coeff_nonsmooth}}{\gtrsim}  \frac{1}{T\eps^2}.
\end{align*}
On the other hand, since $\dum_T^{\sf exact}=0$,  the deviation in the dummy coordinate can be lower bounded as follows:
\begin{align*}
    |\dum_T - \dum_T^{\sf exact}|^2 = \left(\frac{\delta}{\sqrt{2}}\cdot \sum_{t=0}^{T-1} \la^{(T)}_t \right) ^2 \gtrsim \frac{\delta^2}{\eps^2}\,.
\end{align*}
Therefore, combining all together, we get
 \begin{align*}
     \norm{\vx_T -  \vx_T^{\sf exact}}^2 &\geq  \norm{\vy_T -  \vy_T^{\sf exact}}^2+ \norm{\vz_T -  \vz_T^{\sf exact}}^2+ \norm{\dum_T - \dum_T^{\sf exact}}^2\\
     &\gtrsim \frac{1}{T\eps^2} +\frac{\delta^2}{\eps^2}\,,
\end{align*}
as desired.
\end{proof}

\subsection{Inexact initialization model}
\label{pf:lb_init_nonsmooth}

\begin{restatable}{theorem}{lbinitnonsmooth}{\bf (Lower Bound)} \label{thm:lb_init_nonsmooth}
Let $\eps>0$ be a small constant, and $T$ be a given number of iterations. There exists a $O(1)$-Lipschitz (nonsmooth) convex function $f:\R^{O(T)}\to \R$  such that for any \sfa{} algorithm $\alg$   the $(\eps,\delta)$-deviation lower bounded by $\Omega(\frac{1}{T\eps^2}+\delta^2)$ w.r.t. the reference point  $\vx_0^{\sf ref}=\vzero$.
 \end{restatable} 
\begin{proof} 
For the initialization error model, we use a simpler construction.
For $\vx=(\errx,\vy,w,\dum)\in (\R^T)^2 \times \R^2 \to \R$, consider the cost defined as
    \begin{align*} 
        f(\errx,\vy,w) = \max\{0,~ \errx[1]+\vy[1] ,~ \dots,~\errx[T]+\vy[T] \}  +2\eps\cdot \max\{w+1,0 \}\,.
    \end{align*}
Here $\dum$ is a dummy coordinate that does not appear in the cost.
For both max terms above, we consider the subgradient that outputs the gradient of the first argument that achieves the maximum.       Then, $f$ is clearly $O(1)$-Lipschitz.

    We set the reference and inexact intializations as follows: 
\begin{align}\label{nonsmooth_inexact_init}
         \vx_0^{\sf ref}=(0,0,\dots,0)\quad \text{and} \quad \vx_0 = (\underbrace{ \delta/\sqrt{2T},\cdots,\delta/\sqrt{2T}}_{\text{first $T$}}, \underbrace{ 0,\cdots,0}_{\text{second   $T$}}, 0, \delta/\sqrt{2})\,.
     \end{align}
Following the previous notations, we write the iterates as $\vx_t=(\errx_t,\vy_t, w_t)\in (\R^T)^2 \times \R$.
Moreover, we will write the iterates corresponding to the reference initialization  as
$\vx_t^{\sf ref}=((\errx_t)^{\sf ref},\vy_t^{\sf ref}, w_t^{\sf ref})$.

     From the fact that the algorithm has to achieve $\eps$-suboptimality, it follows that  $w_T \leq -1/2$. 
     On the other hand, from the inexact initialization
     \eqref{nonsmooth_inexact_init}, it holds that for all $t=0,\dots, T-1$,
     \begin{align*}
     \exists i_t\in \{1,\dots, T\}~s.t.~
     \begin{cases}
          \frac{\partial}{\partial \vy[i_t]} f(\errx_t,\vy_t,w_t) =1 &\text{and} \\
          \frac{\partial}{\partial \vy[i]} f(\errx_t,\vy_t,w_t) = 0 & \text{for }i\neq i_t.
     \end{cases}
     \end{align*}
  Hence, it holds that $\frac{1}{2\eps} w_T  =  \sum_{i=1}^T  \vy_T[i]$.
Since we know $\vy_T^{\sf ref}=(0,0,\dots, 0)$ and $\dum^{\sf ref}= 0$, the following deviation lower bound holds:
 \begin{align*}
     \norm{\vx_T -  \vx_T^{\sf ref}}^2 &\geq  \norm{\vy_T -  \vy_T^{\sf ref}}^2 +  \norm{\dum_T -  \dum_T^{\sf ref}}^2 =  \sum_{i=1}^{T}  (\vy_T[i])^2 + \dum_T^2  \\
    &\geq   \frac{1}{T}\cdot  \left(\sum_{i=1}^{T} \vy_T[i]    \right)^2  + \frac{\delta^2}{2}= \frac{1}{T}\cdot\frac{1}{4\eps^2}\cdot \left( w_T\right)^2       +\frac{\delta^2}{2}\gtrsim     \frac{1}{T\eps^2}  +\delta^2
    \end{align*} 
    This completes the proof.
\end{proof}

  \subsection{Stochastic inexact gradient model (strongly-convex costs)}
\label{pf:lb_sto_nonsmooth_str}

 \begin{restatable}{theorem}{lbstononsmoothstr}{\bf (Lower Bound)}\label{thm:lb_sto_nonsmooth_str}
Let $\eps>0$  and $T$ be a given number of iterations. There exists a $O(1)$-Lipschitz (nonsmooth) and $\mu$-strongly convex function $f:\R^{O(T)}\to \R$ with a stochastic inexact gradient model such that any \sfa{} algorithm $\alg$ that  starts at $\vx_0=\vzero$ has its $(\eps,\delta)$-deviation   lower bounded by $\Omega(\frac{1}{T\mu^2} \wedge  \frac{\eps}{\mu})$.
 \end{restatable}

\begin{proof}
 For $\vx=(\errx,\vy,\vz,  w)\in (\R^T)^3 \times \R^2$, consider the cost defined as
    \begin{align}\label{def:nonsmooth_str_lower}
    f(\errx,\vy,\vz, w) = \underbrace{\euG(\errx+\delta \ve_1,\vy,\vz) +\frac{\mu}{2} \norm{(\errx,\vy,\vz )}^2}_{=:\euGmu(\errx,\vy,\vz)} +\underbrace{ w + \frac{\mu}{2}w^2}_{=:\ellmu(w)} \,.
    \end{align}
We consider the same inexact gradient oracle defined in \eqref{exp:inexact}.
    Then similarly to
    \autoref{lem:pattern}, it holds that for each $t=1,2,\dots, T-1$, there exists $i_t \in \{1,\dots, t\}$ such that the following holds: 
\begin{align} \label{exp:pattern_recall}
\begin{cases}
     \frac{\partial}{\partial \vy[i_t]} \euG(\errx_t,\vy_t,\vz_t,v_t) =1,~~ \frac{\partial}{\partial \vz[i_t]} \euG(\errx_t,\vy_t,\vz_t,v_t) = 0   & \text{with probability  } $1/2$, \\
       \frac{\partial}{\partial \vy[i_t]}  \euG(\errx_t,\vy_t,\vz_t,v_t) =0, ~~\frac{\partial}{\partial \vz[i_t]} \euG(\errx_t,\vy_t,\vz_t,v_t) = 1  & \text{with probability  } $1/2$.
    \end{cases}
\end{align}
Moreover, for $i\neq i_t$, $\frac{\partial}{\partial \vy[i]} \euG(\errx_t,\vy_t,\vz_t,v_t)=\frac{\partial}{\partial \vz[i]} \euG(\errx_t,\vy_t,\vz_t,v_t)=0$.
Throughout the rest of the proof, we use the following notations: 
\begin{align*}
\begin{cases}
    \gy_t = (\gy_t[1],\gy_t[2],\dots, \gy_t[T]):= \nabla_\vy \euGmu(\errx_t,\vy_t,\vz_t,v_t) ,\\
    \gz_t= (\gz_t[1],\gz_t[2],\dots, \gz_t[T]):= \nabla_\vz \euGmu(\errx_t,\vy_t,\vz_t,v_t) \\ 
    g^w_t:=\frac{\partial \ellmu}{\partial w }(w_t)
\end{cases}
\end{align*}
We prove the following crucial result for the proof.

\begin{lemma} \label{lem:conserve}
 For each $t=0,1,2,\dots, T$,  the output of a \sfa{} algorithm satisfies the following:
\begin{align} \label{sum:preserve_str}
\begin{split}
   w_t   = \sum_{i=1}^T \vy_t[i] + \vz_t[i] \quad \text{and} \quad  g^w_t  = g^v_t +  \sum_{i=1}^T \gy_t[i] +\gz_t[i]\,.
   \end{split} 
 \end{align}
\end{lemma}
\begin{proof}
We prove by induction on $t$.
We first prove the statement for $t=0$. 
Recall the definition $\euGmu(\errx,\vy,\vz):=\euG(\errx+\delta \ve_1,\vy,\vz) +\frac{\mu}{2} \norm{(\errx,\vy,\vz )}^2$. Since the first coordinate of $\errx$ is $\delta$, it follows that $g^w_0 = 1$,   $\gy_0=\ve_1$ and $\gz_0=\bm{0}$. Hence, the statement holds for $t=0$.

Assume that the conclusion holds for some $t$.
We will first show that $w_{t+1}    = \sum_{i=1}^T \vy_{t+1}[i] + \vz_{t+1}[i]$. 
Using the definition of \ref{exp:foi} together with the inductive hypothesis, we have  
\begin{align*}
    w_{t+1} &=   -\sum_{j=0}^{t} \la^{(t+1)}_j g^w_j  =  - \sum_{j=0}^{t} \la^{(t+1)}_j \left( \sum_{i=1}^T \gy_j[i] +\gz_j[i]\right)\\
    &=  -\sum_{i=1}^T\sum_{j=0}^{t} \la^{(t+1)}_j \left( \gy_j[i] +\gz_j[i]\right) =  \sum_{i=1}^T \vy_{t+1}[i] +\vz_{t+1}[i].
\end{align*}
Next, we show that  $g^w_{t+1}  =   \sum_{i=1}^T \gy_{t+1}[i] +\gz_{t+1}[i]$.  Using the conclusion we just proved, we obtain 
\begin{align*}  
g^w_{t+1} &= 1+ \mu w_{t+1} = 1+ \mu \left( \sum_{i=1}^T \vy_{t+1}[i]  + \vz_{t+1}[i]\right) =\sum_{i=1}^T \left(  \gy_{t+1}[i] +\gz_{t+1}[i]\right) \,,
\end{align*}
where in the last equality, we used the  fact that  $\nabla \euGmu(\errx_t,\vy_t,\vz_t)$ is zero except for a single coordinate that is equal to $1$.
\end{proof}

Note that in order for $f(\vx_T)$ to achieve $\eps$-accuracy, it must be that 
\begin{align*}
    w_T  \in \left[-\frac{1}{2\mu} -O(\sqrt{\eps}), -\frac{1}{2\mu} +O(\sqrt{\eps})\right]\,.
\end{align*}
Note that by symmetry,  $\vy_T[i]+\vz_T[i]$ is a deterministic quantity.
\autoref{lem:pattern} ensures that 
\begin{align*}
    \vy_T[i] = \begin{cases}
    \vy_T[i]+\vz_T[i] & \text{with probability}~$1/2$,\\
    0& \text{with probability}~$1/2$,
    \end{cases}\quad \text{for all}~i=2,\dots,T. 
\end{align*}
Hence, the   deviation $\E\norm{\vx_T - \E[\vx_T]}^2$ is again lower bounded by $\frac{1}{4} \sum_{i=2}^T   (\vy_T[i]+\vz_T[i])^2$.
Since either $\vy_T[i]$ or $\vz_T[i]$ has to be zero, it follows that $(\vy_T[i]+\vz_T[i])^2=(\vy_T[i])^2+ (\vz_T[i])^2
$.
From the $\eps$-suboptimality of $\vx_T$, it also holds that $\sum_{i=2}^T (\vy_T[i])^2+ (\vz_T[i])^2 \leq \frac{2}{\mu} \eps$.
Hence either $\sum_{i=2}^T (\vy_T[i])^2+ (\vz_T[i])^2 =\Omega(\frac{\eps}{\mu})$ (in which case the deviation is lower bounded by $\Omega(\frac{\eps}{\mu})$), or we use \autoref{lem:conserve} to conclude that 
\begin{align*}
     \sum_{i=2}^T  (\vy_T[i]+\vz_T[i])^2 &\overset{(a)}{\geq}    \frac{1}{T-1}\cdot \left( \sum_{i=2}^T  \vy_T[i]+\vz_T[i]  \right)^2 \\
     &=    \frac{1}{T-1}\cdot \left(w_T-(\vy_T[1]+\vz_T[1])\right)^2 \gtrsim \frac{1}{T\mu^2}\,,
    \end{align*}  
where $(a)$ follows form the Cauchy-Schwartz inequality, and the last inequality is due to the fact that $(\vy_T[1]+\vz_T[1])^2 \lesssim \eps$. This completes the proof.
\end{proof}

 \subsection{Non-stochastic inexact gradient model (strongly-convex costs)}
\label{pf:lb_det_nonsmooth_str}

 \begin{restatable}{theorem}{lbdetnonsmoothstr}\label{thm:lb_det_nonsmooth_str} {\bf (Lower Bound)}
Let $\eps>0$  and $T$ be a given number of iterations. There exists a $O(1)$-Lipschitz (nonsmooth) and $\mu$-strongly convex function $f:\R^{O(T)}\to \R$ with a non-stochastic inexact gradient model such that any \sfa{} algorithm $\alg$   that starts at $\vx_0=\vzero$ has its $(\eps,\delta)$-deviation   lower bounded by $\Omega((\frac{1}{T\mu^2} +\frac{\delta^2}{\mu^2})\wedge \frac{\eps}{\mu} )$.
 \end{restatable} 

\begin{proof}
We consider the same construction as the  one considered in the proof of \autoref{thm:lb_sto_nonsmooth_str}.
The only difference is that  now we add an extra dummy coordinate $\dum$ (the $(3T+2)$-th coordinate), and add $\frac{\mu}{2}u^2$ to the overall cost (so that the overall cost is still $\mu$-strongly convex).  
Following the previous convention, we will write the iterate as $\vx_t=(\errx_t,\vy_t,\vz_t,w_t,\dum_t)\in (\R^T)^3 \times \R^3$. 
We define the inexact gradient as:
\begin{align*} 
        g(\vx_t) =         \nabla f(\vx_t) + \frac{\delta}{\sqrt{2}}   \ve_{1+t}+  \frac{\delta}{\sqrt{2}}\ve_{3T+3}\,,
    \end{align*}
where $\ve_j$ is the $j$-th coordinate vector.
Then following the same argument as \autoref{lem:conserve}, it holds that
\begin{align} \label{exp:str_dev}  
    w_T   =  \sum_{i=1}^T  \vy_T[i] + \vz_T[i] \quad \text{and}\quad  \delta/\sqrt{2}\cdot w_T =   \dum_T\,.
\end{align} 
From the fact that $\vx_T$ achieves $\eps$-suboptimality, it must be that 
\begin{align*}
   w_T  \in \left[-\frac{1}{2\mu} - O(\sqrt{\eps}), -\frac{1}{2\mu} + O(\sqrt{\eps})\right]\,.
\end{align*}
Now the rest of the proof follows similarly to that of \autoref{thm:lb_sto_nonsmooth_str}.
Using the facts $\vy_T^{\sf exact}=0$, $\vz_T^{\sf exact}=0$ and $\dum_T^{\sf exact}=0$, we have the following
lower bound on the deviation:
 \begin{align*}
     \norm{\vx_T -  \vx_T^{\sf exact}}^2 &\geq  \norm{\vy_T -  \vy_T^{\sf exact}}^2+ \norm{\vz_T -  \vz_T^{\sf exact}}^2+  \norm{\dum_T - \dum_T^{\sf exact}}^2 \\
    &\geq   \dum_T^2+\sum_{i=2}^T [(\vy_T[i])^2 +(\vz_T[i])^2 ]  
\end{align*}
 From the $\eps$-suboptimality of $\vx_T$, it must be that $\dum_T^2 +\sum_{i=2}^T (\vy_T[i])^2+ (\vz_T[i])^2 \leq \frac{2}{\mu} \eps$.
Thus, either $\dum_T^2 +\sum_{i=2}^T (\vy_T[i])^2+ (\vz_T[i])^2 =\Omega(\frac{\eps}{\mu})$ (in which case the deviation is lower bounded by $\Omega(\frac{\eps}{\mu})$), or we use the fact that $(\sum_{i=2}^{T} \vy_T[i] +\vz_T[i])^2 \gtrsim \frac{1}{\mu^2}$ to conclude that 
\begin{align*}
   &\dum_T^2+\sum_{i=2}^T [(\vy_T[i])^2 +(\vz_T[i])^2 ]   \geq \frac{\delta^2}{2}w_T^2 +  \frac{1}{2(T-1)}  \left(\sum_{i=1}^{T} \vy_T[i] +\vz_T[i]  \right)^2   \\ 
    & \quad\gtrsim \delta^2w_T^2 +  \frac{1}{T}  w_T^2 \gtrsim    \frac{\delta^2}{\mu^2}+\frac{1}{T\mu^2}  \,. 
    \end{align*} 
    This completes the proof.

\end{proof}

  \subsection{Inexact initialization model (strongly-convex costs)}
\label{pf:lb_init_nonsmooth_str}
 
\begin{restatable}{theorem}{lbinitnonsmoothstr}{\bf (Lower Bound)} \label{thm:lb_init_nonsmooth_str}
Let $\eps>0$  and $T$ be a given number of iterations. There exists a $O(1)$-Lipschitz (nonsmooth) and $\mu$-strongly convex function $f:\R^{O(T)}\to \R$ with an inexact initialization model such that any \sfa{} algorithm $\alg$   that starts at $\vx_0=\vzero$ has its $(\eps,\delta)$-deviation   lower bounded by $\Omega(\frac{1}{T\mu^2}\wedge \frac{\eps}{\mu})$.
 \end{restatable} 

\begin{proof}
For the initialization error model, we use a simpler construction.
For $\vx=(\errx,\vy,w)\in (\R^T)^2 \times \R \to \R$, consider the cost defined as
    \begin{align*} 
        f(\errx,\vy,w) = \max\{0,~ \errx[1]+\vy[1] ,~ \dots,~\errx[T]+\vy[T] \}  + w + \frac{\mu}{2}\norm{(\errx,\vy,w)}^2\,.
    \end{align*}
    Then, clearly $f$ is clearly $O(1)$-Lipschitz.     We set the reference and inexact intializations as follows: \begin{align}
         \vx_0^{\sf ref}=(0,0,\dots,0)\quad \text{and} \quad \vx_0 = (\underbrace{ \delta,\cdots,\delta}_{\text{first $T$}}, \underbrace{ 0,\cdots,0}_{\text{second   $T$}}, 0)\,.
     \end{align}
Following the previous proofs, we will write $\vx_t=(\errx_t,\vy_t, w_t)\in (\R^T)^2 \times \R$ and the iterates corresponding to the reference initialization  as
$\vx_t^{\sf ref}=((\errx_t)^{\sf ref},\vy_t^{\sf ref}, w_t^{\sf ref})$.
From the fact that the algorithm has to achieve $\eps$-suboptimality, it follows that 
     \begin{align*}
         w_T\in \left[-\frac{1}{2\mu} - O(\sqrt{\eps}), -\frac{1}{2\mu} + O(\sqrt{\eps})\right].
     \end{align*}
 Then following the same argument as \autoref{lem:conserve} it holds that  $w_T  =  \sum_{i=1}^T  \vy_T[i]$.
 Using the fact that $\vy_T^{\sf ref}=\bm{0}$, we have the following deviation lower bound:
  \begin{align*}
     \norm{\vx_T -  \vx_T^{\sf ref}}^2 &\geq  \norm{\vy_T -  \vy_T^{\sf ref}}^2 =  \sum_{i=1}^{T}  (\vy_T[i])^2 
\end{align*}
 Moreover, from the $\eps$-suboptimality of $\vx_T$, it must be that $ \sum_{i=1}^T (\vy_T[i])^2  \leq \frac{2}{\mu} \eps$.
      Thus, either $ \sum_{i=1}^T (\vy_T[i])^2  = \Omega(\frac{\eps}{\mu})$ (in which case the deviation is lower bounded by $\Omega(\frac{\eps}{\mu})$),  or the following deviation lower bound holds:
 \begin{align*}
    \sum_{i=1}^{T}  (\vy_T[i])^2 \geq   \frac{1}{T}\cdot  \left(\sum_{i=1}^{T} \vy_T[i]    \right)^2 = \frac{1}{T} \cdot \left( w_T\right)^2       \gtrsim     \frac{1}{T\mu^2} \,.
    \end{align*} 
    This completes the proof.
\end{proof}

\section{Proof of upper bounds (smooth costs)}
\label{sec:ub_smooth}

\subsection{Stochastic inexact gradient model}

\label{pf:ub_smooth_sto}

 \begin{restatable}{theorem}{ubstosmooth} {\bf (Upper Bound)}
 \label{thm:ub_sto_smooth}
Let $f$ be an $O(1)$-smooth convex cost function. Let $\eps>0$, and $T$ be a given number of iterations.
Under the stochastic inexact gradient model, the $(\eps,\delta)$-deviation of standard SGD with an appropriately chosen step size is $O(\frac{\delta^2}{T\eps^2})$, provided that $T=\Omega(1/\eps^2)$.
\end{restatable}

\begin{proof}
Throughout the proof, let $L$ be the smoothness constant of $f$.
We first derive the deviation bound. 
Let $\{\vx_t\}$ be the GD iterates with stochastic inexact gradients and $\{\vy_t\}$ be the GD iterates with exact gradients.
 Assuming that $\eta_t\leq \frac{2
}{L}$, the standard convex analysis yields the following one-step deviation inequality ($\E$ denotes the conditional expectation over the randomness in $g(\vx_t)$)
\begin{align*}
& \E\norm{\vx_{t+1}-\vy_{t+1}}^2 =\E\norm{(\vx_t-\eta_t g(\vx_t)  -(\vy_t -\eta_t \nabla f(\vy_t)) }^2 \\
  & \quad=  \norm{\vx_t -\vy_t}^2 - 2\eta_t \underbrace{\E \inp{\vx_t-\vy_t}{g(\vx_t) -\nabla f(\vy_t)}}_{=\inp{\vx_t-\vy_t}{\nabla f(\vx_t) -\nabla f(\vy_t)}}+ \eta_t^2 \E\norm{g(\vx_t)- \nabla   f(\vy_t) }^2  \\ 
   & \quad=  \norm{\vx_t -\vy_t}^2 - 2\eta_t \inp{\vx_t-\vy_t}{\nabla f(\vx_t) -\nabla f(\vy_t)}+ \eta_t^2 \E \norm{g(\vx_t)- \nabla   f(\vx_t) }^2  \\
  &\qquad +\underbrace{ 2\eta_t^2 \E\inp{g(\vx_t)-\nabla f(\vx_t)}{\nabla f(\vx_t) -\nabla f(\vy_t)}}_{=0}+ \eta_t^2 \norm{\nabla f(\vx_t) -\nabla f(\vy_t)}^2  \\
    & \quad=  \norm{\vx_t -\vy_t}^2 \underbrace{- 2\eta_t   \inp{\vx_t-\vy_t}{\nabla f(\vx_t)  -\nabla f(\vy_t)}+ \eta_t^2 \norm{\nabla f(\vx_t)- \nabla f(\vy_t) }^2}_{\leq 0}  \\
  &\qquad  + \eta_t^2 \E\norm{\nabla f(\vx_t) -g(\vx_t)}^2  \\
   &\quad \leq  \norm{\vx_t -\vy_t}^2+ \eta_t^2 \delta^2\,.   
 \end{align*}     
Here the last inequality is due to the standard fact about smooth and convex function that  for any $\vx,\vy$, $\frac{1}{L}\norm{\nabla f(\vx)-\nabla f(\vy)}^2 \leq \inp{\nabla f(\vx) -\nabla f(\vy)}{\vx-\vy}$ (see, e.g., \cite[(2.1.11)]{nesterov2018lectures}), together with the fact $\eta_t\leq \frac{2}{L}$.
Hence, we have proved
 \begin{align} \label{dev:smooth}
     \E\norm{\vx_T-\vy_T}^2 \leq \delta^2 \sum_{t} \eta_t^2\,.
 \end{align}

Now for the upper bound, we consider variants of SGD. 
From the standard convergence result (see, e.g., \cite[Thm. 6.3]{bubeck2014convex}), with step size $\eta_t \equiv \frac{1}{L+\nicefrac{1}{\eta}} $ for some $\eta>0$,  
 \begin{align} \label{conv:smoothsgd}
     \E f\left(\frac{1}{T}\sum_{t=1}^T \vx_{t}\right) -f(\vx_*) \leq \frac{L\norm{\vx_0-\vx_*}^2 }{2T}   +\frac{\norm{\vx_0-\vx_*}^2 }{2\eta T}   + \frac{\eta \delta^2}{2}\,.
     \end{align}
For simplicity, let $\bar{\vx}_T:=\frac{1}{T}\sum_t\vx_t$, and 
$\bar{\vy}_T:=\frac{1}{T}\sum_t\vy_t$.
From the convexity of $\norm{\cdot}^2$, we have
 \begin{align} \label{dev:average}
     \E \norm{\bar{\vx}_T -\bar{\vy}_T}^2 \leq \frac{1}{T}\sum_t \E\norm{\vx_t-\vy_t}^2 \,.
 \end{align}
 Now let us combine above results to upper bound $(\eps,\delta)$-deviation.
     
     As a warm-up, let us first consider  SGD with $\eta =O(\nicefrac{1}{\sqrt{T}})$.
     From \eqref{conv:smoothsgd}, it follows that the convergence rate reads $\E f(\bar{\vx}_T) -f(\vx_*) \leq O(\nicefrac{1}{\sqrt{T}})$. 
     With such a choice of $\eta$, the stepsize is
     \begin{align*}
         \eta_t \equiv \frac{1}{L+\Omega(\sqrt{T})} = O(\frac{1}{\sqrt{T}}).
     \end{align*}
     Hence, for the deviation bound, using \eqref{dev:smooth} together with \eqref{dev:average}, we have 
     \begin{align*}
     \E \norm{\bar{\vx}_T -\bar{\vy}_T}^2 \leq \frac{1}{T}\sum_t \E\norm{\vx_t-\vy_t}^2 \lesssim \frac{1}{T} \sum_{t=1}^T \left[t\cdot \frac{1}{T}\cdot \delta^2\right] \lesssim \delta^2.
 \end{align*}
This shows that with  $T =\Omega(\nicefrac{1}{\eps^2})$, the $(\eps,\delta)$-deviation is $O(\delta^2)$.

In order to recover the bound in the theorem statement, we consider a mini-batch SGD.
In particular, the above calculation shows that using a mini-batch of size $b$ at each iteration, it follows that with $O(\frac{b}{\eps^2})$ gradient queries, the deviation is upper bounded by $O(\frac{\delta^2}{b})$.
This precisely corresponds to the $(\eps,\delta)$-deviation bound of $O(\frac{\delta^2}{\eps^2 T})$.

An alternative way is to let the learning rate $\eta_t=1/(\eps T)$.
It then follows from \eqref{conv:smoothsgd} that 
\begin{align*}
    \E f(\bar{\vx}_T)-f(\vx_*)\le O\left(\frac{1}{T}+\eps+\frac{\delta^2}{\eps T}\right)=O(\eps),
\end{align*}
since $T=\Omega(1/\eps^2)$.
Moreover, \eqref{dev:smooth} and \eqref{dev:average} imply
\begin{align*}
    \E \norm{\bar{\vx}_T -\bar{\vy}_T}^2 \leq \frac{1}{T}\sum_t \E\norm{\vx_t-\vy_t}^2 \leq \frac{1}{T}\sum_{t=1}^T\left[t\cdot\delta^2\cdot\frac{1}{\eps^2T^2}\right]\le\frac{\delta^2}{T\eps^2}\,,
\end{align*}
which is the desired upper bound.
  \end{proof}

\subsection{Non-stochastic gradient errors}  
\label{pf:ub_smooth_det}
 
\begin{restatable}{theorem}{ubdetsmooth}{\bf (Upper Bound)} \label{thm:ub_det_smooth}
For $L=O(1)$ and $D=O(1)$, let $f$ be an $L$-smooth convex cost function whose optimum lies in a ball of radius $D$. 
Let $\eps>0$ and $\delta>0$ are such that $\delta \leq \frac{\eps}{2LD}$. Let $T$ be a given number of iterations.
Under the non-stochastic inexact gradient model, there exists a \sfa{} algorithm whose $(\eps,\delta)$-deviation is $O(\frac{\delta^2}{\eps^2})$, provided that $T=\Omega(\frac{1}{\eps})$.
\end{restatable}  
\begin{proof}
Throughout the proof, let $L$ be the smoothness constant of $f$. We consider the projected gradient descent with step size $\eta_t = \frac{1}{L}$  onto the ball of radius $D$ that contains the optimum $\vx_*$. 
It is important to note that this algorithm is a \sfa{} because the projection onto the ball of radius $D$ is a re-scaling, and hence after the projection, the coefficients $\la^{(t)}_i$ are still positive.  

The proximal inequality (e.g., \cite[Proposition 12.26]{bauschke2011convex}) implies that
 \begin{align*}
      &\norm{\vx_{t+1} -\vx_*}^2 - \norm{\vx_t-\vx_*}^2\leq  -\norm{\vx_{t+1}-\vx_t}^2 - 2\eta_t\inp{g(\vx_t)}{\vx_{t+1}-\vx_*}
 \end{align*}
 Let $\no_t$ denote the error due to the non-stochastic inexact gradient model  at iteration $t$, i.e., $\no_t:=g(\vx_t) -\nabla f(\vx_t)$.
Then we have
\begin{align*}
    &\norm{\vx_{t+1} -\vx_*}^2 - \norm{\vx_t-\vx_*}^2\\ &\leq  -\norm{\vx_{t+1}-\vx_t}^2 - 2\eta_t\inp{\nabla  f(\vx_t) + \no_t}{\vx_{t+1}-\vx_*}\\
      &\leq  -\norm{\vx_{t+1}-\vx_t}^2 +2\eta_t \inp{\nabla f(\vx_t) }{\vx_*-\vx_{t+1}} +2\delta D\\
    &= - \underbrace{\left(\norm{\vx_{t+1}-\vx_t}^2 + 2\eta_t \inp{\nabla f(\vx_t)}{\vx_{t+1}-\vx_t}\right)}_{(a)} +\underbrace{2\eta_t \inp{\nabla f(\vx_t) }{\vx_*-\vx_{t}}}_{(b)} +2\delta D \\
    & \leq  - 2\eta_t(f(\vx_{t+1})-f(\vx_t)) + 2\eta_t(f(\vx_*)-f(\vx_t)) + 2\delta D\\
    & = 2\eta_t (f(\vx_*)- f(\vx_{t+1})) + 2\delta D ,
\end{align*}
where $(a)$ is upper bounded using the $L$-smoothness together with $\eta_t=\frac{1}{L}$ as follows:
\begin{align*}
    f(\vx_{t+1}) -f(\vx_t) &\leq \inp{\nabla f(\vx_t)}{\vx_{t+1}-\vx_t} +\frac{L}{2} \norm{\vx_{t+1}-\vx_t}^2\\
    &= \inp{\nabla f(\vx_t)}{\vx_{t+1}-\vx_t} +\frac{1}{2\eta_t} \norm{\vx_{t+1}-\vx_t}^2\,,
\end{align*}
and $(b)$ is handled using convexity.

Summing this over all $t=0,1,\dots, T-1$ gives 
\begin{align*}
    \sum_{t=0}^{T-1} \frac{2}{L} (f(\vx_{t+1}) -f(\vx_*)) \leq \norm{\vx_{0}-\vx_*}^2 + 2\delta T D,
\end{align*}
 which implies the following average-iterate guarantee:
 \begin{align*}
     f(\bar{\vx}_T) -f(\vx_*) \leq \frac{L D^2}{2T} + \delta LD.
 \end{align*}
 Thus, in order to achieve $\eps$-suboptimality, we need $\Omega(1/\eps)$ iterations, since the theorem statement assumed that $\delta \leq \frac{\eps}{2LD}$.
 Next, let us bound the deviation.
 \begin{lemma}\label{lem:contract}
Suppose that $f$ is $L$-smooth.
Let $\no_t$ and $\no_t'$ denote noises in the gradients. 
If  $\eta_t\leq \frac{2}{L}$, then the following one-step deviation inequality holds
\begin{align*}
     \norm{\vx_{t}-\eta_t (\nabla f(\vx_t)+\no_t)- (\vx'_{t}-\eta_t (\nabla f(\vx'_t)+\no'_t)) } \leq \norm{\vx_t-\vx_t'}+\eta_t\norm{\no_t} +\eta_t\norm{\no_t'}. \end{align*}
 \end{lemma}
 \begin{proof}  The proof follows from the following inequality:
 \begin{align*}
     \norm{\vx_{t+1} -\vx_{t+1}'}  &\leq  \norm{\vx_t -\vx_t' -\eta_t(g(\vx_t) -g(\vx_t')) }\\
     &\leq \norm{\vx_t -\vx_t' -\eta_t({\nabla} f(\vx_t) -{\nabla} f(\vx_t')) } +\eta_t \norm{\no_t} +\eta_T\norm{\no_t}'\\
    &\overset{(a)}{\leq}     \norm{\vx_t -\vx_t'} +\eta_t\norm{\no_t} +\eta_t\norm{\no_t}'\,,
 \end{align*}
 Here ($a$) follows from the following fact:
 \begin{align*}
    & \norm{\vx_t -\vx_t' -\eta_t({\nabla} f(\vx_t) -{\nabla} f(\vx_t')) }^2\\
     \quad &= \norm{\vx_t -\vx_t'}^2 -2\eta_t\inp{\vx_t-\vx_t'}{{\nabla} f(\vx_t) -{\nabla} f(\vx_t')} +\eta_t^2 \norm{{\nabla} f(\vx_t) -{\nabla} f(\vx_t')}^2\\
     &\overset{(b)}{\leq} \norm{\vx_t -\vx_t'}^2 -\frac{2}{L}\eta_t \norm{{\nabla} f(\vx_t) -{\nabla} f(\vx_t')}^2  +\eta_t^2 \norm{{\nabla} f(\vx_t) -{\nabla} f(\vx_t')}^2\\
     &\overset{(c)}{\leq} \norm{\vx_t-\vx_t'}^2\,,
 \end{align*}
 where $(b)$ is due to the fact that for a $L$-smooth and convex function $f$, it holds that $$\frac{1}{L}\norm{\nabla f(\vx)-\nabla f(\vy)}^2 \leq \inp{\nabla f(\vx) -\nabla f(\vy)}{\vx-\vy}\quad \text{for any $\vx,\vy$,}$$
 and $(c)$ is because $\eta_t^2 \leq \frac{2}{L}\eta_t$.
 \end{proof}
 Now given the convergence rate and the deviation inequality, we are ready to prove the desired upper bound on ($\eps,\delta$)-deviation. From the triangle inequality, we get 
 \begin{align*}
 \norm{\bar{\vx}_T- \bar{\vx}'_T}
    &= \norm{\frac{1}{T}\sum_{t=1}^T \vx_t - \frac{1}{T}\sum_{t=1}^T \vx_t'} \leq \frac{1}{T} \sum_{t=1}^T \norm{\vx_t-\vx_t'}\\
    &\leq \frac{1}{T} \sum_{t=1}^T O(t\delta) \leq O(T\delta)\,,
 \end{align*}
 where the second line follows from Lemma~\ref{lem:contract}.
 Thus, the $(\eps,\delta)$-deviation is bounded by $O(\frac{\delta^2}{\eps^2})$ using the averaged iterate with $T=\Theta(1/\eps)$. \end{proof} 
 
 One limitation of \autoref{thm:ub_det_smooth} is that it requires the optimum to lie in a bounded domain. 
 Next we show that without this requirement, the gradient descent iterate is still bounded when it first attains $\eps$-accuracy.
 
\begin{theorem} \label{lem:without_proj}
 Suppose $f$ is $L$-smooth with optimum $x^*$.
 Let $\eps,\delta>0$ be given such that $\eps\le2L\|\vx_0-\vx_*\|^2$ and $\delta\le\frac{\eps}{\|\vx_0-\vx_*\|}$.
 Consider gradient descent with a constant learning rate $\eta=\frac{1}{L}$.
 Under the non-stochastic inexact gradient model, for the first iterate $\vx_T$ with $f(\vx_T)-f(\vx_*)\le\eps$, it holds that $\|\vx_T-\vx_*\|\le2\|\vx_0-\vx_*\|$.
\end{theorem}
\begin{proof}
  Let $D_t:=\|\vx_t-\vx_*\|$.
  Following the proof of \autoref{thm:ub_det_smooth} (which still holds when the domain is unbounded), we can show that
  \begin{align}\label{eq:smooth_iter}
    D_{t+1}^2-D_t^2\le2\eta(f(\vx_*)-f(\vx_{t+1}))+2\eta\delta D_{t+1}.
  \end{align}
  Let $T$ denote the first step with $f(\vx_T)\le f(\vx_*)+\eps$.
  We claim that for all $0\le t\le T-2$,
  \begin{align}\label{eq:dist_dec}
    D_{t+1}\le D_t\le D_0.
  \end{align}
   We will prove \eqref{eq:dist_dec} by induction. 
   Given $0\le t\le T-2$, suppose $D_t\le D_0$, and note that the definition of $T$ implies $f(\vx_{t+1})>f(\vx_*)+\eps$.
   It then follows from \eqref{eq:smooth_iter} that
   \begin{align*}
     D_{t+1}^2-2\eta\delta D_{t+1}+2\eta\eps-D_t^2<0.
   \end{align*}
   Let $h(z):=z^2-2\eta\delta z+2\eta\eps-D_t^2$.
   First, 
   \begin{align*}
     h(D_t)=-2\eta\delta D_t+2\eta\eps\ge0, 
   \end{align*}
  since $\delta D_t\le\delta D_0\le\eps$ by the condition of Theorem~\ref{lem:without_proj}.
  Moreover, $D_t$ is larger than $\eta\delta$, the minimum of $h$, because if it is not true, then 
  \begin{align*}
     f(\vx_t)-f(\vx_*)\le\frac{L}{2}D_t^2\le\frac{L\eta^2\delta^2}{2}\le\frac{L\eta^2\eps^2}{2D_0^2}=\frac{\eps^2}{2LD_0^2}\le\eps,
  \end{align*}
   which contradicts the definition of $T$.
   Since $h(D_{t+1})<0$, it then follows that $D_{t+1}\le D_t\le D_0$, and in particular $D_{T-1}\le D_0$.
   Finally, note that smoothness implies
   \begin{align*}
     \|\nabla f(\vx_{T-1})\|\le LD_{T-1}\le LD_0,
   \end{align*}
   and thus 
   \begin{align*}
     D_T\le\|\vx_T-\vx_{T-1}\|+D_{T-1}\le\eta\|\nabla f(\vx_{T-1})\|+D_{T-1}\le D_0+D_{T-1}\le2D_0.
   \end{align*}
   This completes the proof
\end{proof}

 \subsection{Inexact initialization model}
 \label{ub:init_smooth}
  \begin{restatable}{theorem}{ubinitsmooth}{\bf (Upper Bound)}
 \label{thm:ub_init_smooth}
Let $f$ be an $O(1)$-Lipschitz convex cost function. Let $\eps>0$ be a small constant, and $T$ be a given number of iterations.
Then there exists a \sfa{} algorithm whose $(\eps,\delta)$-deviation  is $O(\delta^2)$, provided that $T=\Omega(1/\eps)$.
\end{restatable}
 \begin{proof}
In view of \autoref{lem:contract}, two different runs $\vx_T,\vx_T'$ of gradient descent satisfies
\begin{align*}
    \norm{\vx_T-\vx_T'} \leq \norm{\vx_{T-1}'-\vx_{T-1}'}\leq \cdots \leq \norm{\vx_0-\vx_0'}\leq 2\delta.
\end{align*}
Hence, the statement follows.
\end{proof}

In this section, we consider smooth and strongly-convex costs.

\subsection{Stochastic inexact gradient model (strongly convex costs)}
\label{ub:smooth_sto_str}
We first show an upper bound for the stochastic inexact gradient oracle.
  \begin{restatable}{theorem}{ubstosmoothstr}{\bf (Upper Bound)}
 \label{thm:ub_sto_smooth_str}
Let $f$ be an $O(1)$-smooth $\mu$-strongly convex cost function. Let $\eps>0$ be a small constant, and $T$ be a given number of iterations.
Under the stochastic inexact gradient model, there exists a \sfa{} algorithm whose $(\eps,\delta)$-deviation   is $O\left(\frac{\delta^2}{T\mu^2}\wedge  \frac{\eps}{\mu}\right)$, provided that $T=\Omega(\frac{1}{\eps\mu})$.
\end{restatable}  
\begin{proof}
    The following proof is based on the proof of \citep[Theorem 6.3]{bubeck2014convex}, but here we further make use of strong convexity.
    
    Let $C\subset\R^n$ denote the domain of $f$, and assume it is convex and closed.
    We simply run stochastic gradient descent: starting from some $\vx_0\in C$, let 
    \begin{align*}
        \vx_{t+1}:=\Pi_C\left[\vx_t-\eta_tg(\vx_t)\right]\quad\textup{where}\quad\eta_t:=\frac{1}{L+1/\lambda_t}.
    \end{align*}
    We will pick a value for each $\lambda_t$ below.
    Note that $C$ can just be $\R^n$, in which case no projection is needed, but our analysis can also handle a bounded domain.
    
    Let $\vx_*$ denote the optimal solution, and suppose $f$ is $L$-smooth.
    It follows that 
    \begin{align*}
        f(\vx_{t+1})-f(\vx_t) & \le\langle\nabla f(\vx_t),\vx_{t+1}-\vx_t\rangle+\frac{L}{2}\|\vx_{t+1}-\vx_t\|^2 \\
         & =\langle g(\vx_t),\vx_{t+1}-\vx_t\rangle+\langle\nabla f(\vx_t)-g(\vx_t),\vx_{t+1}-\vx_t\rangle+\frac{L}{2}\|\vx_{t+1}-\vx_t\|^2 \\
         & \le\langle g(\vx_t),\vx_{t+1}-\vx_t\rangle+\frac{\lambda_t}{2}\|\nabla f(\vx_t)-g(\vx_t)\|^2+\frac{L+1/\lambda_t}{2}\|\vx_{t+1}-\vx_t\|^2.
    \end{align*}
    Moreover, the projection step ensures
    \begin{align*}
        \frac{1}{L+1/\lambda_t}\langle g(\vx_t),\vx_{t+1}-\vx_*\rangle & \le\langle \vx_t-\vx_{t+1},\vx_{t+1}-\vx_*\rangle \\
         & =\frac{1}{2}\left(\|\vx_t-\vx_*\|^2-\|\vx_t-\vx_{t+1}\|^2-\|\vx_{t+1}-\vx_*\|^2\right).
    \end{align*}
    Consequently,
    \begin{align*}
        f(\vx_{t+1})-f(\vx_t)&\le\langle g(\vx_t),\vx_*-\vx_t\rangle+\frac{\lambda_t}{2}\|\nabla f(\vx_t)-g(\vx_t)\|^2\\
        &\qquad +\frac{L+1/\lambda_t}{2}\left(\|\vx_t-\vx_*\|^2-\|\vx_{t+1}-\vx_*\|^2\right).
    \end{align*}
    Taking expectation with respect to $g(\vx_t)$, we have 
    \begin{align*}
        \E[f(\vx_{t+1})]\le f(\vx_t)+\langle\nabla f(\vx_t),\vx_*-\vx_t\rangle+\frac{\lambda_t}{2}\delta^2+\frac{L+1/\lambda_t}{2}\E\left[\|\vx_t-\vx_*\|^2-\|\vx_{t+1}-\vx_*\|^2\right].
    \end{align*}
    Further invoking strong convexity, we have 
    \begin{align}\label{eq:smooth_sc_cost}
        \E[f(\vx_{t+1})]\le f(\vx_*)-\frac{\mu}{2}\|\vx_t-\vx_*\|^2+\frac{\lambda_t\delta^2}{2}+\frac{L+1/\lambda_t}{2}\E\left[\|\vx_t-\vx_*\|^2-\|\vx_{t+1}-\vx_*\|^2\right].
    \end{align}
    Pick a small enough $k$ which also satisfies $k\ge4L/\mu$, and let 
    \begin{align*}
        \lambda_t=\frac{2}{(t+k)\mu-2L}.
    \end{align*}
    It follows that $\lambda_t>0$ by construction, and also
    \begin{align*}
        L+\frac{1}{\lambda_t}=\frac{t+k}{2}\mu,\quad\textup{and}\quad L+\frac{1}{\lambda_t}-\mu=\frac{t+k-2}{2}\mu,
    \end{align*}
    and that 
    \begin{align*}
        (t+k-1)\lambda_t=\frac{2}{\mu}\cdot\frac{t+k-1}{t+k-2L/\mu}\le\frac{2}{\mu}\cdot\frac{t+k-1}{(t+k)/2}\le\frac{4}{\mu},
    \end{align*}
    since by the definition of $k$, we have $2L/\mu\le k/2\le(t+k)/2$.
    Therefore if we multiply both sides of \eqref{eq:smooth_sc_cost} by $(t+k-1)$, we get
    \begin{align*}
        (t+k-1)\E\left[f(\vx_{t+1})-f(\vx_*)\right]\le & \ \frac{(t+k-1)(t+k-2)}{4}\mu\E[\|\vx_t-\vx_*\|^2] \\
         & \ -\frac{(t+k)(t+k-1)}{4}\mu\E[\|\vx_{t+1}-\vx_*\|^2]+\frac{2\delta^2}{\mu}.
    \end{align*}
    Now taking the sum from $t=0$ to $T-1$, we have 
    \begin{align*}
        \sum_{t=0}^{T-1}(t+k-1)\E\left[f(\vx_{t+1})-f(\vx_*)\right]\le\frac{(k-1)(k-2)\mu}{4}\|\vx_0-\vx_*\|^2+\frac{2\delta^2T}{\mu}.
    \end{align*}
    Define 
    \begin{align*}
        \tilde{\vx}_T:=\sum_{t=0}^{T-1}\frac{t+k-1}{\sum_{j=0}^{T-1}(j+k-1)}\vx_{t+1}\,.
    \end{align*}
    Then, we have 
    \begin{align*}
        \E\left[f(\tilde{\vx}_T)-f(\vx_*)\right]\le O\left(\frac{(k-1)(k-2)\mu}{4T^2}\|\vx_0-\vx_*\|^2+\frac{2\delta^2}{T\mu}\right).
    \end{align*}
    Since $k=\Theta(L/\mu)$, it follows that as long as $T=\Omega(\frac{1}{\eps\mu})$, we have 
    \begin{align*}
        \E\left[f(\tilde{\vx}_T)-f(\vx_*)\right]\le\eps.
    \end{align*}
    
    Next we analyze the deviation bound. 
    Similarly to the proof of \autoref{thm:ub_sto_smooth}, let $\{\vx_t\}$ denote GD iterates with stochastic inexact gradients, and let $\{\vy_t\}$ denote GD iterates with exact gradients, we can show
    \begin{align*}
        \E\left[\|\vx_{t+1}-\vy_{t+1}\|^2\right]\le & \ \|\vx_t-\vy_t\|^2-2\eta_t\langle\vx_t-\vy_t,\nabla f(\vx_t)-\nabla f(\vy_t)\rangle+\eta_t^2\|\nabla f(\vx_t)-\nabla f(\vy_t)\|^2 \\
         & \ +\eta_t^2\E\left[\|\nabla f(\vx_t)-\nabla g(\vx_t)\|^2\right].
    \end{align*}
    Next we need the following lemma.
    \begin{lemma}\label{lem:contract_sc}
        Suppose $f$ is $L$-smooth and $\mu$-strongly convex. 
        For $\eta\le1/L$, it holds for any $\vx,\vy$ that 
        \begin{align*}
        \norm{\vx -\vy -\eta({\nabla} f(\vx) -{\nabla} f(\vy)) }^2 & \le\left(1-\eta\mu\right)\|\vx-\vy\|^2.
    \end{align*}
    \end{lemma}
    \begin{proof}
        First we have
        \begin{align*}
             & \ \norm{\vx -\vy -\eta({\nabla} f(\vx) -{\nabla} f(\vy)) }^2 \\
            = & \ \norm{\vx -\vy}^2 -2\eta\inp{\vx-\vy}{{\nabla} f(\vx) -{\nabla} f(\vy)} +\eta^2 \norm{{\nabla} f(\vx) -{\nabla} f(\vy)}^2 \\
            \le & \ \|\vx-\vy\|^2-2\eta\langle \vx-\vy,\nabla f(\vx)-\nabla f(\vy)\rangle+\eta\frac{1}{L}\cdot L\langle \vx-\vy,\nabla f(\vx)-\nabla f(\vy)\rangle \\
            = & \ \|\vx-\vy\|^2-\eta\langle \vx-\vy,\nabla f(\vx)-\nabla f(\vy)\rangle,
        \end{align*}
        where the inequality is due to smoothness and $\eta\le1/L$.
        Strong convexity then implies 
        \begin{align*}
            \norm{\vx -\vy -\eta({\nabla} f(\vx) -{\nabla} f(\vy)) }^2 \le \|\vx-\vy\|^2-\eta\mu\|\vx-\vy\|^2.
        \end{align*}
    \end{proof}
    
    Note that in the current setting, $\eta_t=1/(L+1/\lambda_t)\le1/L$, therefore we can invoke  \autoref{lem:contract_sc} and obtain
    \begin{align*}
        \E\left[\|\vx_{t+1}-\vy_{t+1}\|^2\right]\le(1-\eta_t\mu)\E\left[\|\vx_t-\vy_t\|^2\right]+\eta_t^2\delta^2,
    \end{align*}
    which further implies
    \begin{align*}
        \E\left[\|\vx_T-\vy_T\|^2\right]\le\delta^2\sum_{t=0}^{T-1}\eta_t^2\prod_{j=t+1}^{T-1}(1-\eta_j\mu),
    \end{align*}
    since $\vx_0=\vy_0$.
    Note that 
    \begin{align*}
        \eta_t=\frac{2}{(t+k)\mu},\quad\textup{and}\quad1-\eta_t\mu=\frac{t+k-2}{t+k},
    \end{align*}
    therefore 
    \begin{align*}
        \E\left[\|\vx_T-\vy_T\|^2\right] & \le\delta^2\sum_{t=0}^{T-1}\frac{4}{(t+k)^2\mu^2}\frac{(t-1+k)(t+k)}{(T-2+k)(T-1+k)} \\
         & \le\delta^2\sum_{t=0}^{T-1}\frac{4}{\mu^2}\frac{1}{(T-2+k)(T-1+k)} \\
         & \le\delta^2\sum_{t=0}^{T-1}\frac{4}{\mu^2T^2}\le\frac{4\delta^2}{T\mu^2}.
    \end{align*}
    Now define 
    \begin{align*}
        \tilde{\vy}_T:=\sum_{t=0}^{T-1}\frac{t+k-1}{\sum_{j=0}^{T-1}(j+k-1)}\vy_{t+1}.
    \end{align*}
    Since $\tilde{\vx}_T$ and $\tilde{\vy}_T$ are weighted averages of $\vx_t$ and $\vy_t$ respectively, we have
    \begin{align*}
        \E\left[\|\tilde{\vx}_T-\tilde{\vy}_T\|^2\right]\le\frac{4\delta^2}{T\mu^2},
    \end{align*}
    Moreover, since $\tilde{\vy}_T$ is deterministic, an $O(\frac{\delta^2}{T\mu^2})$ deviation bound also follows.
\end{proof}

\subsection{Non-stochastic inexact gradient model (strongly convex costs)}
\label{ub:smooth_det_str}

Next we consider the non-stochastic inexact gradient oracle.
\begin{restatable}{theorem}{ubdetsmoothstr}{\bf (Upper Bound)} \label{thm:ub_det_smooth_str}
For $L=O(1)$ and $D=O(1)$, let $f$ be an $L$-smooth $\mu$-strongly convex cost function whose optimum lies in a ball of radius $D$. 
Let $\eps>0$ and $\delta>0$ are such that $\delta \leq \frac{\eps}{2LD}$, and let $T$ be a given number of iterations.
Under the non-stochastic inexact gradient model, there exists a \sfa{} algorithm whose $(\eps,\delta)$-deviation is $O\left(\frac{\delta^2}{\mu^2}\wedge \frac{\eps}{\mu}\right)$, provided that $T=\Omega(1/\eps)$.
\end{restatable}  

\begin{proof}
    We run projected gradient descent with a constant learning rate $\eta=1/L$.
    For the upper bound on excess error, we simply invoke \autoref{thm:ub_det_smooth}: as long as $T=\Omega(1/\eps)$, it holds that 
    \begin{align*}
        f(\bar{\vx}_T)\le\frac{1}{T}\sum_{t=1}^Tf(\vx_t)\le f(\vx_*)+\eps.
    \end{align*}
    
    To bound the deviation, we follow a similar analysis as in the proof of \autoref{lem:contract}.
    Consider two gradient descent runs $\{\vx_t\}$ and $\{\vx'_t\}$, and let 
    $g(\vx_t)=\nabla f(\vx_t)+\no_t$, and $g(\vx_t')=\nabla f(\vx_t')+\no_t'$.
    First we have
    \begin{align*}
        \norm{\vx_{t+1} -\vx_{t+1}'}  &\leq  \norm{\vx_t -\vx_t' -\eta_t(g(\vx_t) -g(\vx_t')) }\\
         &\leq \norm{\vx_t -\vx_t' -\eta_t({\nabla} f(\vx_t) -{\nabla} f(\vx_t')) } +\eta_t \norm{\no_t} +\eta_t\norm{\no_t'} \\
         & \le \norm{\vx_t -\vx_t' -\eta_t({\nabla} f(\vx_t) -{\nabla} f(\vx_t')) } + \frac{2\delta}{L}.
    \end{align*}
    Moreover, \autoref{lem:contract_sc} implies 
    \begin{align*}
        \norm{\vx_t -\vx_t' -\eta_t({\nabla} f(\vx_t) -{\nabla} f(\vx_t')) }\le\sqrt{1-\frac{\mu}{L}}\|\vx_t-\vx'_t\|.
    \end{align*}
    Therefore
    \begin{align*}
        \|\vx_{t+1}-\vx'_{t+1}\|\le\sqrt{1-\frac{\mu}{L}}\|\vx_t-\vx'_t\|+\frac{2\delta}{L},
    \end{align*}
    and for all $t\ge1$,
    \begin{align*}
        \|\vx_t-\vx'_t\| & \le\left(1-\frac{\mu}{L}\right)^{t/2}\|\vx_0-\vx'_0\|+\frac{2\delta}{L}\cdot\frac{1}{1-\sqrt{1-\mu/L}} \\
         & =\left(1-\frac{\mu}{L}\right)^{t/2}\|\vx_0-\vx'_0\|+\frac{2\delta}{L}\cdot\frac{1+\sqrt{1-\mu/L}}{\mu/L} \\
         & \le\left(1-\frac{\mu}{L}\right)^{t/2}\|\vx_0-\vx'_0\|+\frac{4\delta}{\mu}.
    \end{align*}
    Finally,
    \begin{align*}
        \|\bar{\vx}_T-\bar{\vx}'_T\|&\le\frac{1}{T}\sum_{t=1}^T\|\vx_t-\vx'_t\|\le\frac{4\delta}{\mu}+\frac{1}{T}\underbrace{\|\vx_0-\vx'_0\|}_{=0}\frac{1}{1-\sqrt{1-\mu/L}} =O\left(\frac{\delta}{\mu} \right).
    \end{align*}
    This completes the proof.
\end{proof}

\subsection{Inexact initialization model (strongly convex costs)}
\label{ub:smooth_init_str}

\begin{restatable}{theorem}{ubinitsmoothstr}{\bf (Upper Bound)}
 \label{thm:ub_init_smooth_str}
Let $f$ be an $L$-smooth $\mu$-strongly convex cost function. Let $\eps>0$ be a small constant, and $T$ be a given number of iterations.
Then there exists a \sfa{} algorithm whose $(\eps,\delta)$-deviation  is $O(\exp(-\mu T/L)\delta^2 \wedge \frac{\eps}{\mu})$.
\end{restatable}
 
\begin{proof}
    Let $\vx_0$, $\vx_0'$ denote two initial iterates.
     \autoref{lem:contract_sc} implies
    \begin{align*}
        \|\vx_{t+1}-\vx'_{t+1}\|^2=\|\vx_t-\vx'_t-\eta_t(\nabla f(\vx_t)-\nabla f(\vx'_t))\|^2\le\left(1-\frac{\mu}{L}\right)\|\vx_t-\vx'_t\|^2,
    \end{align*}
    and
    \begin{align*}
        \|\vx_T-\vx'_T\|^2\le\left(1-\frac{\mu}{L}\right)^T\|\vx_0-\vx'_0\|^2\le e^{-\mu T/L}\|\vx_0-\vx'_0\|^2.
    \end{align*}
    This completes the proof.
\end{proof}

\section{Proof of upper bounds (nonsmooth costs)}
\label{sec:ub_nonsmooth}
\subsection{Stochastic inexact gradient model}
\label{pf:ub_nonsmooth_sto}

 \begin{restatable}{theorem}{ubstononsmooth}{\bf (Upper Bound)}
 \label{thm:ub_sto_nonsmooth}
Let $f$ be an $O(1)$-Lipschitz convex cost function. Let $\eps>0$ be a small constant, and $T$ be a given number of iterations.
Under the stochastic inexact gradient model, there exists a \sfa{} algorithm whose $(\eps,\delta)$-deviation   is $O(\frac{1}{T\eps^2})$, provided that $T=\Omega(1/\eps^2)$.
\end{restatable}

\begin{proof} Assume now that $f$ is $G$-Lipschitz but otherwise nonsmooth ($G=O(1)$). Let $\{\vx_t\}$ be the GD iterates with stochastic inexact gradients and $\{\vy_t\}$ be the GD iterates with exact gradients.
Then the one-step deviation bound can be derived as follows ($\E$ denotes the conditional expectation over the randomness in $g(\vx_t)$):
\begin{align*}
& \E\norm{\vx_{t+1}-\vy_{t+1}}^2 =\E\norm{(\vx_t-\eta_t g(\vx_t)  -(\vy_t -\eta_t \nabla f(\vy_t)) }^2 \\
     & \quad=  \norm{\vx_t -\vy_t}^2 - 2\eta_t \underbrace{\inp{\vx_t-\vy_t}{\nabla f(\vx_t) -\nabla f(\vy_t)}}_{\leq 0 ~~(\because~\text{convexity})}   \\
  & \qquad   + \eta_t^2 \E\norm{g(\vx_t)-\nabla f(\vx_t) }^2  + \eta_t^2 \norm{\nabla f(\vx_t)- \nabla f(\vy_t) }^2 \\
   &\quad \leq  \norm{\vx_t -\vy_t}^2+ \eta_t^2 (4G^2+\delta^2)\,.
 \end{align*}    
Since we consider the regime $\delta^2 \lesssim 1$, the one-step bound leads to the following deviation inequality: 
 \begin{align} \label{dev:nonsmooth}
     \E\norm{\vx_T-\vy_T}^2 \lesssim  \sum_{t} \eta_t^2. 
 \end{align}
 Note that the above deviation bound is worse than the smooth case deviation bound \eqref{dev:smooth} which reads  $\E\norm{\vx_T-\vy_T}^2 \leq \delta^2 \sum_{t} \eta_t^2$.

  For the algorithm, we  again consider SGD. Invoking the standard convergence guarantee of SGD for nonsmooth costs (see, e.g., \cite[Thm. 6.1]{bubeck2014convex}), with step size $\eta_t \equiv \eta$ for some $\eta>0$, we have the following convergence rate:
 \begin{align} \label{conv:nonsmoothsgd}
     \E f\left(\frac{1}{T}\sum_{t=1}^T \vx_{t}\right) -f(\vx_*) \leq \frac{\norm{\vx_0-\vx_*}^2 }{2\eta T}   + \frac{\eta G^2}{2}\,.
     \end{align}  
     From \eqref{conv:nonsmoothsgd}, it follows that with $\eta =O(\nicefrac{1}{\sqrt{T}})$, the convergence rate reads $\E f(\bar{\vx}_T) -f(\vx_*) \leq O(\nicefrac{1}{\sqrt{T}})$. With such a choice of $\eta$,  the deviation can be bounded using \eqref{dev:smooth} together with \eqref{dev:average}, 
     \begin{align*}
     \E \norm{\bar{\vx}_T -\bar{\vy}_T}^2 \leq \frac{1}{T}\sum_t \E\norm{\vx_t-\vy_t}^2 \lesssim \frac{1}{T} \sum_{t=1}^T \left[t\cdot \frac{1}{T}\cdot G^2\right] \lesssim G^2.
 \end{align*}
 In fact, by choosing $\eta=\frac{1}{\eps T}$ (since $T=\Omega(1/\eps^2)$, it must be that $\eta = O(\eps)$), it follows that the $(\eps,\delta)$-deviation is upper bounded by
 \begin{align*}
        O\left(\frac{1}{T} \sum_{t=1}^T \left[t\cdot \frac{1}{T^2\eps^2}\cdot G^2\right] \right) \lesssim O(\frac{G^2}{T\eps^2})\lesssim O(\frac{1}{T\eps^2}).
 \end{align*} 
 This completes the proof.
 \end{proof} 
 
\subsection{Non-stochastic inexact gradient model}
\label{pf:ub_nonsmooth_det}

We first prove a deviation bound.
\begin{lemma} \label{lem:dev_det}
Suppose that $f$ is convex and $G$-Lipschitz. Let $\{\vy_t\}$ be the iterates of (projected) GD with stepsize $\eta_t$ with exact gradients and $\{x _t\}$ be the iterates of (projected) GD with the same stepsize with inexact gradients with noise $\{\no_t\}$. Assuming that $\norm{\no_t}\leq \delta$ for each $t$, we have 
 \begin{align} \label{ineq:dev_det}
    \norm{\vx_T-\vy_T} \leq \sqrt{3(2G^2+\delta^2) \cdot \sum_{t=0}^{T-1} \eta_t^2 } +2 \delta \sum_{t=0}^{T-1} \eta_t \,.
\end{align}
\end{lemma}
\begin{proof}
The proof is analogous to \cite[Lemma 3.1]{bassily2020stability}.
First, note that
\begin{align*}
    \norm{\vy_{t+1}-\vx_{t+1}}^2 &\overset{(a)}{\leq} \norm{\vy_{t}-\eta_t \nabla f(\vy_t)-(\vx_{t} -\eta_t (\nabla f(\vx_t)+\no_t))}^2\\
    &=\norm{\vx_t-\vy_t}^2 - 2\eta_t \inp{\vy_t-\vx_t}{\nabla f(\vy_t)-\nabla f (\vx_t)-\no_t }\\
    &\qquad+\eta_t^2 \norm{\nabla f(\vy_t) -
    \nabla f(\vx_t)-\no_t}^2\\
    &\overset{(b)}{\leq} \norm{\vx_t-\vy_t}^2 + 2\eta_t \inp{\vx_t-\vy_t}{\no_t } +\eta_t^2 \norm{\nabla f(\vy_t) -\nabla f(\vx_t)-\no_t}^2\\
    &\overset{(c)}{\leq} \norm{\vx_t-\vy_t}^2 +2\delta \eta_t \norm{\vx_t-\vy_t}  +3\eta_t^2 (2G^2 +\delta^2)\,,
\end{align*}
where $(a)$ is due to the non-expansiveness of the projection step, $(b)$ is due to convexity, and $(c)$ is due to the inequality $\norm{v_1+v_2+v_3}\leq 3\norm{v_1}^2 +3\norm{v_2}^2 +3\norm{v_3}^2$.
Denoting $d_t := \norm{\vx_t-\vy_t}$, we obtain
\begin{align} \label{ineq:dev_sq}
    d_{T}^2 \leq 2 \delta \sum_{t=0}^{T-1} \eta_t   d_t + 3 (2G^2+\delta^2)  \sum_{t=0}^{T-1} \eta_t^2.
\end{align}
We now prove \eqref{ineq:dev_det} by induction. 
If $d_T\leq \max_{t=0,\dots,T-1}d_{t}$, then the conclusion follows from the induction hypothesis.
Hence we may assume that $d_T> \max_{t=0,\dots,T-1}d_{t}$. Then the following inequality follows from \eqref{ineq:dev_sq}:
\begin{align*} 
    d_{T}^2 &\leq 2 \delta  \sum_{t=0}^{T-1} \eta_t  d_t + 3 (2G^2+\delta^2)  \sum_{t=0}^{T-1} \eta_t^2\\
    &\leq 2 \delta  d_T\cdot  \sum_{t=0}^{T-1} \eta_t   + 3 (2G^2+\delta^2)  \sum_{t=0}^{T-1} \eta_t^2\,.
\end{align*}
Solving this, we obtain the desired conclusion \eqref{ineq:dev_det}.
\end{proof}

\begin{restatable}{theorem}{ubdetnonsmooth}{\bf (Upper Bound)}
 \label{thm:ub_det_nonsmooth}
For $G=O(1)$ and $D=O(1)$, let $f$ be an $G$-Lipschitz convex cost function whose optimum lies in a ball of radius $D$. 
Let $\eps>0$ and $\delta>0$ are such that $\delta \leq \frac{\eps}{2D}$. Let $\eps>0$ be a small constant, and $T$ be a given number of iterations.
Under the non-stochastic inexact gradient model, there exists a \sfa{} algorithm whose $(\eps,\delta)$-deviation  is $O(\frac{1}{T\eps^2} +\frac{\delta^2}{\eps^2})$, provided that $T=\Omega(1/\eps^2)$.
\end{restatable}   
 \begin{proof}
 We consider the projected gradient descent with constant stepsize $\eta$ onto the ball of radius $D$ that contains the optimum $\vx_*$. 
 
 Let $\vy_t$ denote the iterate before projection. 
 Let $\no_t$ denote the error due to the non-stochastic inexact gradient model at iteration $t$, i.e., $\no_t:=g(\vx_t) -\nabla f(\vx_t)$.
 Then, we have
 \begin{align*}  \frac{1}{2}\norm{\vx_{t+1}-\vx_*}^2 -\frac{1}{2}\norm{\vx_t-\vx_*}^2  &\leq  \frac{1}{2}\norm{\vy_{t+1}-\vx_*}^2 -\frac{1}{2}\norm{\vx_t-\vx_*}^2 \\
 &\leq  -\eta\inp{\nabla f(\vx_t)+\no_t}{\vx_t-\vx_*} +\frac{1}{2}\eta^2\norm{\nabla f(\vx_t)+\no_t}^2
 \end{align*}
 Hence,
  \begin{align*} 
  & f(\vx_t) -f(\vx_*) +\frac{1}{2\eta}\norm{\vx_{t+1}-\vx_*}^2 -\frac{1}{2\eta }\norm{\vx_t-\vx_*}^2  \\
 &= f(\vx_t) -f(\vx_*) -\inp{\nabla f(\vx_t)+\no_t}{\vx_t-\vx_*} +\frac{1}{2}\eta\norm{\nabla f(\vx_t)+\no_t}^2\\
 &\leq \delta D +
 \eta (G+\delta)^2. 
 \end{align*}
After telescoping the above inequalities from $t=0,\dots, {T-1}$, we obtain the bound
 \begin{align*}
     f\left(\frac{1}{T}\sum_{t=0}^{T-1}\vx_t\right)-f(\vx_*) \lesssim \frac{D^2}{\eta T } +  \eta G^2 +\delta D.
 \end{align*}
 Thus, for $\eps$-accuracy, we need $\eta\lesssim \eps$ and $\eta T \gtrsim 1/\eps$, since the theorem statement assumed that $\delta \leq \frac{\eps}{2D}$. Hence, choosing $\eta = \Theta(\frac{1}{\eps T})$, \autoref{lem:dev_det} gives
 \begin{align*}
     \norm{\vx_t-\vx_t'}^2 &\lesssim \frac{t}{\eps^2T^2} +(t\delta  \eta)^2  = \frac{t}{\eps^2T^2} + \frac{\delta^2t^2}{\eps^2T^2}.
 \end{align*}
Hence,
\begin{align*}
      \norm{\bar{\vx}_T -\bar{\vx}'_T}^2 &\leq \frac{1}{T} \sum_t  \norm{{x}_t  -{x}'_t}^2 \\
    &\lesssim \frac{1}{T} \sum_t\left[ \frac{t}{\eps^2T^2} + \frac{\delta^2t^2}{\eps^2T^2} \right]\approx\frac{1}{\eps^2T} + \frac{\delta^2}{\eps^2}\,,
\end{align*}
as desired.
 \end{proof}

\subsection{Inexact initialization model}
\label{pf:ub_init_nonsmooth}

 \begin{restatable}{theorem}{ubinitnonsmooth}{\bf (Upper Bound)}
 \label{thm:ub_init_nonsmooth}
Let $f$ be an $O(1)$-Lipschitz convex cost function. Let $\eps>0$ be a small constant, and $T$ be a given number of iterations.
Then there exists a \sfa{} algorithm whose $(\eps,\delta)$-deviation  is $O(\frac{1}{T\eps^2}+\delta^2)$, provided that $T=\Omega(1/\eps^2)$.
\end{restatable}
 \begin{proof}

  We consider the subgradient descent with constant step size $\eta$.
 A standard convergence guarantee for GD reads (see, e.g., \cite[Theorem 3.2]{bubeck2014convex})
 \begin{align*}
     f(\bar{\vx}_T)-f(\vx_*) \lesssim \frac{\norm{\vx_0-\vx_*}^2}{\eta T } + \eta G^2 \,.
 \end{align*}
 Hence, in order to have $\eps$-suboptimality, we need to have  $\eta T \approx \frac{1}{\eps}$ and $\eta\lesssim \eps$.

   We now derive a deviation bound.
  A similar calculation to  \autoref{lem:dev_det} yields the following:
\begin{align*}
    \norm{\vy_{t+1}-\vx_{t+1}}^2 & \leq \norm{\vx_t-\vy_t}^2 - 2\eta \inp{\vy_t-\vx_t}{\nabla f(\vy_t)-\nabla f (\vx_t)} +\eta^2 \norm{\nabla f(\vy_t) -
    \nabla f(\vx_t)}^2\\
    & \leq \norm{\vx_t-\vy_t}^2  +\eta^2 \norm{\nabla f(\vy_t) -\nabla f(\vx_t)}^2 \leq  \norm{\vx_t-\vy_t}^2    +4G^2\eta^2 \,.
\end{align*}
Hence, it holds that 
\begin{align*}
    \norm{\vx_t-\vy_t}^2 \leq \norm{\vx_0-\vy_0}^2 + G^2\sum_{t=0}^{t-1} \eta^2\leq \delta^2 + G^2  t\eta^2 = \delta^2 + G^2 \frac{t}{T^2} (\eta T)^2 \approx \delta^2 + G^2 \frac{t}{T^2} 
    \cdot \frac{1}{\eps^2}  \,.
\end{align*}  
 Thus, it follows that 
\begin{align*}
    \norm{\bar{\vx}_t- \bar{\vy}_T}^2 &\leq \frac{1}{T}\sum_{t=0}^{T-1} \norm{\vx_t-\vy_t}^2 \lesssim  \frac{1}{T}\sum_{t=0}^{T-1} \left(\delta^2 + G^2 \frac{t}{T^2} \frac{1}{\eps^2} \right)   \lesssim \delta^2+ \frac{1}{\eps^2 T},
\end{align*}
as desired.
 \end{proof}

\subsection{Stochastic inexact gradient model (strongly convex costs)}
\begin{restatable}{theorem}{ubstononsmoothstr}{\bf (Upper Bound)}   
 \label{thm:ub_sto_nonsmooth_str}
Let $f$ be an $O(1)$-Lipschitz $\mu$-strongly convex cost function. Let $\eps>0$ be a small constant, and $T$ be a given number of iterations.
Under the stochastic inexact gradient model, there exists a \sfa{} algorithm whose $(\eps,\delta)$-deviation   is $O(\frac{1}{T\mu^2}\wedge  \frac{\eps}{\mu})$, provided that $T=\Omega(1/\eps)$.
\end{restatable}  
\begin{proof}
The standard convergence rate bound (e.g., \cite[Theorem 6.2]{bubeck2014convex}) implies that SGD with $\eta_t = \frac{2}{\mu(t+1)}$ satisfies
\begin{align}
   \E f\left(\sum_{t=1}^T \frac{2t}{T(T+1)}\vx_t \right) -f(\vx^*) \lesssim \frac{2G^2 }{\mu (T+1)}\,, 
\end{align}
where $G$ is the Lipschitz constant of $f$.
Hence, letting $\bar{\vx}_T:=\sum_{t=1}^T \frac{2t}{T(T+1)}\vx_t$, it follows that 
\begin{align*}
    \E \norm{\bar{\vx}_T -\vx^*}^2 \lesssim \frac{1}{T\mu^2} \wedge \frac{\eps}{\mu}\,, 
\end{align*}
where $\frac{\eps}{\mu}$ follows from the fact that $\bar{\vx}_T$ achieves $\eps$-accuracy.\end{proof}

\subsection{Non-stochastic inexact gradient model (strongly convex costs)}

   \begin{restatable}{theorem}{ubdetnonsmoothstr}{\bf (Upper Bound)}
 \label{thm:ub_det_nonsmooth_str}
For $G=O(1)$ and $D=O(1)$, let $f$ be an $G$-Lipschitz $\mu$-strongly convex cost function whose optimum lies in a ball of radius $D$. 
Let $\eps>0$ and $\delta>0$ are such that $\delta \leq \frac{\eps}{2D}$. Let $\eps>0$ be a small constant, and $T$ be a given number of iterations. 
Under the non-stochastic inexact gradient model, there exists a \sfa{} algorithm whose $(\eps,\delta)$-deviation   is $O((\frac{1}{T\mu^2} +\frac{\delta^2}{\mu^2})\wedge \frac{\eps}{\mu})$, provided that $T=\Omega(1/\eps)$.
\end{restatable}   
\begin{proof}
 We first prove the convergence rate bound.
 We run projected gradient descent with a constant learning rate $\eta_t = \frac{1}{\mu (t+1)}$.
 Then, it follows that 
 \begin{align*}
     &\frac{\mu(t+1)}{2} \norm{\vx_{t+1}-\vx^*}^2 - 
     \frac{\mu t}{2} \norm{\vx_{t}-\vx^*}^2 \\
     &\quad = \frac{\mu}{2}\norm{\vx_{t}-\vx^*}^2 + \frac{1}{2\eta_t }\left(\norm{\vx_{t+1} - \vx^*}^2 - \norm{\vx_{t} - \vx^*}^2 \right)\\
     &\quad = \frac{\mu}{2}\norm{\vx_{t}-\vx^*}^2 +  \inp{g(\vx_t) }{\vx^*-\vx_t} +\frac{\eta_t}{2} \norm{g(\vx_t)}^2\\
     &\quad \leq  \frac{\mu}{2}\norm{\vx_{t}-\vx^*}^2 +  \inp{\nabla f (\vx_t) }{\vx^*-\vx_t} +\frac{\eta_t}{2} \norm{g(\vx_t)}^2 +\delta D\,,
 \end{align*}
 where the last line follows since every iterate lies in the ball of radius $D$.
 Hence,
 \begin{align*}
    &f(\vx_t) - f(\vx^*)+ \frac{\mu(t+1)}{2} \norm{\vx_{t+1}-\vx^*}^2 - 
     \frac{\mu t}{2} \norm{\vx_{t}-\vx^*}^2\\
     &\quad \leq  \underbrace{f(\vx_t) - f(\vx^*)+ \frac{\mu}{2}\norm{\vx_{t}-\vx^*}^2 +  \inp{\nabla f (\vx_t) }{\vx^*-\vx_t}} +\frac{\eta_t}{2} \norm{g(\vx_t)}^2 +\delta D\\
     &\quad \overset{(a)}{\leq} \frac{\eta_t}{2} \norm{g(\vx_t)}^2 +\delta D \lesssim  \frac{\eta_t}{2} G^2 +\delta D \,,
 \end{align*}
 where ($a$) follows from strong convexity.
 Therefore, it holds that
 \begin{align*}
     &f\left(\sum_{t=1}^T \frac{2(t+1)}{(T+1)(T+2)}\vx_t \right) -f(\vx^*) \leq \sum_{t=1}^T \frac{2(t+1)}{(T+1)(T+2)} (f(\vx_t) -f(\vx^*))\\
     &\quad \lesssim \sum_{t=1}^T \frac{2(t+1)}{(T+1)(T+2)} (\frac{\eta_t}{2}G^2 +\delta D)  \lesssim \frac{G^2}{\mu T } + \delta D. 
\end{align*}
Let  $\bar{\vx}_T:=\sum_{t=1}^T \frac{2t}{T(T+1)}\vx_t$.
We next bound the deviation. 
Again, let $\{\vy_t\}$ be the iterates of (projected) GD with stepsize $\eta_t$ with exact gradients and $\{\vx _t\}$ be the iterates of (projected) GD with the same stepsize with inexact gradients with noise $\{\no_t\}$.
\begin{align*}
    \norm{\vy_{t+1}-\vx_{t+1}}^2 &\overset{(a)}{\leq} \norm{\vy_{t}-\eta_t \nabla f(\vy_t)-(\vx_{t} -\eta_t (\nabla f(\vx_t)+\no_t))}^2\\
    &=\norm{\vx_t-\vy_t}^2 - 2\eta_t \inp{\vy_t-\vx_t}{\nabla f(\vy_t)-\nabla f (\vx_t)-\no_t } \\
    &\qquad +\eta_t^2 \norm{\nabla f(\vy_t) -
    \nabla f(\vx_t)-\no_t}^2\\
    &\overset{(b)}{\leq} (1-2\mu \eta_t)\norm{\vx_t-\vy_t}^2 + 2\eta_t \inp{\vx_t-\vy_t}{\no_t } +\eta_t^2 \norm{\nabla f(\vy_t) -\nabla f(\vx_t)-\no_t}^2\\
    &\overset{(c)}{\lesssim} (1-\frac{2}{t+1} )\norm{\vx_t-\vy_t}^2 + 2\delta \eta_t\norm{\vx_t-\vy_t}  +\eta_t^2 G^2\,,
\end{align*}
where $(a)$ is due to the non-expansiveness of the projection step, $(b)$ is due to convexity, and $(c)$ is due to the inequality $\norm{v_1+v_2+v_3}\leq 3\norm{v_1}^2 +3\norm{v_2}^2 +3\norm{v_3}^2$.
Denoting $d_t := \norm{\vx_t-\vy_t}$, we obtain
\begin{align} \label{ineq:dev_str}
    d_{T}^2 \lesssim  \delta \sum_{t=0}^{T-1} \frac{t^2}{T^2} \eta_t   d_t +  \sum_{t=0}^{T-1} \frac{t^2}{T^2} \eta_t^2.
\end{align}
Now similarly to \autoref{lem:dev_det}, one can deduce from this inequality that
\begin{align*}
     d_{T}^2 \lesssim  \frac{\delta}{\mu} d_T +  \frac{1}{T\mu^2} \quad \Longrightarrow\quad
    d_{T}^2  \lesssim  \frac{\delta^2}{\mu^2}  +  \frac{1}{T\mu^2}\,.
\end{align*}
Now, after applying the Jensen's inequality, we obtain the desired deviation bound of
$\norm{\bar{\vx}_T-\bar{\vx}_T'}^2 \leq (\frac{\delta^2}{\mu^2}  +  \frac{1}{T\mu^2}) \wedge \frac{\eps}{\mu}$, 
where $\frac{\eps}{\mu}$ follows from the fact that $\bar{\vx}_T$ achieves $\eps$-accuracy.
\end{proof}

\subsection{Inexact initialization model (strongly convex costs)}

  \begin{restatable}{theorem}{ubinitnonsmoothstr}{\bf (Upper Bound)}
 \label{thm:ub_init_nonsmooth_str}
Let $f$ be an $O(1)$-Lipschitz $\mu$-strongly convex cost function. Let $\eps>0$ be a small constant, and $T$ be a given number of iterations.
Then there exists a \sfa{} algorithm whose $(\eps,\delta)$-deviation  is $O(\frac{1}{T\mu^2} \wedge \frac{\eps}{\mu})$, provided that $T=\Omega(1/\eps)$.
\end{restatable}
\begin{proof}
The standard convergence rate bound (e.g., \cite[Theorem 2.4]{bansal2019potential}) implies that GD with $\eta_t = \frac{2}{\mu(t+1)}$ satisfies
\begin{align}
     f\left(\sum_{t=1}^T \frac{2t}{T(T+1)}\vx_t \right) -f(\vx^*) \lesssim \frac{2G^2 }{\mu (T+1)}\,, 
\end{align}
where $G$ is the Lipschitz constant of $f$.
Hence, letting $\bar{\vx}_T:=\sum_{t=1}^T \frac{2t}{T(T+1)}\vx_t$, it follows that 
\begin{align*}
    \norm{\bar{\vx}_T -\vx^*}^2 \lesssim \frac{1}{T\mu^2} \wedge \frac{\eps}{\mu}\,, 
\end{align*}
where $\frac{\eps}{\mu}$ follows from the fact that $\bar{\vx}_T$ achieves $\eps$-accuracy.\end{proof}

\section{Proof of upper bound for finite-sum setting (\autoref{thm:ub_finite_nonsmooth})}
\label{sec:ub_finite}

Recall \autoref{thm:ub_finite_nonsmooth}:
\ubfinitenonsmooth* 
\begin{proof} 
We first prove a deviation bound similar to that of \autoref{lem:dev_det}.
Let us denote $\no_t := g_{i_t}(\vx_t)-\nabla f_{i_t} (\vx_t)$.
 Let $\{\vy_t\}$ be the iterates of (projected) GD with inexact gradients.  Then we have
\begin{align*}
    &\E_{i_t}\norm{\vx_{t+1}-\vy_{t+1}}^2 \\
    &\quad \overset{(a)}{\leq} \E_{i_t}\norm{\vx_{t}-\eta_t g_{i_t} (\vx_t)-(y_{t} -\eta_t (\nabla f(y_t)))}^2\\
    &\quad= \norm{\vx_t-y_t}^2 - 2\eta_t \E_{i_t}\inp{\vx_t-y_t}{g_{i_t}(\vx_t)-\nabla f (y_t)} +\eta_t^2 \E_{i_t}\norm{g_{i_t}(\vx_t)-\nabla f (y_t)}^2\\
 &\quad= \norm{\vx_t-y_t}^2 - 2\eta_t \E_{i_t}\inp{\vx_t-y_t}{\nabla f_{i_t}(\vx_t) +\no_t-\nabla f (y_t)} +\eta_t^2 \E_{i_t}\norm{g_{i_t}(\vx_t)-\nabla f (y_t)}^2\\
  &\quad= \norm{\vx_t-y_t}^2 - 2\eta_t \inp{\vx_t-y_t}{\nabla f (\vx_t) +\no_t-\nabla f (y_t)} +\eta_t^2 \E_{i_t}\norm{g_{i_t}(\vx_t)-\nabla f (y_t)}^2\\
    &\quad\overset{(b)}{\leq} \norm{\vx_t-y_t}^2 - 2\eta_t \inp{\vx_t-y_t}{\no_t } +\eta_t^2 \E_{i_t}\norm{\nabla f_{i_t}(\vx_t) + \no_t-\nabla f (y_t)}^2\\
    &\quad\overset{(c)}{\leq} \norm{\vx_t-y_t}^2 +2\delta \eta_t \norm{\vx_t-y_t}  +3\eta_t^2 (2G^2 +\delta^2)\,,
\end{align*}
where $(a)$ is due to the non-expansiveness of the projection step, and $(b)$ is due to convexity, and $(c)$ is due to the inequality $\norm{v_1+v_2+v_3}\leq 3\norm{v_1}^2 +3\norm{v_2}^2 +3\norm{v_3}^2$.
Taking expectations on both sides, we obtain
\begin{align*}
    \E\norm{\vx_{t+1}-y_{t+1}}^2  &\leq \E\norm{\vx_t-y_t}^2 +2\delta \eta_t  \E\norm{\vx_t-y_t}  +3\eta_t^2 (2G^2 +\delta^2)\\
    &\leq \E\norm{\vx_t-y_t}^2 +2\delta \eta_t   \sqrt{\E\norm{\vx_t-y_t}^2}  +3\eta_t^2 (2G^2 +\delta^2)\,,
\end{align*}

Denoting $d_t := \sqrt{\E\norm{\vx_t-y_t}^2}$ and telescoping the above inequality, we obtain
\begin{align} \label{ineq:dev_sq:2}
    d_{T}^2 \leq 2 \delta \sum_{t=0}^{T-1} \eta_t   d_t + 3(2G^2+\delta^2)  \sum_{t=0}^{T-1} \eta_t^2.
\end{align} 
This is precisely equal to \eqref{ineq:dev_sq} from the proof of \autoref{lem:dev_det}.
Following the same recursion, we obtain the following bound:
\begin{align*} 
   d_T \leq \sqrt{3(2G^2+\delta^2) \cdot \sum_{t=0}^{T-1} \eta_t^2 } +2 \delta  \sum_{t=0}^{T-1} \eta_t \,.
\end{align*}
Squaring both sides, we obtain
\begin{align}
    \E\norm{\vx_T-y_T}^2 \leq  6(2G^2+\delta^2) \cdot \sum_{t=0}^{T-1} \eta_t^2   +8\delta^2  \left( \sum_{t=0}^{T-1} \eta_t \right)^2\,.
\end{align}
We next prove the bound on the convergence rate.

 {\bf Convergence rate bound.}
 We consider the projected gradient descent with constant stepsize $\eta$ onto the ball of radius $D$ that contains the optimum $\vx_*$.
 Let $z_t$ denote the iterate before projection. 
 As before, let $\no_t := g_{i_t}(\vx_t)-\nabla f_{i_t} (\vx_t)$.
 Then, we have
 \begin{align*}  &\frac{1}{2}\E_{i_t}\norm{\vx_{t+1}-\vx_*}^2 -\frac{1}{2}\norm{\vx_t-\vx_*}^2  \leq  \frac{1}{2}\E_{i_t}\norm{z_{t+1}-\vx_*}^2 -\frac{1}{2}\norm{\vx_t-\vx_*}^2 \\
 &\quad = \frac{1}{2}\E_{i_t}\norm{\vx_{t}- \eta_t g_{i_t}(\vx_t)-\vx_*}^2 -\frac{1}{2}\norm{\vx_t-\vx_*}^2 \\
 &\quad=  -2\eta_t\E_{i_t}\inp{\nabla f_{i_t}(\vx_t)+\no_t}{\vx_t-\vx_*} +\frac{1}{2}\eta_t^2\E_{i_t}\norm{\nabla f_{i_t}(\vx_t)+\no_t}^2\\
 &\quad=  -2\eta_t \inp{\nabla f (\vx_t)+\no_t}{\vx_t-\vx_*} +\frac{1}{2}\eta_t^2 \E_{i_t}\norm{\nabla f_{i_t}(\vx_t)+\no_t}^2\\
 &\quad\leq   -2\eta_t \inp{\nabla f (\vx_t)+\no_t}{\vx_t-\vx_*} + \eta_t^2 (G^2+\delta^2)\,.
 \end{align*}
 Hence,
  \begin{align*} 
  &   f(\vx_t) -f(\vx_*) +\frac{1}{2\eta_t}\E_{i_t}\norm{\vx_{t+1}-\vx_*}^2 -\frac{1}{2\eta_t } \norm{\vx_t-\vx_*}^2  \\
 &\quad\leq    f(\vx_t) -f(\vx_*) -\inp{\nabla f (\vx_t)+\no_t}{\vx_t-\vx_*} +\eta_t (G^2+\delta^2)\\
 &\quad\leq \delta D +
 \eta_t (G^2+\delta^2). 
 \end{align*}
Choosing $\eta_t\equiv \eta$ and after telescoping the above inequalities from $t=0,\dots, {T-1}$, we obtain the bound
 \begin{align*}
     \E f\left(\frac{1}{T}\sum_{t=0}^{T-1}\vx_t\right)-f(\vx_*) \lesssim \frac{D^2}{\eta T } +  \eta G^2 +\delta D.
 \end{align*}
 Thus, for $\eps$-accuracy, we need $\eta\lesssim \eps$ and $\eta T \gtrsim 1/\eps$, since the theorem statement assumed that $\delta \leq \frac{\eps}{2D}$. Hence, choosing $\eta = \Theta(\frac{1}{\eps T})$, the deviation bound we proved gives
\begin{align*}
     \E \norm{\vx_t-\vx_t'}^2 &\lesssim \frac{t}{\eps^2T^2} +(t\delta  \eta)^2  = \frac{t}{\eps^2T^2} + \frac{\delta^2t^2}{\eps^2T^2}.
 \end{align*}
Hence,
\begin{align*}
    \E \norm{\bar{\vx}_T -\bar{\vx}'_T}^2 &\leq \frac{1}{T} \sum_t\E \norm{{x}_t  -{x}'_t}^2 \\
    &\lesssim \frac{1}{T} \sum_t\left[ \frac{t}{\eps^2T^2} + \frac{\delta^2t^2}{\eps^2T^2} \right]\approx\frac{1}{\eps^2T} + \frac{\delta^2}{\eps^2}\,,
\end{align*}
as desired.
\end{proof}

\end{document}